\documentclass[
	fontsize=10pt, 
	twoside=false, 
	a4paper,
]{scrartcl}
 
\usepackage[bookmarkdepth=subsection,font=test]{templ}

\usepackage{tikz}
\usetikzlibrary{patterns}

\usepackage{bm}
\usepackage{mathrsfs}

\usepackage{tikz}
\usetikzlibrary{arrows}
\usetikzlibrary{positioning}
\usetikzlibrary{shapes.geometric}
\usepackage{pgf}
\usetikzlibrary{patterns}
\usetikzlibrary{patterns.meta}
\usepgflibrary{plothandlers} %
\usepgflibrary[plothandlers] %
\usetikzlibrary{plothandlers} %
\usetikzlibrary[plothandlers] %
\usetikzlibrary{topaths} %
\usetikzlibrary[topaths] %
\usetikzlibrary{decorations.markings}

\makeatletter
\newcommand{\leqnomode}{\tagsleft@true\let\veqno\@@leqno}
\newcommand{\reqnomode}{\tagsleft@false\let\veqno\@@eqno}
\makeatother
\newcommand{\eps}{\varepsilon}
\newcommand{\sem}{\setminus\{0\}}
\newcommand{\RR}{{\mathbb{R}}}

\newcommand{\CC}{{\mathbb{C}}}

\newcommand{\dd}{\mathrm{d}}

\newcommand{\norm}[1]{\left\lVert#1\right\rVert}
\newcommand{\abs}[1]{\left\lvert#1\right\rvert}
\newcommand{\del}{\partial}
\newcommand{\supp}{\mathrm{supp}}

\newcommand{\td}[1]{\tilde{#1}}

\newcommand{\id}{\mathrm{id}}
\newcommand{\wf}{\mathrm{WF}_{\mathrm{a}}}

\newcommand{\ud}[1]{\underline{#1}}

\newcommand{\bdry}{(h\td{\del}_\nu+(\td{\del}_\nu\phi))}

\usepackage{xparse}

\NewDocumentCommand{\op}{O{h}m}{\mathrm{Op}_{#1}\left(#2\right)}
\NewDocumentCommand{\tbdry}{O{}}{\tau_{#1}\bdry}

\newcommand{\ccdot}{\,\cdot\,}

\begin{document} 

\title{Density of Products of Solutions and the Linearized Calderón Problem in an Analytic Background}

\author{Leonard Busch\and Tony Liimatainen \and Gunther Uhlmann}

\date{\today}

\maketitle

\qquad

\begin{abstract}
We study the linearization, with respect to zeroth-order perturbations, of the local Dirichlet-to-Neumann map of a formally self-adjoint elliptic operator $P=-\Delta_g+Q$ on a compact Riemannian manifold $(M,g)$ of dimension $n\geq 2$ with boundary, where $Q$ is of order at most one. Near a boundary point where the boundary, metric, and $P$ are analytic, we show that the linearized map is a classical analytic pseudodifferential operator of order $-1$ if and only if the zeroth-order perturbation $q$ extends analytically across the boundary. We also prove that local boundary data in an arbitrary open analytic boundary patch determines $q$ locally. In particular, if the measurement set meets every connected component of $M$, the local linearized DN map is injective on $L^\infty$ perturbations. Finally, we obtain a global completeness result for products of pairs of solutions whose boundary traces are supported in the measurement set. The global recovery proceeds by local steps using Runge approximation.
\end{abstract}

\section{Introduction}

In \cite{MR590275}, Calder\'on posed the problem of whether one can recover the electrical conductivity of the interior of some Euclidean medium from measurements of voltage and current at the boundary. In fact, in the same paper, Calder\'on proved the injectivity of the linearization of the measurement operator at constant conductivities. The nonlinear problem, now called the Calder\'on problem, has since been studied extensively and we refer to the survey articles \cite{zbMATH06457660}, \cite{zbMATH05655833}. For a sufficiently regular scalar conductivity on a bounded Euclidean domain, the conductivity equation underlying the Calder\'on problem can be reduced to a Schr\"odinger equation, and the inverse problem then corresponds to determining a scalar potential from its Dirichlet-to-Neumann map \cite{zbMATH04015323}. 

Let $(M,g)$ be a smooth compact Riemannian manifold of dimension $n\geq 2$ with smooth boundary. Let $Q$ be a first-order differential operator with smooth coefficients on $M$ and define
\begin{equation}\label{eq:def_of_P}
    P \coloneqq -\Delta_g + Q\,,
\end{equation}
where $\Delta_g = \dd^\ast \dd$. We assume that $L=-P$ is strongly elliptic, formally self-adjoint and does not have Dirichlet eigenvalue $0$ (see \cref{sec:EVrunge}). This assumption implies that the Green's and Poisson operators $G = G_P = -G_L$ and $K=K_P=K_L$ are well-defined for $P$ on $M$ (\cref{prop:GK}). The Dirichlet-to-Neumann map for $P$ is given by 
\[
    \Lambda\colon H^{1/2}(\del M) \ni f \mapsto \Lambda f \coloneqq \del_\nu Kf\vert_{\del M} \in H^{-1/2}(\del M)\,.
\]

When, in the Calder\'on problem, the unknown conductivity is allowed to be anisotropic, i.e., depends on direction, it may be represented geometrically by a Riemannian metric, motivating the setting of a general manifold. The \emph{anisotropic} Calder\'on problem asks to determine an unknown metric from the Dirichlet-to-Neumann map for the conductivity equation, or by a reformulation valid in dimension $n\geq 3$ (see \cite{zbMATH04150775}), for the Schr\"odinger equation. In our notation this corresponds to recovering $g$ from the knowledge of $\Lambda$, when $Q=0$ in the definition of $P$. In dimensions $n\geq 3$, the Dirichlet-to-Neumann map is invariant under pullbacks by boundary-fixing diffeomorphisms, so that recovery of the metric is possible only modulo this gauge. It has recently been announced that the anisotropic Calder\'on problem has been resolved in the smooth setting \cite{oa1,oa2}, we comment on this later. The real-analytic case, however, is more classical: in \cite{zbMATH04150775}, Lee and Uhlmann proved a result under additional geometric and topological assumptions, and Lassas and Uhlmann proved recovery of a real-analytic metric from Cauchy data on an arbitrary non-empty open analytic boundary patch on connected real-analytic manifolds \cite{zbMATH01736651}. See also \cite{zbMATH07203082}, which developed the Poisson embedding method, reformulating the Calder\'on problem in terms of solutions rather than coefficients.

When the unknown metric is restricted to a conformal class, the anisotropic Calder\'on problem is reduced to recovering a scalar potential on a known manifold from the Dirichlet-to-Neumann map: replacing $Q$ in $P$ by a scalar potential $V$, can one determine $V$ from the knowledge of $\Lambda$? In the Euclidean setting, this problem was solved in \cite{zbMATH04015323,zbMATH05254953} for $n\geq 3$ and $n=2$ respectively. It has also been solved on conformally transversally anisotropic (CTA) manifolds under conditions on the transversal geometry, initially under simplicity or injectivity assumptions for the geodesic ray transform \cite{zbMATH05602517,zbMATH06656250}. 

In practical situations, it is more realistic to assume that one only has access to the map $\Lambda$ on some open subset $\Sigma\subset \del M$. Define 
\[
    H^{s}_\Sigma \coloneqq \{f\in H^{s}(\del M), \supp f \subset \Sigma\}\,, \quad s\in \RR\,,\qquad \Lambda^\Sigma \colon H^{1/2}_\Sigma \ni f\mapsto \Lambda f\vert_\Sigma \in H^{-1/2}(\Sigma)\,.
\]
The local conformal Calder\'on problem asks whether, replacing $Q$ in $P$ by a scalar potential $V$, one can recover $V$ from $\Lambda^\Sigma$. Via an application of the Green’s identity, this problem is resolved if one answers the following question positively: for two real scalar potentials $V_1, V_2$ on $M$, does
\begin{equation}\label{eq:localcald}
    \int_M (V_1-V_2)u_1u_2 \dd V_g(y) = 0
\end{equation}
for all $u_j$ solving $(-\Delta_g+V_j)u_j=0$ with $\supp u_j\vert_{\del M}\subset \Sigma, j\in\{1,2\}$ imply that $V_1 = V_2$? 

One may distinguish the sets upon which one prescribes data ($f \in H^{1/2}_{\Sigma_1}$) and on which one measures data ($\Lambda f\vert_{\Sigma_2}$) where one allows $\Sigma_1 \neq \Sigma_2$. In this case one speaks of the \emph{partial data} Calder\'on problem, whereas in our setting when $\Sigma_1=\Sigma_2 = \Sigma$, i.e., the input and output sets coincide we use the term \emph{local}  Calder\'on problem. In dimension $n=2$, the local (conformal) problem is resolved for Euclidean domains (\cite{zbMATH05775680}) and for Riemann surfaces (\cite{zbMATH05904312}). In dimensions $n\geq 3$, the local problem remains open in general, even in the Euclidean setting. In dimension $n=3$, \cite{2606.10247v1} proves local recovery for strictly convex Euclidean domains under additional smoothness and support assumptions. Under suitable geometric assumptions, there are results in the partial data case, beginning with \cite{recoveringCauchydata} and \cite{zbMATH05180743}, see also a systematic account in \cite{recentProg}. 

The Dirichlet-to-Neumann map depends nonlinearly on the scalar potential $V$, so that one may also study its linearization at a fixed background. Using the same argument leading to \cref{eq:localcald}, injectivity of the linearized Dirichlet-to-Neumann map corresponds to completeness, or equivalently density in $L^1(M)$ of products of solutions, of the same background equation. As mentioned previously, this result was proved for the Euclidean Laplacian with full boundary data in \cite{MR590275}. Subsequently, in \cite{zbMATH05707695}, the local-data analogue was proved: products of harmonic functions whose boundary values are supported in an arbitrary non-empty open boundary subset are dense in $L^1$ on a bounded Euclidean domain. Our first main result removes the restriction to the Laplacian on a Euclidean domain.
\begin{theorem}\label{thm:harmonic}
    Assume that $(M^{\mathrm{int}},g)$ is analytic and $P$ has analytic coefficients in $M$. Let $q\in L^\infty(M)$ and assume there is some open $\Sigma \subset \del M$ that has non-empty intersection with each connected component of $M$ and assume $\del M$ is analytic on $
    \Sigma$. If
    \[
        0 = \int_M q u_1 \overline{u_2} d V_g
    \]
    for all solutions $u_j \in C^\infty(M)$ to $P u_j = 0$ with $\supp u_j\vert_{\del M}\subset \Sigma$, $j \in \{1,2\}$, then $q=0$ a.e.
\end{theorem}
This result extends \cite{zbMATH05707695} to a general geometric setting and replaces the Euclidean Laplacian at zero background $-\Delta$ with the more general $P$, while imposing additional analyticity of a boundary patch. 

There are also full-boundary-data results known on manifolds with additional global structure. On K\"ahler manifolds having sufficiently many holomorphic functions, completeness was first proved in \cite{zbMATH07076582}, and later generalized in \cite{zbMATH07940027}. On transversally anisotropic manifolds, \cite{zbMATH07323401} showed the recovery of transversal wavefront set under geometric assumptions on pairs of geodesics, while \cite{zbMATH07534686} established a completeness result when the transversal manifold is analytic and every transversal unit covector is generated by a suitable pair of geodesics.

In contrast with these results, \cref{thm:harmonic} relies only on the analyticity of the background geometry and operator. Its proof relies on Runge approximation for a global to local reduction and a local recovery result that is essentially already contained in \cite{zbMATH06600504}. We give the full proof nevertheless. We point out that as in \cite{zbMATH07305900}, \cref{thm:harmonic} has immediate implications for inverse problems for quasi-linear differential operators.

Finally, we comment on the announced resolutions of anisotropic Calderón problems \cite{oa2,oa1}, which appeared shortly before this preprint. Our method of proof is a density argument, whereas theirs relies on unique continuation rather than on the Alessandrini identity. Consequently, it is not clear that the injectivity they announce for the nonlinear map implies injectivity of the linearization. That approach therefore does not appear to yield density of products of pairs of solutions. Traditionally, these two problems have been closely related. A solution method for the nonlinear problem via the Alessandrini identity has also yielded a solution to the linearized problem, and hence density. Our approach, in contrast, directly establishes the required density. Moreover, unlike ours, the announced results do not consider general zeroth-order terms in the equations.

Let us next describe the linearized Dirichlet-to-Neumann map in detail. For arbitrary $q\in L^\infty(M)$, consider the family of operators $P_t \coloneqq P + tq$, and let $\Lambda_t$ be the corresponding Dirichlet-to-Neumann maps for $t \in \RR$. For sufficiently small $t$, a Neumann series argument implies that each $\Lambda_t$ is well-defined. The linearization $\dot \Lambda_q$ at $t=0$ of $\Lambda_t$ is given by 
\[
    \dot\Lambda_q = -K^\ast q K\,,
\]
see \cite{zbMATH06600504}, where $K^\ast$ is the (Hilbert-space) adjoint of $K$. Observe that the linearized map is linear in the perturbation $q$. 

There are classical boundary determination results that show that the full boundary jet of a smooth scalar conductivity can be recovered from local (nonlinear) boundary measurements (\cite{zbMATH03939884}), which thus determines a real-analytic conductivity. This boundary determination was extended to the anisotropic setting in \cite{zbMATH04150775}. We shall show, in some sense, that the linearized map $\dot\Lambda_q$ is defined via the full boundary jet of the perturbation $q$ when it exists.
\begin{theorem}\label{thm:symb}
    Let $\hat y\in \partial M$ near which $\del M, g$ and (the coefficients of) $P$ are analytic. For any $q \in L^\infty(M)$, $\dot \Lambda_q$ is a classical analytic pseudodifferential operator of order $-1$ near $\hat y$ if and only if $q$ extends analytically across the boundary $\del M$ near $\hat y$. In this situation it is elliptic if and only if $\lim_{M^{\mathrm{int}}\ni y\to \hat y}q(y) \neq 0$.
\end{theorem}
\cref{prop:sigmatotq} expresses the full formal symbol of $\dot\Lambda_q$ in terms of the boundary jet of $q$. In particular, by order $-1$, we mean order at most $-1$, depending on the vanishing of the boundary jet of $q$. The nomenclature `classical analytic pseudodifferential operator' is justified after the statement of \cref{thm:localsymb}, which is a stronger statement than \cref{thm:symb}: the symbol of $\dot\Lambda_q$ need only be a classical analytic symbol along a one-dimensional fiber direction, rather than in the full cotangent bundle to conclude that $q$ must be analytic up to the boundary. This restriction is natural because it makes the problem formally determined: the unknown $q$ depends on $n$ variables, while the symbol of a pseudodifferential operator on the boundary depends on $2(n-1)$ variables; restricting the fiber variable to one dimension leads to $n-1+1$ variables, compare with \cite{zbMATH06600504}. 

\cref{thm:localsymb}, and thus its corollary \cref{thm:symb} are generalizations of the Euclidean \cite[Thm.~1.6,~Prop.~1.7]{zbMATH06600504} to a general geometric setting and general background operator $P$. Our method of proof partially follows theirs, but differs when we identify the formal symbol of the pseudodifferential operator from the boundary jet of the perturbation in \cref{prop:sigmatotq}. We point out that \cite{zbMATH06600504} has been superficially corrected after publication in the revised \cite{1312.4065v2}. 

For an open $\Sigma \subset \del M$ recall for $s\in \RR$ the notation $H^{s}_\Sigma \coloneqq \{f\in H^{s}(\del M), \supp f \subset \Sigma\}$ and for arbitrary $q\in L^\infty(M)$ define
\[   
    \dot\Lambda_q^{\Sigma} \colon H^{1/2}_\Sigma  \ni f \mapsto \dot\Lambda_q(f)\vert_{\Sigma} \in H^{-1/2}(\Sigma)\,.
\]
As a corollary of \cref{thm:harmonic}, we have the global injectivity statement
\begin{theorem}
    Assume the hypotheses of \cref{thm:harmonic}, and let $q\in L^\infty(M)$ be arbitrary. If $\dot\Lambda^\Sigma_q = 0$, then $q=0$ a.e.; in other words, the map $q \mapsto \dot\Lambda^\Sigma_q$ is injective.
\end{theorem}
\begin{proof}
    For arbitrary $f_j \in C_c^\infty(\Sigma)$ set $u_j = Kf_j, j\in \{1,2\}$. Because $\dot\Lambda_q = -K^\ast q K$, we have
    \[
        \langle \dot\Lambda^\Sigma_q f_1, f_2 \rangle = -\int_M q u_1 \overline{u_2}\dd V_g\,.
    \]
    In particular, if $\dot\Lambda_q^\Sigma =0$, then \cref{thm:harmonic} implies that $q=0$ a.e., completing the proof. 
\end{proof}
In the Euclidean Schr\"odinger setting, the corresponding result is \cite[Cor.~1.8]{zbMATH06600504}. 

Finally, we prove local injectivity of the linearized map at arbitrary $L^\infty$ perturbations. 
\begin{theorem}\label{thm:lin}
	Let $\hat y\in \partial M$ near which $\del M, g$ and (the coefficients of) $P$ are analytic. For any open neighborhood $\Sigma\subset \del M$ of $\hat y$, and any $q_1, q_2 \in L^\infty(M)$, if we have 
	\[
		\dot\Lambda_{q_1}^\Sigma = \dot\Lambda_{q_2}^\Sigma\,,
	\]
	then $q_1 = q_2$ a.e. near $\hat y$. 
    
    In particular, if every connected component of $M$ contains a boundary point $\in \del M$ near which $\del M, g$ and $P$ are analytic, and $\Sigma\subset \del M$ is an open neighborhood of these (finitely many) boundary points, and if $q\in C^\omega(M^{\mathrm{int}})\cap L^\infty(M)$ is a potential so that $\dot\Lambda^\Sigma_q = 0$, then $q\equiv 0$ on $M$.
\end{theorem}
In total, our results demonstrate that the Euclidean local completeness and analytic boundary determination results of \cite{zbMATH05707695} and \cite{zbMATH06600504} remain true on general analytic geometric backgrounds and for general analytic background differential operators.

\subsubsection*{Sketch of the proof of \cref{thm:harmonic}}

By applying a Runge approximation-type argument (\cref{sec:EVrunge}), the global result of \cref{thm:harmonic} is reduced to proving local vanishing of the unknown potential $q$ near a point $\hat y\in\del M$ near which $\del M$ is analytic (\cref{thm:localmain}). Relying on microlocal analytic continuation (\cite[Thm.~8.5.6']{hoermander1}), vanishing of $q$ near $\hat y$ is reduced to showing $(\hat y, \hat\xi)\not\in \mathrm{WF}_a(q)$, where $\hat\xi$ is the normal of $\del M$ at $\hat y$, and this in turn corresponds to showing that some (germ) FBI transform of $q$ is exponentially decaying (\cite{AST_1982__95__R3_0}, \cite[Thm.~22.3.15]{MR4436039}). In particular, we have reduced the proof of \cref{thm:harmonic} to showing that there are solutions $u_j$ of $P u_j = 0$ so that $u_1 \overline{u_2}$ is the kernel of some (germ) FBI transform (with $\supp u_j\vert_{\del M}\subset \Sigma, j\in\{1,2\}$ where $\Sigma\subset\del M$ is some open set). Constructing $u_1, u_2$ via analytic WKB solutions (\cref{prop:K1,prop:K2}) then achieves this goal. 

We point out explicitly that if we had desired to prove only \cref{thm:harmonic}, the WKB solution with FBI-type phase in \cref{prop:K1} would have sufficed, and we could have foregone \cref{prop:K2}. After appropriate modification of \cref{sec:localinj}, \cref{thm:localmain} would be shown to hold with only WKB solutions of the type coming from \cref{prop:K1}, which would then prove \cref{thm:harmonic}. In particular, the shortest route to prove \cref{thm:harmonic} would rely only on \cref{prop:K1} and suitably modified \cref{cor:Ks,thm:localmain,sec:localinj}. The entirety of \cref{sec:findq} and many results in \cref{sec:localsol} thus exist only to prove \cref{thm:symb}. \cref{thm:localmain} (or rather the most significant part of it, \cref{cor:vanish}) thus serves the dual purpose of being the main local recovery result for both \cref{thm:harmonic} and \cref{thm:symb} simultaneously.

\subsubsection*{Outline of the paper}

\cref{sec:localsol} is dedicated to constructing and analyzing $Kf$ when $f$ is a (Gaussian) wave packet or a (localized) plane wave. We also construct and study the `analytic symbol' of the operator $\dot\Lambda_q$. The main local results are stated in \cref{thm:localmain,thm:localsymb}. In \cref{sec:localinj} we prove local recovery of a potential $q$ from local boundary measurements using the special solutions constructed previously, after which we give the proofs of \cref{thm:harmonic,thm:lin}, the proof of the former relying on a Runge approximation argument. In \cref{sec:findq}, we analyze the symbol of $\dot\Lambda_q$ and reduce \cref{thm:symb} to the local recovery result proved in \cref{sec:localinj} and give its proof. The appendices~\ref{sec:cla} and~\ref{sec:comb} recall standard facts about classical analytic symbols and some combinatorial estimates, while in \cref{sec:EVrunge} we prove a general Runge approximation argument and related results.

\subsubsection*{Statement of AI use}
AI tools were used to improve the readability of the manuscript. The mathematical results and proofs were developed and verified by the authors without the use of AI.

\subsubsection*{Acknowledgments}
T. L. is partly supported by the Research Council of Finland through the Centre of Excellence in Inverse Modelling and Imaging and the FAME Flagship, grants 353091 and 359208.

\subsection{Notation and definitions}

The proof of the main theorems involve Fourier-integral-operator (FIO)-type operators, which, in the analytic setting, only admit a well-established calculus on weighted spaces of holomorphic functions and germs of those spaces (\cite{AST_1982__95__R3_0}). Results on the composition or inversion of these operators exist only in this setting and we will follow Tr\`eves' exposition of it (\cite[Part~VI]{MR4436039}) in this paper. In particular, our definitions and cited results are taken to match those given by Tr\`eves, while these results were originally established by Sj\"ostrand, see \cite{AST_1982__95__R3_0} and references therein. 

The convenient notion of localization in the setting of analytic/holomorphic functions is that of germs: For an open set $\Omega \subset \CC^n$ we write $\mathcal{O}(\Omega)$ for the set of holomorphic functions on $\Omega$. When $z \in \Omega$, then define $\bl{O}_z$ as the germ of holomorphic functions at $z$: 
\[
	\bl{O}_z \coloneqq \left\{f \colon \exists\ \text{open neighborhood}\ \Omega' \text{of}\ z\colon f\in \mathcal{O}(\Omega')\right\}/ \sim\,.
\]
Here, $\sim$ denotes the equivalence relation of functions that are holomorphic near $z$ and agree near $z$.

FIO-type operators in the analytic setting will have phases with non-trivial imaginary part and the handling thereof is taken care of by spaces of holomorphic functions defined by an exponent weight following \cite[Def.~19.1.10,~Def.~20.2.7]{MR4436039}, see also \cite{AST_1982__95__R3_0}, \cite{zbMATH07060662}. 
\begin{definition}
	For any $\Phi \in C(\Omega;\RR)$, let $\mathcal{O}^{(\Phi)}(\Omega)$ be the set of holomorphic functions $f(z;h)$ on $\Omega$ so that for all $\eps>0$ and $K\subset \Omega$ compact there is $C>0$ so that for all $0<h\leq 1$
	\[
		\max_K \abs{e^{-h^{-1}\Phi}f} \leq C e^{h^{-1}\eps}\,.
	\]
	We define the space $\mathcal{O}^{(-\omega)}(\Omega)$ as the set of holomorphic functions $f(z;h)$ on $\Omega$ so that for any compact $K\subset \Omega$ there are $\eps,C>0$ so that for all $0<h\leq 1$,
	\[
		\max_K \abs{f} \leq C e^{-h^{-1}\eps}\,.
	\]
	Finally, for any $z\in \Omega$ we let $\bl{O}^{(\Phi)}_z$ and $\bl{O}^{(-\omega)}_z$ be the germs of these spaces at $z$. 
\end{definition}
For any $f\in \mathcal{O}(\Omega)$ we write $f^\ast(z) \coloneqq \bar f(\bar z)$ for the holomorphic conjugate of $f$ where it is defined. Throughout the rest of this document we use the following conventions: when $z\in \CC^n$, then $z^{\RR}$ and $z^{i\RR}$ denote the real and imaginary parts of $z$ respectively, and if $f$ is a complex-valued function (with domain in $\CC^n$), $\del_z f = \frac{1}{2}\left(\del_{z^\RR}-i\del_{z^{i\RR}}\right)f$ always denotes the Wirtinger derivative. Note that when $f$ is holomorphic, we have $\bar\del_z f = \frac{1}{2}\left(\del_{z^\RR}+i\del_{z^{i\RR}}\right)f=0$ so that $\del_z f = \del_z f  +\bar \del_z f = \del_{z^\RR}f$, which is the usual partial derivative in $z^{\RR}$. In particular, $\del_z f^{i\RR} = \frac{1}{2i} \del_z f = \frac{1}{2i} \del_{z^{\RR}} f$ and $\del_z^2 f^{i\RR} = \frac{1}{2i} \del_z^2 f = \frac{1}{2i}\del_{z^\RR}^2 f$ for holomorphic $f$. Note that the chain rule for Wirtinger derivatives is (componentwise) $\del_z (f\circ g) = ((\del_z f)\circ g)\del_z g+((\bar \del_z f)\circ g)\del_z \bar g$, which reduces to $((\del_z f)\circ g)\del_z g$ if $f$ is holomorphic.

We shall use the FBI transform on germ spaces (\cite{AST_1982__95__R3_0}, \cite{zbMATH07060662}, \cite{MR4436039}) to test for the presence of analytic wavefront set. This requires that we introduce a few more notions. 

\begin{definition}[{\cite[\S~20.2]{MR4436039}}]\label{def:well-shaped}
    Let $\Omega\subset \CC^n$ be a domain. A \emph{contour of integration} in $\Omega$ is a connected open subset $\mathfrak{C}$ of a totally real smooth submanifold $Y$ of $\Omega$ so that $\mathfrak{C}$ has real dimension $n$ and $\bar{\mathfrak{C}}\subset Y$ and $\del\mathfrak{C}$ in $Y$ is Lipschitz continuous. Here, a linear subspace $L$ of $\CC^n$ is \emph{totally real} if $L\cap i L = \{0\}$, see also the paragraph directly after \cite[Def.~9.1.14]{MR4436039}. 
    
    If $f \in C(\Omega;\RR)$, a contour of integration $\mathfrak{C}$ in $\Omega$ is \emph{strongly well-shaped} for $f$ at $\hat z \in \mathfrak{C}$ if there is $c>0$ so that
    \begin{equation}\label{eq:f_neg_def}
            f(z)-f(\hat z) \leq -c \abs{z-\hat z}^2
    \end{equation}
    for all $z\in \mathfrak{C}$. 

    If $f \in C^\infty(\Omega;\RR)$, then $\hat z \in \Omega$ is a \emph{saddle point} of $f$ if $d f(\hat z) = 0$ and $\mathrm{sign}(\del_{(z^\RR,z^{i\RR})}^2 f)(\hat z) = (n,n)$ (which automatically implies the non-degeneracy of this Hessian). Here $\mathrm{sign}$ is the \emph{signature} of a matrix, which is the pair of amounts of positive and negative eigenvalues of said matrix. 
\end{definition}

In the proof of the main result we will encounter an elliptic germ FIO-like operator which we would like to invert. A central assumption to facilitate this inversion will be establishing the following property for its phase.
\begin{definition}[{\cite[Def.~20.5.1]{MR4436039}}]\label{def:sjostrand_Triad}
	Let $\Omega_j \subset \CC^{n_j}$, $\Xi \subset \CC^N$, $j\in \{1,2\}$, $\hat z \in \Omega_1$, $\hat w\in \Omega_2$, $\hat y\in \Xi$ and $\varphi \in \mathcal{O}(\Omega_1\times\Omega_2\times \Xi)$, and $\Phi_j \in C^\infty(\Omega_j;\RR)$ exponent weights so that $\varphi(\hat z, \hat w, \hat y) = 0 = \Phi_1(\hat z) = \Phi_2(\hat w)$. Define
    \[
        \Psi_1(z,w,y) = \Phi_1(z)-\varphi^{i\RR}(z,w,y)\,,\quad \Psi_2(z,w,y) = \Phi_2(w)+\varphi^{i\RR}(z,w,y)\,.
    \]

	We call $(\varphi,\Phi_1,\Phi_2)$ a \emph{Sj\"ostrand triad at $(\hat z, \hat w, \hat y)$} if
    \begin{enumerate}
        \item $(\hat z, \hat y)$ is a saddle point of $\Psi_1(z,\hat w, y)$,
        \item $(\hat w, \hat y)$ is a saddle point of $\Psi_2(\hat z, w, y)$,
        \item the critical value of $(z,y)\mapsto \Psi_1(z,w,y)$ is equal to $\Phi_2(w)$ near $(\hat z, \hat w, \hat y)$.
    \end{enumerate}
\end{definition}
Throughout \cref{sec:localinj}, we will reference Fourier-like and Fourier integral operator-like operators defined on weighted germ spaces, for which we refer to \cite[Part~VI]{MR4436039}, see also \cite{AST_1982__95__R3_0}, \cite{zbMATH07060662}. We shall, however, introduce the notion of an FBI phase from \cite[Def.~22.1.1]{MR4436039}, see also \cite[\S~7]{AST_1982__95__R3_0}, \cite[\S~2.6]{zbMATH07060662}.
\begin{definition}
	Let $\Omega,\Xi\subset \CC^n$ be domains, $\hat z \in \Omega, \hat y \in \Xi \cap \RR^n$ and $\varphi \in \mathcal{O}(\Omega\times \Xi)$. We call $\varphi$ an \emph{FBI phase function at $(\hat z, \hat y)$} if
\begin{enumerate}
	\item $\varphi(\hat z, \hat y) = 0, \del_y \varphi(\hat z, \hat y) = -\hat \xi \in \RR^n\sem$,
	\item $\det \del_{z}\del_y\varphi(\hat z, \hat y) \neq 0$, 
	\item $\mathrm{Im} \del_y^2 \varphi(\hat z, \hat y)$ is positive definite.
\end{enumerate}
\end{definition}
In the definition above one may think of $y$ as the input variable and $z$ as the output variable. FBI phases capture microlocal information by allowing the output $z$ to be complex: its real part corresponds to the base and its imaginary part replaces the usual fiber variable, see also the remarks near \cite[Eq.~(3.1.2)]{MR1872698} and the beginning of \cite[\S~9.6]{hoermander1}. The usual example of an FBI phase is
\[
    \varphi(z,y) = -(y-\hat y)\cdot z +\frac{i}{2} (z-\hat y)\cdot (z-\hat y)
\]
with distinguished points $\hat z= 0$ and $\hat y$.

In the analytic setting we shall work with classical analytic symbols, see \cite{AST_1982__95__R3_0}.
\begin{definition}\label{def:sym}
	For $m\in \RR, \Omega \subset \CC^n$ open and $\Gamma \subset \CC^N\sem$ an open set conic under real dilations, when $(p_j)_{j\in\mathbb{N}} \subset \mathcal{O}(\Omega\times \Gamma)$ is a sequence that is positively homogeneous of degree $m-j$ in $\xi\in \Gamma$ satisfying: for all compact $K\subset \Omega \times (\Gamma\cap\{\abs{\xi}=1\})$ there is $C>0$ so that $\sup_K \abs{p_j} \leq C^{j+1}j!$ for all $j$, we call the formal sum $\ud{p}(z,\xi) \coloneqq \sum_{j\geq 0} p_j(z,\xi)$ a \emph{formal classical analytic symbol (formal cl.a.s.) of order $m$}. 

	We say that a measurable function $p(z,\xi)$ with $p(\cdot,\xi) \in \mathcal{O}(\Omega)$ for all $\xi\in \Gamma$ is a \emph{finite realization} of $\ud{p}$ if for all sufficiently large $R>0$, for all $K\subset \Omega\times (\Gamma\cap \{\abs{\xi}=1\})$ compact there are $C,\eps>0$ so that for all $h \in (0,1]$,
	\[
		\sup_{(z,\xi)\in K}\abs{p(z,h^{-1}\xi) - \sum_{0 \leq j\leq R^{-1}h^{-1}} h^{j-m}p_j(z,\xi)} \leq Ce^{-\eps h^{-1}}\,,
	\]
	in which case we write $p\sim \ud{p}$ and say that $p$ is a \emph{classical analytic symbol (cl.a.s.).} We shall also use the notation $p(z,\xi;h)$ for $p(z,h^{-1}\xi)$ and $\ud{p}(z,\xi;h)$ for $\sum_{j\geq 0} h^{j-m}p_j(z,\xi)$ and restrict $\xi$ to a compact subset of $\Gamma$ in this case. 
\end{definition}
That finite realizations of such formal symbols exist is shown in \cite[\S~1]{AST_1982__95__R3_0}, \cite[\S~17.2.7]{MR4436039}.

We will also require the notion of a sequence of cut-offs satisfying some growth conditions.
\begin{definition}
	Let $U' \subset U$ be two open subsets of $\RR^n$ with $\mathrm{dist}(\RR^n\setminus U, U') > 0$. Let $(\chi_j)_{j\in\mathbb{N}}\subset C_c^\infty(\RR^n)$ be a bounded sequence (in $C_c^\infty$) satisfying $0 \leq \chi_j \leq 1$ and $\chi_j\vert_{U'} \equiv 1$ and $\chi_j\vert_{\RR^n\setminus U} \equiv 0$ for all $j\in \mathbb{N}$. If there is some $C>0$ so that for all multi-indices $\beta$ there is $C_\beta > 0$ so that for all $j$ we have
	\[
		\max\abs{\del^\alpha \del^\beta \chi_j} \leq C_\beta (C j)^{\abs{\alpha}} \quad \forall \abs{\alpha} \leq j\,,
	\]
	then $(\chi_j)_{j\in \mathbb{N}}$ is called an \emph{Ehrenpreis cut-off sequence}.
\end{definition}
That such a sequence always exists follows from \cite[Thm.~1.4.10]{hoermander1}, see also \cite[Prop.~3.2.1]{MR4436039}, \cite[Lem.~2.2]{zbMATH03359011}. 

To have convenient notation for a set to which an analytic positively homogeneous function (think phases) can be holomorphically extended, we introduce 
\[
	\Theta_\eps \coloneqq \left\{\eta' \in \CC^{n-1} \colon \eps\abs{(\eta')^{\RR}} > \abs{(\eta')^{i\RR}}\right\}\,,\quad \Theta_\eps^\circ \coloneqq \Theta_\eps \cup B_\eps^{\CC^{n-1}}(0)\,.
\]
Throughout the paper we shall use the notation $\abs{\ccdot}_+ \coloneqq \max\{1,\abs{\ccdot}\}$. 

\section{Local solution constructions}\label{sec:localsol}

Recall that we assumed that $P$ does not have Dirichlet eigenvalue $0$ on $(M,g)$, so that the Green's and Poisson kernel are well-defined, see \cref{prop:GK}. Let $\hat y \in \del M$ be fixed and assume that $\partial M$ is analytic near $\hat y$ and let $W$ be an open neighborhood of $\hat y$ on which we can introduce analytic boundary normal coordinates. We assume also that the metric $g$ and the coefficients of $P$, defined in \eqref{eq:def_of_P}, are analytic on $W$.  In these coordinates in $W$, we shall assume $\hat y = 0$, and $W$ is a subset of $\{y_n\geq 0\}$.

We may always assume that in $W$ we can write
\[
    -P = D_{y_n}^2 + A(y)D_{y_n} + R(y, D_{y'})
\]
where $R$ is a second order differential operator with positive principal symbol $r(y,\eta')$. This follows from \cite[Thm.~C.5.1,~Cor.~C.5.3]{zbMATH05129478} in the case $Q = 0$, and writing $Q(y,\del_{y}) = A(y)D_{y_n} + Q'(y,D_{y'})$, %
we have the desired conclusion. %

In the following we will write $r^{1/2}(y,\eta')$ for the principal branch of the square root of $r$, which is holomorphic for complex $y$ near $0$ and $\eta'\in \Theta_\eps$ for some $\eps >0$. We point out for emphasis that the positivity of $r$ implies that for all complex $y$ near $0$ and $\eta'$ away from the branch cut we have $\mathrm{Re}r^{1/2}(y,\eta') > 0$. %

Before the next proposition, we fix the notation. Let $\xi'_0 \in \RR^{n-1}$ be the fixed tangential frequency. We write $y = (y',y_n) \in M$ for a point in boundary normal coordinates, with $y_n \geq 0$ and $\partial M = \{y_n = 0\}$, and let $x' \in \RR^{n-1}$ denote the boundary variable, so that $x' = y'$ on $\partial M$. The output parameters are $(v',\tau) \in U_1 \subset \CC^{n-1} \times \CC$, a complex neighborhood of $(0,1)$. %
Throughout, $h \in (0,1]$ is the semiclassical parameter. The operator $K$ denotes the Poisson operator sending boundary data on $\partial M$ to the solution of $Pu = 0$ in $M$.

\begin{proposition}\label{prop:K1}
    For fixed $\xi'_0 \in \RR^{n-1}\sem$, define $\psi_0(v',\tau,x') \coloneqq \frac{i}{2}(v'-x')^2+\frac{i}{2}(\tau-1)^2 - x'\cdot \xi'_0$ which is an entire function. There is a complex open neighborhood $U=U_1\times U_2$ of $((0,1),0) \in \CC^{n}\times \CC^{n}$ so that the following is true. For any $\chi \in C_c^\infty(U_2\cap \RR^n)$ identically $1$ near $0$, for $y\in M$ and $(v',\tau)\in U_1$ and $h>0$ sufficiently small, we have
    \[
        K(\chi\vert_{\{x_n=0\}}e^{\frac{i}{h}\psi_0})(v',\tau,y;h) = \chi(y) e^{\frac{i}{h}\phi_0(v',\tau,y)} b(v',\tau,y;h) + C^\infty(M;\mathcal{O}^{(-\omega)}(U_1))\,,
    \]
    where $b(v',\tau,y;h)$ is any finite realization of some elliptic analytic symbol of order $0$ satisfying
    \[
        \mathcal{O}(U)\ni b(v',\tau,y;h) \sim \sum_{j\geq 0}h^j b_j(v',\tau,y)\,,\quad %
	(b_j)_{j\in \mathbb{N}} \subset \mathcal{O}(U)\,,\quad b_j(v',\tau,y',0) = \delta_{0j}\,, \quad b(v',\tau,y',0)=1\,,
    \]
    where $\delta_{0j}$ is the Kronecker delta and $\phi_0 \in \mathcal{O}(U)$ satisfies 
    \[
        \phi_0(v',\tau,y',0) = \psi_0(v',\tau,y')\,,\quad (\del_{y_n}\phi_0)(v',\tau,y) = ir^{1/2}(y,\del_{y'}\phi_0(v',\tau,y))\,.%
    \]
    Here and in the proof below, we write $f \in C^\infty(M;\mathcal{O}^{(-\omega)}(U_1))$ if $f(y,(v',\tau))\in C^\infty(M\times U_1)$ and for all $\alpha$ and all $y\in M$ we have $\del_y^\alpha f(y, \cdot) \in \mathcal{O}^{(-\omega)}(U_1)$.
\end{proposition}
\begin{proof}
    The Cauchy-Kovalevskaya theorem \cite[Thm.~9.4.6]{hoermander1}  guarantees the existence of $\phi_0$ satisfying the stated properties, i.e. it solves the eikonal equation (we use that the hypersurface $\{y_n=0\}$ is non-characteristic). 

    According to the analytic WKB method (\cite[Prop.~23.2.7]{MR4436039} or after reduction to a zero-th order equation \cite[Thm.~2.8.1]{zbMATH07060662}) there is a formal analytic symbol $\ud{b} = \sum_{j\geq 0} b_j(v',\tau,y)$ so that $P(e^{\frac{i}{h}\phi_0}\ud{b}) = 0$ formally, and the $b_j$ satisfy the properties in the statement. Letting $b \sim \ud{b}$ be a finite realization satisfying the properties imposed by the statement, \cite[Prop.~24.5.2]{MR4436039} implies that 
    \[
        P(e^{\frac{i}{h}\phi_0}b) \in \mathcal{O}^{(-\omega)}(U)
    \]
    if $U$ is sufficiently contracted (ie. chosen sufficiently small) about $0$.

    For $\chi$ as in the statement we define
    \[
        u = \chi e^{\frac{i}{h}\phi_0}b\,, \qquad \hat u = u - GP u\,.
    \]
    One easily verifies that $\hat u\vert_{\partial M} = \chi\vert_{\{y_n=0\}} e^{\frac{i}{h}\psi_0}$ and $P\hat u = 0$ using the properties of the Green's function $G$, so that indeed
    \[
        K(\chi\vert_{\{y_n=0\}}e^{\frac{i}{h}\psi_0}) = \hat u\,.
    \]
    Thus, showing that $GP u \in C^\infty(M;\mathcal{O}^{(-\omega)}(U_1))$ will complete the proof. Note that
    \[
        P u = \chi P(e^{\frac{i}{h}\phi_0}b) + [P,\chi]e^{\frac{i}{h}\phi_0}b\,,
    \]
    where $[P,\chi]$ is supported away from $0$. By the fundamental theorem of calculus, we have
    \[
        \phi_0(v',\tau,y) = \psi_0(v',\tau,y') + iy_n\int_0^1 r^{1/2}(y',sy_n,\del_{y'}\phi_0(v',\tau,y',sy_n)) ds\,,
    \]
    so that $\phi_0^{i\RR} > \delta > 0$ on the support of $[P,\chi]$ as long as $\abs{(v')^{i\RR}}, \abs{\tau^{i\RR}}$ are sufficiently small (which we can phrase as taking $U_1$ sufficiently small). Here we use the positivity of $(r^{1/2})^{\RR}$.

    We conclude that $P u \in C^\infty(M; \mathcal{O}^{(-\omega)}(U_1))$ and the continuity of $G$ on $H^s(M), s\geq 0$ implies that $\sup_{v' \in U_1}\norm{GPu}_{H^{s+2}(M)} \leq C \sup_{v' \in U_1}\norm{Pu}_{H^s(M)}$ with $C>0$ independent of $u$, which together with elliptic regularity (\cite[Cor.~8.3.2]{hoermander1}) gives
    \[
        GPu \in C^\infty(M; \mathcal{O}^{(-\omega)}(U_1))\,,
    \]
    completing the proof. 
\end{proof}
\begin{remark}\label{rem:K1imp}
	The result of \cref{prop:K1} remains true with all constants uniform for $\xi'_0$ if we allow $\xi'_0$ to vary in a compact subset of $\RR^{n-1}\sem$. 
\end{remark}

As explained in the introduction, it would suffice to rely entirely on solutions coming from \cref{prop:K1} to prove \cref{thm:harmonic}. To prove \cref{thm:symb} it becomes necessary to prove a similar statement when the phase $\psi_0$ above is replaced by the usual plane wave $x'\cdot \xi'$. The main difference to the previous result is that the characterization of the solution is valid only in some neighborhood of the origin and not on all of $M$. Furthermore, since we have lost the localization provided by the Gaussian factor in the phase, we will work with a sequence of cut-off functions.
\begin{proposition}\label{prop:K2}
     Let $\xi'_0 \in S^{n-2}$ be fixed, and define $\psi_1(\tau,x') = -\tau x'\cdot \xi'_0$ which is entire. There is a complex open neighborhood $U=U_1\times U_2$ of $(1,0) \in \CC\times \CC^{n}$ so that the following is true. For any $(\chi_j)_{j\in\mathbb{N}} \subset C_c^\infty(U_2)$ Ehrenpreis cut-off sequence identically $1$ near $0$ there is some complex open neighborhood $U^\ast$ of $0$ so that for all $\tau \in U_1$ and $y\in U^\ast$, and $h>0$ sufficiently small, we have
    \[
        K(\chi_{j}\vert_{\{x_n=0\}}e^{\frac{i}{h}\psi_1})(\tau,y;h) = e^{\frac{i}{h}\phi_1(\tau, y)}d(\tau,y;h) + R_j(\tau,y;h)\,,\quad \abs{R_j} \leq h^m C (Cjh)^j e^{\eps/h(\abs{\tau^{i\RR}}+\abs{y^{i\RR}})}
    \]
    where $\eps, m \in \RR, C>0$ are independent of $j$ and $R_j$ is holomorphic with respect to $(\tau, y)\in U_1\times U^\ast$, and where $d(\tau,y;h)$ is any finite realization of some elliptic classical analytic symbol of order $0$ satisfying
    \[
        \mathcal{O}(U)\ni d(\tau,y;h) \sim \sum_{j\geq 0}h^j d_j(\tau,y)\,,%
	\quad (d_j)_{j\in \mathbb{N}} \subset \mathcal{O}(U)\,,\quad d_j(\tau,y',0) = \delta_{0,j}\,,\qquad d(\tau,y',0) = 1\,,
    \]
    and $\phi_1 \in \mathcal{O}(U)$ is homogeneous of degree $1$ in $\tau$ and satisfies 
    \begin{equation}\label{eq:phi}
		\phi_1(\tau, y',0) = \psi_1(\tau,y') = -\tau y'\cdot\xi'_0\,,\quad \del_{y_n}\phi_1(\tau,y) =  ir^{1/2}(y,\del_{y'}\phi_1(\tau,y))\,.
	\end{equation}
	In particular, there is $\vartheta_0>0$ so that for all $\vartheta \in (0,\vartheta_0]$, defining $j_\vartheta(h) \coloneqq \lfloor \vartheta h^{-1}\rfloor$, possibly contracting $U_1$ about $1$ and $U^\ast$ about $0$, we have
	\[
		K(\chi_{j_\vartheta(h)}\vert_{\{x_n=0\}}e^{\frac{i}{h}\psi_1})(\tau,y;h) = e^{\frac{i}{h}\phi_1(\tau, y)}d(\tau,y;h) + \mathcal{O}^{(-\omega)}(U_1\times U^\ast)\,.
	\]
\end{proposition}
\begin{proof}
    As in the previous result, we employ the Cauchy-Kovalevskaya theorem and analytic WKB method to construct $u = \chi_{j} e^{\frac{i}{h}\phi_1}d$ and $\hat u = K(\chi_{j}\vert_{\{x_n=0\}}e^{\frac{i}{h}\psi_1})(\tau,y;h) = u - GPu$, and we merely have to investigate the error term $GP u$. Note again that
    \[
        Pu = \chi_{j} P(e^{\frac{i}{h}\phi_1}d) + [P,\chi_{j}]e^{\frac{i}{h}\phi_1}d\,,
    \]
    where the second term vanishes on the closed set $U_0 = \{\chi_{j}\equiv 1\forall j\}$. We assume that $U_1, U_2$ (and thus $\supp \chi_{j}, U_0$) are sufficiently small so that $P(e^{\frac{i}{h}\phi_1}d)\in \mathcal{O}^{(-\omega)}(U_1\times U_2)$. In particular, $G(\chi_{j} P(e^{\frac{i}{h}\phi_1}d))$ satisfies the bounds imposed on $R_j$ for all $j$.

    To estimate the second term above, note that for $z$ near $0$
    \[
	G([P,\chi_{j}]e^{\frac{i}{h}\phi_1}d)(z) = \int_M G(y,z) [P,\chi_{j}]e^{\frac{i}{h}\phi_1}d \dd V_g(y)\,,
    \]
    and the kernel $G(y,z)$ of the Green's function is jointly analytic in its domain of integration. We thus perform $j$ partial integrations with the operator $L = \frac{h}{i} \frac{(\del_{y'} \phi_1)^\ast \del_{y'}}{\abs{\del_{y'} \phi_1}^2}$ to obtain 
    \[
	G([P,\chi_{j}]e^{\frac{i}{h}\phi_1}d)(z) = \int_M e^{\frac{i}{h}\phi_1} (L^\ast)^j(G(y,z) [P,\chi_{j}]d)\dd V_g(y)\,,
    \]
    which, using that $\phi_1(\tau,y',0)= -\tau y'\cdot \xi'_0$, is bounded above by $h^m C(Chj)^{j}e^{\eps/h(\abs{\tau^{i\RR}}+\abs{y^{i\RR}})}$ for some $\eps, C>0, m\in\RR$ independent of $j$. For emphasis we point out that we have thus differentiated $\chi_j$ up to $j+2$ many times, but this merely modifies the constant $C>0$.

    Finally, choosing $\vartheta_0 = e^{-1}C^{-1}$ for the $C$ arising from the repeated partial integrations, %
    we obtain the desired exponential decay if $(\tau, y) \in U^\ast$ with complex open $U^\ast$ chosen sufficiently contracted about $0$.
	This concludes showing that for some $U^\ast$ we have $GPu \in \mathcal{O}^{(-\omega)}(U_1\times U^\ast)$.  
\end{proof}

\begin{remark}\label{rem:uniformity}
    The construction in Proposition~\ref{prop:K2} is uniform in  near a fixed direction: for every $\xi'_{0} \in S^{n-2}$ there is a neighborhood $V \subset S^{n-2}$ of $\xi'_{0}$ such that the neighborhoods $U_1$, $U^\ast$, the constants $\eps, m, C, \vartheta_0 > 0$, and the symbol $d$ can be chosen uniformly for all $\xi'_{0\ast} \in V$. This follows from continuous dependence of the construction on $\xi'_0$ as the Cauchy--Kovalevskaya solution, the transport equations, and the error estimates all depend continuously on $\xi'_0$.
\end{remark}

Combining the two results we obtain
\begin{corollary}\label{cor:Ks}
    Let $q\in L^\infty(M)$ and $\xi'_0 \in S^{n-2}$ be fixed. Let $\chi, (\td{\chi})_{j_\vartheta(h)} \in C_c^\infty(W)$ be supported sufficiently close to $0$ and identically $1$ near $0$, where $\supp \chi \subset \{\td{\chi}_{j_\vartheta(h)}\equiv 1\}$ and $\td{\chi}_{j_\vartheta(h)}$ is the $j_\vartheta(h)$-th element of a sequence of Ehrenpreis cut-offs as in \cref{prop:K2}. We have 
    \[
    \langle K^\ast qK\td{\chi}_{j_\vartheta(h)} e^{\frac{i}{h}\psi_1}, \chi e^{\frac{i}{h}\psi_0}\rangle = \int_M q \overline{K(\chi e^{\frac{i}{h}\psi_0})}K(\td{\chi}_{j_\vartheta(h)} e^{\frac{i}{h}\psi_1}) d V_g(y) = \int_M e^{\frac{i}{h}(-\phi_0^\ast+\phi_1)}b^\ast d \chi q \dd V_g(y) + \bl{O}^{(-\omega)}_{(0,1)}\,,
    \]
    which is a function of $(v',\tau) \in \CC^{n}$ defined near $(0,1)$. 
\end{corollary}
\begin{proof}
    Due to \cref{prop:GK}, for any compact $V\subset U_1$, there are $C,\eps>0$ so that
    \[
        \sup_{\tau \in V}\norm{K(\td{\chi}_{j_\vartheta(h)} e^{\frac{i}{h} \psi_1})}_{H^2(M)} \leq C\sup_{\tau \in V}\norm{\td{\chi}_{j_\vartheta(h)} e^{\frac{i}{h}\psi_1}}_{H^{3/2}(\del M)} \leq Ch^{-2}\sup_{\tau \in V} e^{\frac{\eps}{h} \abs{\tau^{i\RR}}}\,,
    \]
    and we obtain from \cref{prop:K1} that, for $\abs{\tau^{i\RR}}$ sufficiently small (ie. $\tau$ sufficiently near $1$ in $\CC$),  
    \[
        \int_M q \overline{K(\chi e^{\frac{i}{h}\psi_0})} K(\td{\chi}_{j_\vartheta(h)} e^{\frac{i}{h}\psi_1}) d V_g(y) = \int_M q\chi e^{-\frac{i}{h}\phi_0^\ast} b^\ast K(\td{\chi}_{j_\vartheta(h)} e^{\frac{i}{h}\psi_1}) d V_g(y) + \bl{O}_{(0,1)}^{(-\omega)}\,.
    \]
    A similar application of \cref{prop:K2} completes the proof.
\end{proof}

Furthermore, using \cref{prop:K2}, we have a formula for the kernel of $K$ near the diagonal. We represent 
it as an oscillatory integral (an FIO kernel) plus an analytic error.
\begin{corollary}\label{cor:kerndia}
    There are complex open neighborhoods $U_1$ of $\hat y=0$ and $\eps>0$ so that for $y \in U_1 \cap M$ and $x' \in U_1 \cap \del M$, we have %
    \[
        K(y,x') = \int e^{i((y'-x')\cdot \eta' + \psi(y,\eta'))} a(y,\eta') \dd \eta' + K_a(y,x')\,,\quad K_a \in \bl{O}_{(0,0)}\,,
    \]
    where $a(y,\eta') \in \mathcal{O}(U_1\times \Theta_\eps^\circ)$ is the finite realization of some elliptic classical analytic symbol of order $0$, 
    \begin{equation}\label{eq:symbola}
        a(y,\eta') \sim \sum_{j\geq 0}\abs{\eta'}^{-j} a_j(y,\abs{\eta'}^{-1}\eta')\,,\quad %
	(a_j)_{j\in \mathbb{N}} \subset \mathcal{O}(U_1\times U_2)\,,\quad a_j(y',0,\eta') = (2\pi)^{1-n}\delta_{0j}\,,\quad a(y',0,\eta')=(2\pi)^{1-n}\,,
    \end{equation}
    and $\psi = \phi - y'\cdot \eta'$, where $\phi \in \mathcal{O}(U_1\times \Theta_\eps)$ is homogeneous of degree $1$ in $\eta'$ and satisfies 
    \begin{equation}\label{eq:phi2}
		\phi(y',0,\eta') = y'\cdot\eta'\,,\quad \del_{y_n}\phi(y,\eta') =  ir^{1/2}(y,\del_{y'}\phi(y,\eta'))\,.
    \end{equation}
\end{corollary}
\begin{proof} 
	Using the compactness of the unit sphere and homogeneity of degree $1$ of the phase $\phi_1$ (writing $h^{-1}\tau \xi'_0 = \eta'$, $h^{-1} = \abs{\eta'}, \abs{\tau}=1$) we glue local (unique) solutions from \cref{prop:K2} together to obtain $\phi$ and $a$ as in the statement, so that
	\[
		P \left(e^{i(y'\cdot \eta'+\psi(y,\eta'))}a(y,\eta')\right) = R(y,\eta')
	\]
	where $R$ is some function that is holomorphic in $y$ near $0$ and measurable in $\eta'$ so that $\abs{R(y,\eta')} \leq Ce^{-\eps \abs{\eta'}}$ for some $C,\eps>0$ and all $y$ in an open complex neighborhood of $0$. This argument required relying on \cref{rem:uniformity}. 
	
	Thus, if we define
	\[
		A(y,x') \coloneqq \int e^{i((y'-x')\cdot \eta' + \psi(y,\eta'))} a(y,\eta')\dd\eta'\,,
	\]
	as an oscillatory integral for $y$ near $0$, %
	we have 
	\[
		P_yA(y,x') = \int e^{-ix'\cdot \eta'}R(y,\eta')\dd\eta'\,,
	\]
	which can be holomorphically extended in $(y,x')$ near $(0,0)$ since the integrand remains absolutely integrable in $\eta'$ there.
	Similarly, because $\psi\vert_{y_n=0}=0$ and $a\vert_{y_n=0}=(2\pi)^{1-n}$, we have $A(y,x')\vert_{y_n = 0} = \delta(x'-y')$. 

	Finally, after extending $A(y,x')$ smoothly outside of the coordinate neighborhood and appealing to local analytic elliptic boundary regularity (\cite[Thm.~A]{zbMATH03134527}, applicable by reduction to zero-boundary data using the Cauchy-Kovalevskaya theorem), we deduce that $K_{a}(y,x') \coloneqq K(y,x')-A(y,x')$ is analytic in $y$ near $0$ for each fixed complex $x'$ near $0$. In fact, \cite[Thm.~2]{zbMATH03134527} states that the neighborhood of analyticity in $y$ of $K_a$ is locally uniform with respect to $x'$. 

	Note that we have, for $R' \coloneqq \int e^{-ix'\cdot\eta'}R(y,\eta')\dd\eta' = P_yA$, writing $K_a = K-A$ and $K_a = G_yP_y K_a + K_y K_a$, and using that $K_a\vert_{y_n=0} = K - A\vert_{y_n=0}=0$, we obtain
	\[
		-K_a(\cdot, x') = G P(A-K)(\cdot,x') = GR'(\cdot,x')
	\]
	for all complex $x'$ sufficiently near $0$ (where we suppressed notation for smoothly extending functions above). Since $G, K$ are continuous linear operators with the properties stated in \cref{prop:GK}, using the Sobolev embedding, for each fixed real $y$ near $0$, and all multi-indices $\alpha$
	\[
		x' \mapsto \del^\alpha_y K_a(y,x')
	\]
	is holomorphic in some fixed complex neighborhood of $0$ independent of $\alpha$. Using the local uniformity (with respect to $x'$) of the neighborhood on which $K_a$ is analytic in $y$, we have
	\begin{equation}\label{eq:joint}
		\sup_{y \in U, x'\in V}\abs{\del^\alpha_y K_a(y,x')} \leq C R^{\abs{\alpha}}\alpha!\quad \forall \alpha
	\end{equation}
	for $V$ some complex neighborhood of $0$ and $U$ some real neighborhood of $0$, and $C,R>0$ some constants. In particular, for fixed $y_0 \in U$ real, 
	\[
		\sum_{\alpha \in \mathbb{N}^n} \frac{\del^\alpha_y K_a(y_0,x')}{\alpha!}(y-y_0)^\alpha
	\]
	is absolutely convergent for complex $y$ near $y_0$, and since each coefficient of this series depends holomorphically on $x'$, using \cref{eq:joint}, this series defines a holomorphic function in $(y,x')$ near $(0,0)$. Since this series agrees with $K_a(y,x')$ when $y$ is real, the proof is complete. 
\end{proof}

We give a description of the symbol of $-\dot \Lambda_q$. 

\begin{proposition}\label{prop:anasymb}%
	Let $q\in L^\infty(M)$ and assume that $\dot\Lambda_q(x',z')$ is analytic on $W'\times W'\setminus \Delta_{W'}$, where $W' \subset W$ is an open neighborhood of the origin. Let $r_0>0$ be sufficiently small and non-negative $\rho \in C^\infty(B_{r_0}(0))$ vanishing on $\abs{t'} \leq r_0/2$ and identically $1$ on $B_{r_0}(0)\setminus B_{5r_0/8}(0)$. For $\gamma_\theta(t') = t'+ir_0\rho(t')\theta\abs{\theta}^{-1}$, define the contour integral, for $\xi'\in \RR^{n-1}\sem$, real $x'$ near $0$,
	\[
		\sigma_a^q(x',\xi') = - \int_{\gamma_{\xi'}(B_{r_0}(0))} e^{i(z'-x')\cdot\xi'}\dot\Lambda_q(x',z')\dd z'\,.
	\]
	For $x', z'$ sufficiently near the origin, we have
	\begin{equation}\label{eq:sigma_aissymb}
		(2\pi)^{1-n}\int e^{i(x'-z')\cdot\xi'} \sigma_a^q(x',\xi') \dd \xi' = -\dot\Lambda_q(x',z') + R(x',z')\,,
	\end{equation}
	where $R$ is analytic in a neighborhood of $(x',z')$.

	Furthermore, if $(\chi_N)_{N\in\mathbb{N}} \subset C_c^\infty(B_{7r_0/8}(0))$ is a sequence of Ehrenpreis cut-offs that is identically $1$ on $B_{3r_0/4}(0)$, then there is some open complex neighborhood $U$ of $(0,1)$ so that for all compact $K\subset U$ and $(x',\tau)\in K$ there are some $C,c,c'>0$ so that for all $\xi'_0\in S^{n-2}$,
	\[
		\abs{\sigma_a^q(x',h^{-1}\tau\xi'_0)+e^{-\frac{i}{h}\tau x'\cdot\xi'_0}\dot\Lambda_q(\chi_N(\cdot)e^{\frac{i}{h} \tau(\cdot)\cdot\xi'_0})(x')}\leq C^{N+1}e^{\frac{c'}{h}(\abs{(x')^{i\RR}}+\abs{\tau^{i\RR}})}((Nh)^N+e^{-c\tau^\RR r_0/h})
	\]
	for all $N\geq 1$, and the term in the absolute value on the LHS is holomorphic with respect to $(x',\tau)$ in $U$.
\end{proposition}
\begin{proof}
	Write $\xi' = h^{-1}\tau \xi'_0$, and assume for now that $\tau=1$ and $x'$ is real near $0$. If $\beta \in C_c^\infty(B_{r_0}(0))$ is any function that is identically $1$ on $B_{5r_0/8}(0)$, the definition of $\sigma_a^q$ gives
	\begin{equation}\label{eq:sigmaaasdistro}
		-\sigma_a^q(x',\xi') = \langle \dot\Lambda_q(x',t'), \beta(t') e^{i(t'-x')\cdot \xi'} \rangle_{t'} + \int_{B_{r_0}(0)} (1-\beta(t')) e^{i(\gamma_{\xi'_0}(t')-x')\cdot\xi'} \dot\Lambda_q(x',\gamma_{\xi'_0}(t'))J_{\xi'_0}(t')\dd t'\,,
	\end{equation}
	where the first term is a distributional pairing and the second is an ordinary integral where $J_{\theta}(t') = \mathrm{det}(\id+ir_0 \nabla_{t'}\rho\otimes \theta\abs{\theta}^{-1})$, which is well-defined for $r_0>0$ sufficiently small because on the support of $(1-\beta)$ the kernel $\dot\Lambda_q$ has a holomorphic extension. Notice also that the second term is supported where $i\gamma_{\xi'}(t')\cdot\xi' = it' \cdot\xi' - r_0\abs{\xi'}$. In particular, there is an open complex neighborhood of $0$ on which the second term has a holomorphic extension in $x'$ for all $\xi' \neq 0$ and on which the term is bounded above by $Ce^{-c\abs{\xi'}}$ for some $C,c>0$. Since such exponentially decaying terms lead to analytic kernels (\cite[17.1.19]{MR4436039}), and the first term will, by ordinary Fourier inversion, lead to $-\beta(z')\dot\Lambda_q(x',z')$, the proof of \cref{eq:sigma_aissymb} is complete.

	From now on we will allow $(x',\tau)$ to be complex near $(0,1)$ so that \cref{eq:sigmaaasdistro} is seen to be the holomorphic extension of $-\sigma_{a}^q(x',h^{-1}\tau\xi'_0)$ in these variables. Recall that we use the notation $\xi' = h^{-1}\tau\xi'_0$.

	For $\chi_N$ as in the statement, define
	\[
		\td{\chi}_N(t',s) \coloneqq \sum_{k=0}^{N-1} \frac{(isr_0)^k}{k!}\left(\xi'_0\del_{t'}\right)^k \chi_N(t')\,,
	\]
	the order $N$ almost holomorphic extension of $\chi_N$ and remark that by a telescoping property,
	\[
		\abs{\left(\del_s -ir_0\xi'_0\del_{t'}\right)\td{\chi}_N(t',s)} = \abs{-ir_0 \frac{(isr_0)^{N-1}}{(N-1)!}\left(\xi'_0\del_{t'}\right)^N\chi_N(t')}  \leq C \frac{(Cr_0 N)^N}{(N-1)!} s^{N-1}\,.
	\]
	Define the integral
	\[
		I_N(x',\xi';s) \coloneqq -\int_{B_{r_0}(0)} e^{i(\gamma_{\xi'_0}(t',s)-x')\cdot\xi'}\dot\Lambda_q(x',\gamma_{\xi'_0}(t',s))\td{\chi}_N(t',s)J_{\xi'_0}(t',s) \dd t'\,,
	\]
	where $\gamma_{\xi'_0}(t',s) = t'+isr_0\rho(t')\xi'_0$ and $J_{\xi'_0}(t',s) = \det(\id + isr_0 \nabla_{t'}\rho(t')\otimes \xi'_0)$. We calculate that
	\begin{equation}\label{eq:delsIn}
		-\del_s I_N(x',\xi';s) = \int_{B_{r_0}(0)} e^{i(\gamma_{\xi'_0}(t',s)-x')\cdot\xi'}\dot\Lambda_q(x',\gamma_{\xi'_0}(t',s))\left(\del_s-ir_0\xi'_0\del_{t'}\right)\td{\chi}_N(t',s)\dd t'\,,
	\end{equation}
	where we used integration by parts, the identity
	\[
		\del_s(e^{i(\gamma_{\xi'_0}(t',s)-x')\cdot\xi'}\dot\Lambda_q(x',\gamma_{\xi'_0}(t',s))J_{\xi'_0}(t',s)) = ir_0\xi'_0\del_{t'}(\rho e^{i(w'-x')\cdot\xi'}\dot\Lambda_q(x',w'))\vert_{w' = \gamma_{\xi'_0}(t',s)}\,,
	\]
	and that $\rho \equiv 1$ on $\supp \del_s \td{\chi}_N, \supp \nabla_{t'}\td{\chi}_N$. By plugging in $s=0$, we obtain
	\[
		I_N(x',\xi';0) = -e^{-ix'\cdot\xi'}\dot\Lambda_q(\chi_N(\cdot)e^{i(\cdot)\cdot\xi'})(x')
	\]
	and for $s=1$,
	\begin{equation}\label{eq:s-1}
		\sigma_a^q(x',\xi') - I_N(x',\xi';1) = -\int_{B_{r_0}(0)} e^{i(\gamma_{\xi'_0}(t',1)-x')\cdot\xi'}\dot\Lambda_q(x',\gamma_{\xi'_0}(t',1))(1-\td{\chi}_N(t',1))J_{\xi'_0}(t',1) \dd t'\,.
	\end{equation}
	Notice that both integrands in \cref{eq:delsIn,eq:s-1} are supported where $t'$ is away from $0$, so that $\dot\Lambda_q$ is holomorphic where it is evaluated. In particular, writing again $\xi' = h^{-1}\tau\xi'_0$,
	\[
		\sigma_a^q(x',h^{-1}\tau\xi'_0) - I_N(x',h^{-1}\tau\xi'_0;0) = \sigma_a^q(x',h^{-1}\tau\xi'_0) - I_N(x',h^{-1}\tau\xi'_0;1)+\int_0^1 \del_s I_N(x',h^{-1}\tau\xi'_0;s) \dd s\,,
	\]
	is holomorphic with respect to $(x',\tau)$ near $(0,1)$. To bound this term, we observe that (estimating $\abs{\td{\chi}_N}\leq e^{Cr_0 N}$ and using its support properties)
	\[
		\abs{\sigma_a^q(x',h^{-1}\tau\xi'_0) - I_N(x',h^{-1}\tau\xi'_0;1)} \leq C^{N+1}e^{\frac{c'}{h}(\abs{(x')^{i\RR}}+\abs{\tau^{i\RR}})}e^{-\frac{c}{h}r_0\tau^{\RR}}
	\]
	for some $C,c,c'>0$ and 
	\[
		\int_0^1 \del_s I_N(x',h^{-1}\tau\xi'_0;s) \dd s \leq e^{\frac{c'}{h}(\abs{(x')^{i\RR}}+\abs{\tau^{i\RR}})}\int_0^1 e^{-\frac{s}{h}r_0\tau^{\RR}}s^{N-1}\dd s\frac{(Cr_0N)^N}{(N-1)!} \leq e^{\frac{c'}{h}(\abs{(x')^{i\RR}}+\abs{\tau^{i\RR}})}(CNh)^N\,,
	\]
	which leads to the desired estimate. 
\end{proof}

Based on this preparation, we shall prove two local results. 
\begin{theorem}\label{thm:localmain}
    Let $q\in L^\infty(M)$ and $\xi'_0 \in S^{n-2}$ be fixed. Let $\chi, \td{\chi}_{j_\vartheta(h)} \in C_c^\infty(W)$ be supported sufficiently close to $0$ and identically $1$ near $0$, where $\supp \chi \subset \{\td{\chi}_{j_\vartheta(h)}\equiv 1\}$ and $\td{\chi}_{j_\vartheta(h)}$ is the $j_\vartheta(h)$-th element of a sequence of Ehrenpreis cut-offs as in \cref{prop:K2}. If 
    \[
        \int_M q \overline{K(\chi e^{\frac{i}{h}\psi_0})} K(\td{\chi}_{j_\vartheta(h)} e^{\frac{i}{h}\psi_1}) d V_g(y) \in \bl{O}^{(-\omega)}_{(0,1)}\,,
    \]
    then $q \equiv 0$ a.e. near $\hat y$.
\end{theorem}
By $\overline{K(\chi e^{\frac{i}{h}\psi_0})}$ we mean the holomorphic conjugate of $K(\chi e^{\frac{i}{h}\psi_0})$ in $(v',\tau)$ when it exists (rather than merely the complex conjugate), which, in particular, is for $(v',\tau)$ near $(0,1)$ by \cref{prop:K1}.

\cref{thm:symb} is now a direct consequence of 
\begin{theorem}\label{thm:localsymb}
    Let $q\in L^\infty(M)$. For any $\xi'_0 \in S^{n-2}$ fixed the following two statements are equivalent
    \begin{enumerate}
	\item $\dot \Lambda_q$ is analytic off $\Delta_W$ near $(0,0)$ and $\sigma_a^q(x',h^{-1}\tau \xi'_0)$ is a cl.a.s. of order $-1$ for $(x',\tau)$ near $(0,1)$
	\item there is some $\td{q} \in C_c^\infty(\RR^n), \td{q}$ analytic near $0$ so that in normal coordinates in $W$ we have $q = \mathbb{1}_{\{y_n\geq 0\}}\td{q}$ near $\hat y=0$. 
    \end{enumerate}
    Furthermore, in this case $\sigma_a^q(x',h^{-1}\tau \xi'_0)$ is elliptic near $(x',\tau)=(0,1)$ if and only if $\td{q}(0) \neq 0$. 
\end{theorem}
Note that \cref{eq:sigma_aissymb} justifies the use of the phrasing `classical analytic pseudodifferential operator' in \cref{thm:symb}, see also \cite[Def.~17.2.32]{MR4436039}. 
\cref{thm:localsymb} is a generalization of \cite[Thm.~1.6]{zbMATH06600504}, but we shall give an alternate proof thereof in \cref{sec:findq}.

In fact, we can already prove one of the implications of \cref{thm:localsymb}:
\begin{proposition}\label{prop:direc1}
	The implication 2.$\implies$ 1. in \cref{thm:localsymb} is true.
\end{proposition}
\begin{proof}
	We follow the proof of \cite[Prop.~1.7]{zbMATH06600504} (in fact, see the proof of \cite[Prop.~1.5]{1312.4065v2}). Let $\chi \in C_c^\infty(W)$ have sufficiently small support and be identically $1$ near the origin. To show that $\dot \Lambda_q$ is analytic off $\Delta_{W}$ near $0$, we must prove (\cite[Thm.~22.3.15]{MR4436039}, \cite[\S~6]{AST_1982__95__R3_0}, \cite[\S~2.6]{zbMATH07060662}) that for all $v = (v_1,v_2), u= (u_1,u_2) \in \RR^{2(n-1)}$ with $v_1, u_1$ near $0$ we have
	\[
		v_1\neq u_1\,, \abs{v_2}+\abs{u_2} = 1\implies \langle f_{v}(x')f_{u}(z'), \dot\Lambda_q(x',z') \rangle \in \bl{O}_{(v_1,u_1)}^{(-\omega)}\,,
	\]
	where
	\[
		f_{v}(x') = \chi(x',0)e^{-\frac{1}{2h}(v_1-x')^2 + \frac{i}{h} x'\cdot v_2}
	\]
	and $f_u$ is defined analogously. In this proof we write $\mu(y) \dd y$ for the measure $\dd V_g(y)$ in the coordinate chart containing $W$, where $\mu$ is some positive analytic function. Notice first that
	\[
		-\langle f_{v}(x')f_{u}(z'), \dot\Lambda_q(x',z') \rangle = \int q \overline{Kf_v} K f_u \dd V_g(y)\,,
	\]
	and if, for some $\eps>0$, we have $\abs{v_2}, \abs{u_2} > \eps$, then by \cref{prop:K1} (with $\tau = 1$ therein) and \cref{rem:K1imp}, we have
	\[
		\int q \overline{Kf_v} K f_u \dd V_g(y) = \int q \chi^2 b^\ast b e^{-\frac{1}{2h}(v_1-y')^2 -\frac{1}{2h}(u_1-y')^2+ \frac{i}{h} y'\cdot (u_2-v_2) + \frac{1}{h}\mathcal{O}(y_n)} \mu(y)\dd y + \bl{O}_{(v_1,u_1)}^{(-\omega)}\,,
	\]
	where the $\mathcal{O}(y_n)$ term has non-positive real part, so that because $v_1\neq u_1$, we obtain the desired exponential decay. 

	Thus, the only case left to treat is when, say, $\abs{u_2}$ is near $0$, in which case we obtain
	\begin{equation}\label{eq:someint}
		\int q \overline{Kf_v} K f_u \dd V_g(y) = \int q \chi b^\ast e^{-\frac{1}{2h}(v_1-y')^2 - \frac{i}{h} y'\cdot v_2 + \frac{1}{h}\mathcal{O}(y_n)} Kf_u \mu\dd y+ \bl{O}_{(v_1,u_1)}^{(-\omega)}\,.
	\end{equation}
    Using local elliptic analytic boundary regularity (\cite{zbMATH03134527}), $Kf_u(y)$ is analytic in $y$ near $0$, and we write $Kf_u(y)$ also for its holomorphic extension to a complex neighborhood of $0$. One can deduce from the proof in \cite{zbMATH03134527} (or see the explicit \cite[Prop.~2.1]{zbMATH06073021}), that for sufficiently small $\delta > \delta' > 0$,
    \[
        \sup_{B_{\delta'}^{\CC^n}(0)}\abs{Kf_u} \leq C_{\delta,\delta'} \left(\sup_{B_{\delta}^{\CC^{n-1}}(0)}\abs{f_u} + \norm{Kf_u}_{H^1(M)}\right)\,.
    \]
    By \cref{prop:GK}, the second term on the RHS above can be be bounded by $\norm{f_u}_{H^{3/2}(\del M)} \lesssim C h^m$ for some $m>0$, and by choosing $\delta>0$ small and estimating $f_u$ directly, we obtain 
    \[
        \sup_{B_{\delta'}^{\CC^n}(0)}\abs{Kf_u} \leq e^{c(\delta^2+\delta\abs{u_2})/h}
    \]
    for some $C,c>0$.

	Introduce a sequence of Ehrenpreis cut-offs $(\chi_N)_{N\in\mathbb{N}} \subset C_c^\infty(\{\chi \equiv 1\})$ and identically $1$ near $0$. Split the integral on the RHS of \cref{eq:someint} into two terms corresponding to $\chi = \chi_N + \chi (1-\chi_N)$. The term containing the cut-offs $\chi (1-\chi_N)$ is supported away from the origin, so that the Gaussian factor $e^{-\frac{1}{2h}(v_1-y')^2}$ induces exponential decay if $v_1$ and $\abs{u_2}$ and $\delta>0$ are sufficiently near $0$ (so that the bound $Ce^{c(\delta^2+\delta\abs{u_2})/h}$ for the absolute value of the holomorphic extension of $Kf_u(y)$ is negligible). Thus, we must only estimate 
	\[
		\int q \chi_N b^\ast e^{-\frac{1}{2h}(v_1-y')^2 - \frac{i}{h} y'\cdot v_2 + \frac{1}{h}\mathcal{O}(y_n)} Kf_u \dd y\,,
	\]
	for which we rely on repeated partial integrations and the fact that $\abs{v_2} > \eps$ for some $\eps>0$ and $q$ is analytic near $0$. After $N$ repeated partial integrations, using the Cauchy inequalities, one obtains a bound of the form $(Ch N/\delta)^N e^{c(\delta^2+\delta\abs{u_2})/h}$, so that assuming $\delta, \abs{u_2}>0$ are sufficiently small and taking $N = N(h) = \lfloor \vartheta h^{-1} \rfloor$ for some $\vartheta>0$ sufficiently small leads to the desired exponential decay. 

	We have completed the proof of the fact that $\dot\Lambda_q$ is analytic off the diagonal near $0$. We shall prove that $\sigma_a^q(x',\xi')$ is a cl.a.s. of order $-1$ with fiber variable $\xi' \in \CC^{n-1}\sem$, for which we shall pass to the semi-classical realm: write $\xi' = h^{-1} \tau \xi'_0, \abs{\xi'_0}=1$, with complex $\tau$ near 1. Again, let $(\chi_j)_{j\in\mathbb{N}}$ be a sequence of Ehrenpreis cut-offs identically $1$ near the origin and choose $j_\vartheta(h)$ according to \cref{prop:K2} so that
	\[
		K(\chi_{j_\vartheta(h)} e^{\frac{i}{h}
        \tau(\cdot)\cdot\xi'_0}) = e^{\frac{i}{h}\phi(\tau,y,\xi'_0)} a(y,h^{-1}\tau\xi'_0) + \bl{O}^{(-\omega)}_{(0,1)}\,,
	\]
	where the glued solutions $a, \phi$ are from \cref{cor:kerndia}.
	Using the analytic WKB method to solve an inhomogeneous equation (\cite[Thm.~2.8.1]{zbMATH07060662}), there is a cl.a.s. $f$ of order $-1$ with $f\vert_{x_n=0}=0$ so that modulo a term in $\bl{O}^{(-\omega)}_{0} \subset \bl{O}^{(0)}_{0}$, for $x$ near $0$,
	\[
		G(q e^{\frac{i}{h}\phi}a)(x) = f(x,h^{-1}\tau\xi'_0) e^{\frac{i}{h}\phi(\tau,x,\xi'_0)}\,.
	\]
	Notice also that by the same argument $G(\bl{O}^{(-\omega)}_{(0,1)}) = \bl{O}^{(-\omega)}_{(0,1)}$ so that differentiating once in $y_n$, multiplying by $e^{- \frac{i}{h} \tau x'\cdot \xi'_0}$ and evaluating at $x_n=0$ gives that $e^{-\frac{i}{h}\tau x'\cdot\xi'_0}\dot\Lambda_q(\chi_{j_\vartheta(h)}(\cdot)e^{\frac{i}{h}\tau (\cdot)\cdot\xi'_0})(x')$ is a cl.a.s. (recall $K = \del_{x_n} G \vert_{x_n=0}$). However, by \cref{prop:anasymb}, this property (recall the `optimal' choice of $j=j_\vartheta(h)$ makes the error in this reference exponentially decaying for $\vartheta>0$ sufficiently small) is true also for $\sigma_a^q$, completing this proof.
\end{proof}

We shall introduce a symbol different from $\sigma_a$ that will allow us to explicitly calculate each term of the formal expansion of $\sigma_a$ if it is a cl.a.s. This new symbol $\sigma_b$ will, however, \emph{not} be the analytic symbol of $\dot\Lambda$, it will only agree with $\sigma_a$ modulo $\mathcal{O}(\abs{\xi}^{-\infty})$. 
\begin{proposition}\label{prop:symbcalc}
    For all $x'$ sufficiently near $\hat y = 0$ in $W$, for $\chi \in C_c^\infty(W)$ identically $1$ near $0$ with sufficiently small support and any $\xi' \in \RR^{n-1}\sem$ we define, for $a, \psi$ from \cref{cor:kerndia} and $q\in L^\infty(M)$, 
    \begin{equation}\label{eq:definesigma}
        \sigma_b^q(x',\xi') \coloneqq (2\pi)^{n-1}\iint e^{i(x'-y')\cdot(\eta'-\xi')}e^{i(\psi(y,\xi')-\psi^\ast(y,\eta'))}a(y,\xi') a^\ast(y,\eta')\dd\eta' \chi(y)q(y) \mu(y)\dd y\,,
    \end{equation}
    where $\mu(y) \dd y$ is the measure $\dd V_g(y)$ in the coordinate chart containing $W$, notation we shall suppress in $\sigma_b^q$ from hereon out (one may subsume $\mu$ into $q$, say). 

	Let $\td{q}\in C_c^\infty(\RR^n)$ be analytic near the origin, $\xi'_0\in S^{n-2}$ fixed. We have
	\[
		\left(\sigma_a^{1_{\{y_n\geq 0\}}\td{q}}- \sigma_b^{1_{\{y_n\geq 0\}}\td{q}}\right)(x',h^{-1}\tau\xi'_0) = \mathcal{O}(h^\infty)
	\]
	uniformly for real $(x',\tau)$ near $(0,1)$.

\end{proposition}
\begin{proof}
    Use \cref{cor:kerndia} to obtain, if $x'$ near $\hat y$ and the support of $\chi$ is sufficiently small,
    \begin{equation}\label{eq:eq}
        \sigma_b^q(x',\xi') = %
        (2\pi)^{n-1}\int \overline{(K(y,x')-K_a(y,x'))}\chi(y)q(y) e^{i(y'-x')\cdot \xi' + i\psi(y,\xi')}a(y,\xi') \dd y\,.
    \end{equation}

	For $(\td{\chi}_j)_{j\in \mathbb{N}}$ an Ehrenpreis cut-off sequence identically $1$ and supported near $0$, define
	\[
		g_{N,\tau,h}(z') = \chi_{N}(z')e^{\frac{i}{h}\tau z'\cdot\xi'_0}\,.%
	\]
	According to \cref{prop:anasymb} we have
	\[
		\sigma_a^{1_{\{y_n\geq 0\}}\td{q}}(x',h^{-1}\tau\xi'_0) = e^{-\frac{i}{h}\tau x'\cdot\xi'_0}K^\ast \td{q}K(g_{N,\tau,h}) + \mathcal{O}_N(h^N)\,.
	\]
	According to \cref{eq:eq}, we have 
	\[
		\sigma_b^{1_{\{y_n\geq 0\}}\td{q}}(x',h^{-1}\tau\xi'_0) = e^{-\frac{i}{h}\tau x'\cdot\xi'_0}(K^\ast-K_a^\ast)(\chi \td{q} K g_{N,\tau,h} + \mathcal{O}_N(h^N))\,.
	\]
	Using the continuity of $K$ on Sobolev spaces (and directly estimating $K_a$) we obtain, for some fixed $m>0$,
	\[
		(\sigma_a - \sigma_b)^{1_{\{y_n\geq 0\}}\td{q}}(x',h^{-1}\tau\xi'_0) = e^{-\frac{i}{h}\tau x'\cdot\xi'_0}(B_1+B_2)g_{N,\tau,h} + \mathcal{O}_N(h^{N-m})\,,
	\]
	where
	\[
		B_1 = K^\ast (1-\chi) \td{q} K\,,\qquad B_2 = K_a^\ast \chi \td{q} K= K^\ast(\overline{K_a(\cdot,x')}\chi(\cdot)\td{q}(\cdot))\,.
	\]
	Since $K_a$ is smooth near the origin, if $\chi$ is supported sufficiently close to the origin, $B_2$ is seen to be a smooth function for each fixed $x'$. Since it also depends smoothly on $x'$, the kernel $B_2$ is smooth in a neighborhood of the origin. That $B_1$ is smooth on a neighborhood of the origin follows from the fact that when $(x',z')$ near the origin, then $y\in \supp(1-\chi)$ is separated from $x',z'$. 

	Finally, by applying repeated partial integrations, one sees that $(B_1+B_2)g_{N,\tau,h} = \mathcal{O}_N(h^{N-m'})$ for some $m'>0$ (independent of $N$), so that in total, taking $N$ sufficiently large the proof is complete.
\end{proof}

\section{Local Injectivity}\label{sec:localinj}

Let us fix $\xi'_0 \in S^{n-2}$ arbitrarily throughout this section. The arguments in this section follow the approach in \cite{zbMATH06600504} but make more explicit choices of phases and transforms considered. Recall the setting of \cref{sec:localsol}, and the functions in \cref{prop:K1,prop:K2}. Motivated by the expression in \cref{cor:Ks}, and that we have access to analytic boundary normal coordinates on the support of $\chi$ therein, define the map
\begin{equation}\label{eq:defofS}
    S \colon L^\infty_{c,+}(\RR^n) \ni u \mapsto \int e^{\frac{i}{h}(-\phi_0^\ast(v',\tau,y)+\phi_1(\tau,y))}b^\ast(v',\tau,y;h) d(\tau,y;h) u(y) \dd y \in \bl{O}_{(0,1)}\,,
\end{equation}
where $L^\infty_{c,+}(\RR^n) = \{v \in L^\infty(\RR^n)\colon \supp\, v \subset \{y_n\geq 0\} \cap W'\ \text{is compact}\}$, where $W'$ is some fixed small open neighborhood of the origin that is sufficiently small for the expressions in the integrand of $S$ to be well-defined. We also write $S_{\xi'_0}$ when we wish to emphasize the dependence on $\xi'_0$, but note that all results of this section hold for any $\xi'_0 \in S^{n-2}$. 

\begin{lemma}\label{lem:phi3props}
    Writing $w = (v',\tau)$, the phase 
    \[
        \phi_2(w,y) \coloneqq -\phi^\ast_0(v',\tau,y) + \phi_1(\tau,y)
    \]
    satisfies at $\hat w = (0,1)$ and $\hat y=0$,
	\begin{enumerate}
		\item $\phi_2(\hat w, \hat y) = 0$, $\del_{w}\phi_2(\hat w, \hat y) =0$, $\del_y \phi_2(\hat w, \hat y) = (0,2ir^{1/2}(0,\xi'_0))^T$,
		\item $\del_{w}^2\phi_2(\hat w, \hat y) = i\id$, $\del_{y'}^2\phi_2(0,1,0) = i\id$,
		\item $\del_y\del_{w}\phi_2(\hat w, \hat y) = -\begin{pmatrix}i\id & \xi'_0 \\ -(\del_{\eta'}r^{1/2}(0,\xi'_0))^T & -ir^{1/2}(0,\xi'_0)\end{pmatrix}$, which is invertible. Starting at the top-left and going clockwise, these blocks are with respect to $\del_{y'}\del_{v'}$, $\del_{y'}\del_{\tau}$, $\del_{y_n}\del_{v'}$, and $\del_{y_n}\del_{\tau}$. 
	\end{enumerate} 
\end{lemma}
\begin{proof}
    The properties of $\phi_2$ all arise from explicit computations using the Cauchy data of $\phi_0, \phi_1$ from \cref{prop:K1,prop:K2} where we also employ the homogeneity of $r^{1/2}$ of degree $1$ in $\eta'$ to obtain point 3., and the fact that $r(0,\xi'_0) = r(0,-\xi'_0)$ to obtain point 1. The Schur complement of $\del_y\del_{w}\phi_2(\hat w, \hat y)$ is $ir^{1/2}(0,\xi'_0)-(\del_{\eta'}r^{1/2}(0,\xi'_0))^T(i\id)(-\xi'_0) = 2ir^{1/2}(0,\xi'_0)$ by homogeneity, which is non-zero, showing invertibility.
\end{proof}

Define the function
\begin{equation}\label{eq:defhatphi}
	\hat \phi(z,y) \coloneqq i\frac{\hat c}{2}y^2 - i\hat cy\cdot z- y \cdot \hat \xi\,,\quad\text{where}\quad\hat \xi = -\begin{pmatrix}0\\ 2ir^{1/2}(0,\xi'_0)\end{pmatrix}\,,%
\end{equation}
and $\hat c>0$ will be chosen sufficiently large throughout the next few results.

For $u\in L^\infty_{c,+}(\RR^n)$ with support sufficiently contracted about $0$ we define, for $z \in \CC^n$ near $0$ and $0< h \leq 1$,
\begin{equation}\label{eq:defofAFBI}
	Au(z;h) \coloneqq \int e^{\frac{i}{h}\hat\phi(z,y)} u(y)\dd y\,,
\end{equation}
which is a germ Fourier-like transform (\cite[\S~20.4]{MR4436039}) with $\hat\phi$ as phase.

\subsection{Obtaining an estimate on $Au$}\label{sec:Au}

In this subsection we will show that $S u$ is, modulo an exponentially decaying term, the image of $Au$ under an elliptic germ FIO transform (\cite[\S~20.5]{MR4436039}) that will be inverted on some suitable function space. To do so we would like to write $Su = SBA u$, where $B$ is some suitable operator and show that $SB$ is invertible. A good candidate for $B$ would be a (germ-space) inverse of $A$, which we expect to be a germ Fourier-like operator with phase $-\hat \phi$, and such $B$ indeed exists (\cite[\S~21.2.5]{MR4436039}). However, in order to guarantee the invertibility of the operator $SB$ we will require a good contour of integration for the phase $\phi_2 - \hat \phi$, which will preclude us from choosing $B$ as this germ Fourier-like inverse of $A$. Nevertheless, this is a technicality that is resolved by part 2. of \cref{lem:contour} and the general method of proof remains. We shall retain $y$ as the variable for $u$, and introduce the new intermediate variable $x\in \CC^n$. 
\begin{lemma}\label{lem:tdTinvnew}
	Defining, with $z = (z',z_n)$, 
	\begin{equation}\label{eq:def_phi4}
		\phi_4(z,w,x)\coloneqq \phi_2(w,x)-\hat\phi(z,x)\,,\quad\text{and}\quad	\Phi^{\sharp}(z)\coloneqq \frac{\hat c}{2}((z')^{\RR})^2\,,%
	\end{equation}
	there is some smooth exponent weight $\Phi^{\sharp,\ast}$ so that $(\phi_4, \Phi^\sharp, \Phi^{\sharp,\ast})$ is a Sj\"ostrand triad (\cite[Def.~20.5.1]{MR4436039}) at $(\hat z, \hat w, \hat x) = (0,(0,1),0)$ and $-\del_z \phi_4(\hat z, \hat w, \hat x) = -2i\del_z\Phi^\sharp(\hat w, \hat x)=0$, and 
	\[
		\det \begin{pmatrix}\del_z\del_w\phi_4 & \del_z\del_x \phi_4 \\ \del_w\del_x\phi_4 & \del_x^2\phi_4
		\end{pmatrix}(\hat z, \hat w, \hat x)\neq 0\,.
	\]
\end{lemma}
\begin{proof}
We begin by showing that 
\[
	\Psi_1(z,x) \coloneqq \Phi^\sharp(z)-\phi_4^{i\RR}(z,\hat w, x)
\]
has a saddle point at $(0,0)$ and that $\phi_4(0,(0,1),0) = 0 = \Phi^\sharp(0)$, the latter part being true by explicit computation.

\textbf{First order:} Note first that
\[
	-\del_x\Psi_1(0,0) = \del_x \phi_4^{i\RR} = \frac{1}{2i}\del_{x^\RR}(\phi_2-\hat\phi) = 0\,,
\]
and $\del_z \Phi^\sharp = 0$ and $\del_z \phi_4^{i\RR}=0$ which gives $\del_z\Psi_1 = 0$. Thus, indeed $\del_{(z,x)} \Psi_1 (0,0) = 0$.

\textbf{Second order:} Notice that %
\begin{equation}\label{eq:hesspsi1}
	M \coloneqq \del_{(z^\RR,x^\RR,z^{i\RR},x^{i\RR})}^2\Psi_1 = \begin{pmatrix}A^{i\RR}+B& A^{\RR} \\ A^{\RR} & -A^{i\RR}\end{pmatrix}\,, 
\end{equation}
where starting from the top-left and going clockwise the blocks are given by $\del^2_{(z^\RR,x^\RR)}$, $\del_{(z^\RR,x^\RR)}\del_{(z^{i\RR},x^{i\RR})}$, $\del_{(z^{i\RR},x^{i\RR})}\del_{(z^\RR,x^\RR)}$, and $\del^2_{(z^{i\RR},x^{i\RR})}$ and
\begin{equation}\label{eq:hesspsi1AB}
	 A \coloneqq -\del_{(z^\RR,x^\RR)}^2\phi_4\,,\quad B \coloneqq \begin{pmatrix}\hat c\id_{(n-1)} & 0 \\ 0&0 \end{pmatrix}\,.
\end{equation}
We calculate that
\[
	A = \begin{pmatrix}0 & -i\hat c\id \\ -i\hat c\id & i\hat c\id-\del_{x^\RR}^2\phi_2\end{pmatrix}\,,
\]
which is invertible and note that
\[
	A^{i\RR}+ B = \begin{pmatrix} \begin{pmatrix}\hat c\id_{(n-1)} & 0 \\ 0&0 \end{pmatrix} & -\hat c\id \\ -\hat c\id & \hat c\id - \del_{x^\RR}^2\phi_2^{i\RR}\end{pmatrix}\,.
\]
Using that $A^{i\RR}$ is invertible by explicit computation, Sylvester's law of inertia, or the Haynsworth inertia additivity formula imply that 
\[
	\mathrm{sign}(M) = \mathrm{sign}(-A^{i\RR}) + \mathrm{sign}(A^{i\RR}+B-A^\RR(-A^{i\RR})^{-1} A^\RR)\,.
\]
Because $A^\RR = \begin{pmatrix}0 & 0 \\ 0 & -\del_{x^\RR}^2\phi_2^{\RR}\end{pmatrix}$, by explicit computation we find that $A^\RR(-A^{i\RR})^{-1} A^\RR = 0$ so that
	\begin{equation}\label{eq:signM}
		\mathrm{sign}(M) = \mathrm{sign}(-A^{i\RR}) + \mathrm{sign}(A^{i\RR}+B)\,.
	\end{equation}
Using Haynsworth's formula,
\[
	\mathrm{sign}(A^{i\RR}+B) = \mathrm{sign}(\hat c\id-\del_{x^\RR}^2\phi^{i\RR}_2) + \mathrm{sign}\left(\begin{pmatrix}\hat c\id_{(n-1)} & 0 \\ 0&0 \end{pmatrix}-\hat c^2 (\hat c \id -\del_{x^\RR}^2\phi_2^{i\RR})^{-1}\right)\,.
\]
First, if $\hat c>0$ is sufficiently large, the first term on the RHS is $(n,0)$. Second, %
using Neumann series, 
\[
	-\hat c^2 (\hat c \id -\del_{x^\RR}^2\phi_2^{i\RR})^{-1}=-\hat c\id - \del_{x^\RR}^2\phi_2^{i\RR} + \mathcal{O}(\hat c^{-1})\,,
\]
so that 
\begin{align*}
	\begin{pmatrix}\hat c\id_{(n-1)} & 0 \\ 0&0 \end{pmatrix}&-\hat c^2 (\hat c \id -\del_{x^\RR}^2\phi_2^{i\RR})^{-1} = \begin{pmatrix}0& 0 \\ 0&-\hat c \end{pmatrix} - \del_{x^\RR}^2\phi_2^{i\RR} + \mathcal{O}(\hat c^{-1}) = \begin{pmatrix}-\id& * \\ * & -\hat c+ * \end{pmatrix} + \mathcal{O}(\hat c^{-1})\,,
\end{align*}
by the explicit formula for $\del_{x^\RR}^2\phi_2^{i\RR}$ calculated in \cref{lem:phi3props}, and where $*$ is a term we do not compute but that is independent of $\hat c$. Using Haynsworth's formula, we find that the signature of the first matrix on the RHS is 
\[
	\mathrm{sign}(-\id) + \mathrm{sign}(-\hat c + *)\,,
\]
which is $(0,n)$ for $\hat c>0$ sufficiently large. In particular, we have for sufficiently large $\hat c>0$ that $\mathrm{sign}(\begin{pmatrix}\hat c\id_{(n-1)} & 0 \\ 0&0 \end{pmatrix}-\hat c^2 (\hat c \id -\del_{x^\RR}^2\phi_2^{i\RR})^{-1}) = (0,n)$, and thus we conclude that $\mathrm{sign}(A^{i\RR}+B)=(n,n)$.

The same type of arguments (but simpler) lead to 
\[
	\mathrm{sign}(A^{i\RR}) = \mathrm{sign}(\hat c\id-\del_{x^\RR}^2\phi_2^{i\RR}) + \mathrm{sign}(- \hat c^2(\hat c\id -\del_{x^\RR}^2\phi_2^{i\RR})^{-1})\,,
\]
where taking $\hat c>0$ large enough yields that $A^{i\RR}$ (and thus also $-A^{i\RR}$) has signature $(n,n)$.

By \cref{eq:signM}, this implies that the matrix in \cref{eq:hesspsi1} has signature $(2n,2n)$ (and thus is non-degenerate) so that we have the desired saddle point of $\Psi_1$. %

\textbf{The `dual' exponent weight:}

Introduce $\Phi^{\sharp,\ast}(w) = (\Phi^\sharp-\phi_4^{i\RR})(z(w),w,x(w))$, where $(z(w),x(w))$ is defined by $\del_{(z,x)}(\Phi^\sharp-\phi_4^{i\RR})(z,w,x) = 0\iff (z,x)=(z(w),x(w))$. Note also that $\Phi^{\sharp,\ast}(\hat w)=0$.

If we show that 
\[
	\Psi_2(w,x) \coloneqq \Phi^{\sharp,\ast}(w)+\phi_4^{i\RR}(\hat z, w, x)
\]
has a saddle point at $(\hat w, \hat x) = ((0,1),0)$, then by our choice of $\Phi^{\sharp,\ast}$ we will have completed the proof of the fact that $(\phi_4,\Phi^\sharp,\Phi^{\sharp,\ast})$ is a Sj\"ostrand triad at $(0,(0,1),0)$. 

\textbf{First order:} By \cref{lem:phi3props},
\[
	\del_{w}\Phi^{\sharp,\ast}(\hat w) = -\del_{w}\phi_4^{i\RR}(\hat z, \hat w, \hat x) = -\del_{w}\phi_2^{i\RR}(\hat w, \hat x) = 0 = \del_w\phi^{i\RR}_4(\hat z,\hat w,\hat x)\,,%
\]
so that $\del_w \Psi_2(\hat w, \hat x) =0$. Furthermore, by explicit calculation $\del_{x}\Psi_2(\hat w,\hat x) = \del_{x}\Psi_1(\hat w, \hat x) = 0$ so that we have the desired critical point at $(\hat w, \hat x)$.

\textbf{Second order:} denote $A_{u,v} \coloneqq \del_{(u^\RR,u^{i\RR})}\del_{(v^\RR,v^{i\RR})}(\Phi^\sharp -\phi_4^{i\RR})(0,(0,1),0)$ for $u,v \in \{w,z,x\}$ and $\td{M} \coloneqq \del^2_{(z^\RR,z^{i\RR},x^{\RR},x^{i\RR})}\Psi_1(0,0)$, which is a reordering of the matrix $M$ in \cref{eq:hesspsi1}:
\[
	\td{M} \coloneqq \begin{pmatrix}M_1 & M_2 \\ M_3 & M_4\end{pmatrix}\,,\quad\text{where}\quad M_4 = \del_{(x^\RR,x^{i\RR})}^2(-\phi^{i\RR}_4)(0,(0,1),0) = \begin{pmatrix}\hat c\id & 0 \\ 0&-\hat c\id\end{pmatrix} + \del_{(x^\RR,x^{i\RR})}^2(-\phi^{i\RR}_2)(0,(0,1),0)\,.
\]
In particular, for sufficiently large $\hat c>0$, the signature of $M_4$ is $(n,n)$ which also implies that it is non-degenerate. Furthermore, we had already computed that $M$ has signature $(2n,2n)$ (and thus $\td{M}$ has signature $(2n,2n)$ by Sylvester's law of inertia) so that the Haynsworth formula tells us that
\begin{equation}\label{eq:signZ}
	\mathrm{sign}(M_1 - M_2M_4^{-1}M_3) = (n,n)\,.
\end{equation}
Furthermore, as a consequence of the implicit function theorem, we have
\[
	\del_{(w^\RR,w^{i\RR})}^2\Phi^{\sharp,\ast}(0,(0,1)) = A_{w,w} - A_{w,(z,x)}\td{M}^{-1}A_{w,(z,x)}^T\,,
\]
and in particular,
\[
	W \coloneqq \del_{(w^{\RR},w^{i\RR},x^{\RR},x^{i\RR})}^2 \Psi_2((0,1),0) = \begin{pmatrix} - A_{w,(z,x)}\td{M}^{-1}A_{w,(z,x)}^T & -A_{w,x} \\ -A_{w,x}^T & -A_{x,x}\end{pmatrix}\,.
\]
Note that $A_{x,x} = M_4$, so that we can use Haynsworth's formula to calculate that 
\begin{equation}\label{eq:signW}
	\mathrm{sign}(W) = \mathrm{sign}(-M_4) + \mathrm{sign}(- A_{w,(z,x)}\td{M}^{-1}A_{w,(z,x)}^T + A_{w,x}M_4^{-1} A_{w,x}^T)\,.
\end{equation}
Defining $Z = (A_{z,z}-A_{z,x}A_{x,x}^{-1}A_{x,z})^{-1} = (M_1 - M_2M_4^{-1}M_3)^{-1}$ (which has signature $(n,n)$ due to \cref{eq:signZ}),
\[
	\td{M}^{-1} = \begin{pmatrix}Z & -ZM_2M_4^{-1} \\ -M_4^{-1}M_3Z & M_4^{-1}+M_4^{-1}M_3ZM_2M_4^{-1}\end{pmatrix}\,,
\]
and by explicit computation we find that for $V \coloneqq A_{w,z}-A_{w,x}A_{x,x}^{-1}A_{x,z}$,
\[
	- A_{w,(z,x)}\td{M}^{-1}A_{w,(z,x)}^T + A_{w,x}M_4^{-1} A_{w,x}^T = -VZV^T\,.
\]
In particular, if $V$ is non-degenerate, then Sylvester's law of inertia says that $\mathrm{sign}(-VZV^T) = \mathrm{sign}(-Z) = (n,n)$. Therefore, showing that $V$ is non-degenerate will imply that $\mathrm{sign}(W) = (2n,2n)$ by \cref{eq:signW}. 

Now notice that $A_{w,z} = 0$ and $A_{x,z} = -\del_{(x^\RR,x^{i\RR})}\del_{(z^\RR,z^{i\RR})}\phi^{i\RR}_4 = \begin{pmatrix}-\hat c \id & 0 \\ 0 & \hat c\id\end{pmatrix}$, which is invertible. In particular, $V$ is non-degenerate if and only if $A_{w,x}$ is non-degenerate, and 
\[
	A_{w,x} = -\begin{pmatrix}\td{B}^{i\RR} & \td{B}^\RR \\ \td{B}^\RR & -\td{B}^{i\RR}\end{pmatrix}\,,\quad\text{with}\quad \td{B} = \del_{w^\RR}\del_{x^\RR}\phi_4 = -\begin{pmatrix}i\id & \xi'_0 \\ -(\del_{\eta'}r^{1/2}(0,\xi'_0))^T & -ir^{1/2}(0,\xi'_0)\end{pmatrix}^T \,,
\]
where the Schur complement test implies that $A_{w,x}$ is non-degenerate. Thus, we have verified that $\Psi_2$ has the desired saddle point. 

In particular, by the definition of $\Phi^{\sharp,\ast}$, we have indeed verified that $(\phi_4, \Phi^\sharp, \Phi^{\sharp,\ast})$ is a Sj\"ostrand triad at $(\hat z, \hat w, \hat x)$.

Finally, we calculate that 
\[
	 \begin{pmatrix}\del_z\del_w\phi_4 & \del_z\del_x \phi_4 \\ \del_w\del_x\phi_4 & \del_x^2\phi_4\end{pmatrix} =  \begin{pmatrix}0 & i\hat c\id \\ \del_w\del_x\phi_4 & \del_x^2\phi_4\end{pmatrix}\,,
\]
which is non-degenerate if and only if
\[
	\del_w\del_x\phi_4 = \del_w\del_x\phi_2
\]
is non-degenerate, which it is. This completes the proof.
\end{proof}

Throughout the rest of this subsection we write $\hat w = (0,1), r_0 \coloneqq r^{1/2}(0,\xi'_0)>0$ (so that $\hat \xi = -2ir_0 e_n$, see \eqref{eq:defhatphi}), and $B_\rho \coloneqq \{x\in \CC^n\colon \abs{x}< \rho\}$ for $\rho>0$. Recall from the proof of \cref{lem:tdTinvnew} the function
\begin{equation}\label{eq:defofpsi1}
    \Psi_1(z,x) = \Phi^\sharp(z) - \phi_4^{i\RR}(z,\hat w, x)\,,
\end{equation}
where $\phi_4$ and $\Phi^{\sharp,\ast}$ were defined in \eqref{eq:def_phi4}.  
As explained at the beginning of \cref{sec:Au}, we would like the operator with phase $\phi_4$ and some amplitude to be a germ FIO (\cite[Def.~20.5.7]{MR4436039}), for which we require a contour of integration in $(z,x)$ that is strongly well-shaped for $\Psi_1$ at $(0,0)$. Define, for $\rho>0$, the contour of integration
\[
    \mathfrak{C}_\rho \coloneqq \{(z,x)\in \CC^{2n} \colon z= \bar x, \abs{x} < \rho\}\,,
\]
which is an open subset of the totally real affine subspace $\{(\bar x, x)\}$ of $\CC^{2n}$ of real dimension $2n$.
\begin{lemma}\label{lem:contour}
    For $\hat c>0$ sufficiently large, there is $\rho_0>0$ so that for all $\rho \in (0,\rho_0]$,
    \begin{enumerate}
        \item The contour $\mathfrak{C}_\rho$ is strongly well-shaped at $(0,0)$ for $\Psi_1$.
        \item Let $\Phi(x) \coloneqq -\phi_2^{i\RR}(\hat w, x)$. There is a constant $c_0\in \CC\sem$ depending only on $\hat c>0$ and the dimension $n$ so that with $c_h \coloneqq c_0 h^n$, there is a compact $K\subset B_\rho$ and complex open neighborhood $U$ of $0$ and $C,\eps >0$ so that for all $g\in \mathcal{O}^{(\Phi)}(B_\rho)$,
        \[
            \int_{\mathfrak{C}_{\rho/2}} e^{\frac{i}{h}(\hat \phi(z,y)-\hat\phi(z,x))} g(x;h) \dd z\dd x - c_h g(y;h) \in \mathcal{O}^{(-\omega)}(U)\,,
        \]
        and the LHS above is bounded on $U$ by $Ce^{-\eps/h} \sup_{x\in K} e^{-\Phi(x)/h}\abs{g(x;h)}$. 
    \end{enumerate}
\end{lemma}
\begin{proof}
    We first verify \eqref{eq:f_neg_def}. By the definitions of $\hat\phi, \Phi^\sharp$ we have
    \begin{equation}\label{eq:contcalc}
        \hat \phi^{i\RR}(\bar x, x) = -\frac{\hat c}{2}\abs{x^\RR}^2 - \frac{3\hat c}{2}\abs{x^{i\RR}}^2+2r_0 x_n^\RR\,, \quad \Phi^\sharp(\bar x) = \frac{\hat c}{2}\abs{(x^\RR)'}^2\,,
    \end{equation}
    and by the properties of $\phi_2$ (\cref{lem:phi3props}) and Taylor's theorem, with $M = \del_y^2 \phi_2(\hat w,0)$,
    \[
        -\phi^{i\RR}_2(\hat w, x) = -2r_0x_n^\RR - \frac{1}{2}\mathrm{Im}(x^TMx)+\mathcal{O}(\abs{x}^3)\,,\quad \mathrm{Im}(x^TMx) \geq \abs{(x^\RR)'}^2 - 2C_M(\abs{(x^\RR)'}(\abs{x^\RR_n}+\abs{x^{i\RR}})+\abs{x_n^\RR}^2+\abs{x^{i\RR}}^2)
    \]
    for some $C_M>0$ depending only on $M$. Thus,
    \[
        \Psi_1(\bar x, x) \leq -\frac{1}{2}\abs{(x^\RR)'}^2 + C_M\abs{(x^\RR)'}(\abs{x^\RR_n}+\abs{x^{i\RR}}) - \left(\frac{\hat c}{2}-C_M\right)(\abs{x_n^\RR}^2+\abs{x^{i\RR}}^2) - \frac{\hat c}{2}\abs{x^{i\RR}}^2 + \mathcal{O}(\abs{x}^3)\,,
    \]
    so that for $\hat c>0$ sufficiently large, the quadratic form (on $x$) on the RHS is negative definite. In particular, there is some $c_\psi>0$ so that $\Psi_1(\bar x, x) \leq -c_\psi\abs{x}^2 = -c_\psi (\abs{\bar x}^2+\abs{x}^2)/2$ on $\mathfrak{C}_{\rho/2}$ for small $\rho>0$. Since $\Psi_1(0,0)=0$ the proof of part 1. is complete.

    We introduce $\zeta = -i\hat c z$ and $E(x;h) \coloneqq e^{-\frac{1}{h}(\frac{\hat c}{2}x\cdot x+2r_0x_n)}$ which is a non-vanishing entire function. By \cref{eq:defhatphi},
    \[
        e^{\frac{i}{h}(\hat \phi(z,y)-\hat\phi(z,x))} = e^{\frac{i}{h}(y-x)\cdot \zeta}E(y;h)E(x;h)^{-1}\,,
    \]
    and after the change of variables $z\mapsto \zeta$,
    \begin{equation}\label{eq:germpsido}
        \int_{\mathfrak{C}_{\rho/2}} e^{\frac{i}{h}(\hat \phi(z,y)-\hat\phi(z,x))} g(x;h) \dd z\dd x = c_1E(y;h)\int_{\td{\mathfrak{C}}_{\rho/2}}e^{\frac{i}{h}(y-x)\cdot \zeta}g(x;h)E(x;h)^{-1}\dd x \dd \zeta\,,
    \end{equation}
    where $c_1$ is a non-vanishing constant depending only on $n$ and $\hat c$ and $\td{\mathfrak{C}}_\rho = \{(x,-i\hat c\bar x)\colon \abs{x}<\rho\}$. The integral on the RHS is \cite[Eq.~(21.1.2)]{MR4436039} with symbol identically $1$, output variable $y$ and input $g E^{-1}\in \mathcal{O}^{(\Phi_E)}(B_\rho)$, where 
    \[
        \Phi_E(x) \coloneqq \Phi(x)+\frac{\hat c}{2}\mathrm{Re}(x\cdot x)+2r_0x_n^\RR\,.
    \]
    Observe that $\Phi_E(0)=0, \del_x\Phi_E(0) = 0$, and the Hessian of $\Phi_E$ is $\hat c \mathrm{diag}(\id, -\id) + \mathcal{O}(1)$ which is non-degenerate of signature $(n,n)$ for $\hat c>0$ large enough. In particular, $0$ is a saddle point of $\Phi_E$ and on $\td{\mathfrak{C}}_\rho$, by \cref{eq:contcalc},
    \[
        \Phi_E(x) - \mathrm{Im}((0-x)\cdot \zeta) = \Phi_E(x) -\hat c \abs{x}^2 = \Psi_1(\bar x, x)- \frac{\hat c}{2}\abs{(x^\RR)'}^2 \leq -c_\psi\abs{x}^2\,,
    \]
    so that $\td{\mathfrak{C}}_\rho$ is a strongly well-shaped contour for $\Phi_E(x) - \mathrm{Im}((0-x)\cdot \zeta)$ at $(0,0)$. This verifies the hypotheses of \cite[Thm.~21.1.2]{MR4436039} so that the integral on the RHS of \cref{eq:germpsido} is a germ pseudo-differential operator on $\mathcal{O}^{(\Phi_E)}_0/\mathcal{O}^{(-\omega)}_0$ with symbol $1$, which is the identity by \cite[Cor.~21.1.8]{MR4436039} (using the independence with respect to the contour \cite[Rem.~21.1.9]{MR4436039}). We obtain the first part of statement 2. 

    The quantitative part of statement 2. is obtained by tracing that \cite[Prop.~20.2.14]{MR4436039} is stated with an error bound $Ce^{-\kappa/h}\sup_{K}e^{-\Phi_E/h}\abs{gE^{-1}(\cdot;h)}$ for some $C,\kappa>0$, and (the proofs of) \cite[Thm.~21.1.7,~Prop.~20.4.3,~Prop.20.4.8]{MR4436039} bound error terms by $C\sup_{K}e^{-\Phi_E/h}\abs{g E^{-1}}$ for some compact $K\Subset B_\rho$ containing all contours involved, uniformly in $y$ near $0$. Because $e^{-\Phi_E/h}\abs{g E^{-1}} = e^{-\Phi/h}\abs{g}$ and $\abs{E(y;h)}\leq e^{C\abs{y}/h}$ for some $C>0$, after shrinking $U$ the proof is complete. (Note that the standing assumption of \cite[\S~21.1.3]{MR4436039} that $\del_x\Phi_E(0) \neq 0$ is not used in the proofs of the results cited here.)
\end{proof} 

We require the quantitative version of statement 2. above only in order to apply that result a family of functions $g$ depending holomorphically on some parameter. 

For $\rho \in (0,\rho_0]$ and $c_h$ as in \cref{lem:contour}, for $f$ holomorphic near $\bar B_\rho$ and $w$ near $\hat w$ define
\begin{equation}\label{eq:germFIO}
    \mathcal{T}_\rho f(w;h) \coloneqq \int_{\mathfrak{C}_{\rho/2}} e^{\frac{i}{h}\phi_4(z,w,x)}c_h^{-1}b^\ast(w,x;h)d(w_n,x;h) f(z;h)\dd z \dd x\,,
\end{equation}
which is the integral in \cite[Eq.~(20.5.7)]{MR4436039} with phase $\phi_4$, contour $\mathfrak{C}_{\rho/2}$, and amplitude $p \coloneqq c_h^{-1}b^\ast d = (2\pi h)^n c_0^{-1}b^\ast d$, which is independent of $z$, and is an elliptic cl.a.s. because its principal part $(2\pi h)^n c_0^{-1}b_0^\ast d_0$ equals $(2\pi h)^{n}c_0^{-1}$ on $\{x_n = 0\}$ (\cref{prop:K1,prop:K2}) and is thus non-zero near $(\hat w,0)$. That products of cl.a.s. are cl.a.s. follows from the results in \cite[\S~19.1]{MR4436039}.

\begin{lemma}\label{lem:replace}
    Let $\hat c>0$ be sufficiently large and $\rho \in (0,\rho_0]$ with $\rho_0$ sufficiently small. 
    \begin{enumerate}
        \item The integral \cref{eq:germFIO} defines a classical elliptic germ FIO transform 
        \[
            \mathcal{T} \colon \bl{O}^{(\Phi^\sharp)}_0/\bl{O}^{(-\omega)}_0 \to \bl{O}^{(\Phi^{\sharp,\ast})}_{\hat w}/\bl{O}^{(-\omega)}_{\hat w}
        \]
        that is invertible on these spaces.
        \item If $u\in L_{c,+}^\infty(\RR^n)$ has support sufficiently contracted about $0$, then for $A$ from \cref{eq:defofAFBI},
        \[
            \mathcal{T}_{\rho}(Au) - Su\in \bl{O}^{(-\omega)}_{\hat w}\,.
        \]
    \end{enumerate}
\end{lemma}
Above we have written the germ FIO transform $\mathcal{T}$ without a subscript $\rho$, because once the transform is well-defined for a sufficiently small $\rho$, as an operator on germ spaces it no longer depends on the contour of integration (in particular $\rho$ itself) modulo exponentially decaying terms, see \cite[\S~20.5]{MR4436039}.
\begin{proof}
    \cref{lem:contour} implies that $\mathcal{T}$ is well-defined (\cite[Thm.~20.5.6,~Def.~20.5.7]{MR4436039}). By \cref{lem:tdTinvnew}, $(\phi_4,\Phi^\sharp, \Phi^{\sharp,\ast})$ is a Sj\"ostrand triad at $(0,\hat w, 0)$, and the same lemma gives the hypotheses of \cite[Thm.~21.2.12]{MR4436039}, which completes the proof of the first part. 

    Define $g_w(x;h)\coloneqq e^{\frac{i}{h}\phi_2(w,x)}b^\ast(w,x;h)d(w_n,x;h)$ which is holomorphic on $B_\rho$ for $w$ near $\hat w$ provided $\rho$ is sufficiently small. Using Fubini's theorem
    \begin{equation}\label{eq:fub}
        \mathcal{T}_\rho (Au)(w;h) = c_h^{-1}\int u(y) \int_{\mathfrak{C}_{\rho/2}}e^{\frac{i}{h}(\hat \phi(z,y)-\hat \phi(z,x))}g_w(x;h)\dd z \dd x \dd y\,.
    \end{equation}
    Because $b^\ast d = \mathcal{O}(h^0)$ near $(\hat w,0)$, we obtain
    \[
        \abs{g_w(x;h)} \leq Ce^{(\Phi(x)+L\abs{w-\hat w})/h}
    \]
    for $x\in B_\rho$ and $w$ near $\hat w$, $\Phi(x) = -\phi^{i\RR}_2(\hat w,x)$ and $L$ a Lipschitz constant of $\phi_2^{i\RR}$ in $w$. We find that \cref{lem:contour} implies that the inner contour integral in the RHS of \cref{eq:fub} evaluates to $c_hg_w(y;h)+R_w(y;h)$ with $\sup_{y\in U}\abs{R_w(y)}\leq Ch^{-n}e^{(-\eps+L\abs{w-\hat w})/h}$. If $\supp u \subset U$, then by \cref{eq:defofS}, 
    \[
        \mathcal{T}_\rho (Au)(w;h) = \int u(y)g_w(y;h)\dd y+ c_h^{-1}\int u(y)R_w(y;h)\dd y = Su(w;h)+\bl{O}^{(-\omega)}_{\hat w}
    \]
    if we choose $\abs{w-\hat w}< \eps/(2L)$, since $\mathcal{T}_\rho (Au)(w;h)$ and $Su(w;h)$ are holomorphic in $w$ near $\hat w$.
\end{proof}

\begin{lemma}\label{lem:phiflat}
	Define 
	\[
		\Phi^\flat(z) \coloneqq \sup_{\RR^n\ni y, y_n\geq 0}-\hat\phi^{i\RR}(z,y) 	
	\]
	which is continuous on $\CC^n$ and satisfies
	\[
			\Phi^\flat(z) = \frac{\hat c}{2}((z')^{\RR})^2 + \begin{cases}\hat c/2(z_n^\RR-2\hat c^{-1}r^{1/2}(0,\xi'_0))^2&\quad\text{if}\ z_n^\RR \geq 2\hat c^{-1}r^{1/2}(0,\xi'_0) \\ 0 &\quad\text{if}\ z_n^\RR < 2\hat c^{-1}r^{1/2}(0,\xi'_0)\end{cases}\,.
	\]
	Furthermore, $\hat\phi$ is an FBI phase function at $(z^\circ_\nu,0), z^\circ_\nu = i\hat c^{-1}\hat\xi + i\nu e_n$, where $e_n$ is the $n$-th standard basis vector and $\nu > 0$ is arbitrary, with $\del_y\hat\phi(z^\circ_\nu,0)=\hat c\nu e_n$.

	In particular, $A \colon L_{c,+}^\infty(\RR^n)\to \mathcal{O}^{(\Phi^\flat)}(\CC^n)$ is continuous (without any germ restriction).
\end{lemma}
\begin{proof}
For $y$ real,
\[
	-\hat\phi^{i\RR}(z,y) = -{\hat c\over 2}y^2 +\hat cz^\RR \cdot y - 2r^{1/2}(0,\xi'_0)y_n\,,
\]
for which one can explicitly see that the supremum in the definition of $\Phi^\flat$ is attained and can be computed explicitly as stated. To verify that $A$ is a germ FBI transform at $(z^\circ_\nu,0)$ we must only check that $\hat \phi$ is an FBI phase function at $(z^\circ_\nu,0)$: 
\[
	\hat\phi(z^\circ_\nu,0) = 0\,, \del_y\hat\phi(z^\circ_\nu,0) = \hat c\nu e_n\,, \det\del_z\del_y\hat\phi(z^\circ_\nu,0) = (-i \hat c)^n \neq 0\,, \mathrm{Im}\del_y\del_y\hat\phi(z^\circ_\nu,0) = \hat c\id > 0\,.
\]
A minor modification of \cite[Prop.~22.2.7]{MR4436039} to reflect the additional assumption that if an input function is supported in $\{y_n\geq 0\}$, we may define $\Phi^\flat$ by taking the supremum over $y_n\geq 0$, the proof is complete.
\end{proof}

Combining the previous two results, 
\begin{lemma}\label{lem:fbifirst}
Let $\hat c>0$ be sufficiently large. If $u\in L_{c,+}^\infty(\RR^n)$ satisfies $Su\in\bl{O}^{(-\omega)}_{(0,1)}$ and has support sufficiently contracted about $0$, then $Au \in \bl{O}^{(-\omega)}_0$.
\end{lemma}
\begin{proof}
Let $\rho\in (0,\rho_0]$ with $\rho_0>0$ sufficiently small from \cref{lem:contour}. Near $0$ we have $\Phi^\sharp = \Phi^\flat$ (defined in \cref{lem:tdTinvnew}, \cref{lem:phiflat} respectively), and \cref{lem:phiflat} then gives $Au \in \mathcal{O}^{(\Phi^\sharp)}(B_\rho)$ for $\rho< \rho_0$ sufficiently small. By \cref{lem:replace} and the hypothesis on $Su$ we have $\mathcal{T}_{\rho}(Au) \in \bl{O}_{\hat w}^{(-\omega)} \subset \bl{O}_{\hat w}^{(\Phi^{\sharp,\ast})}$ (using $\Phi^{\sharp,\ast}((0,1))=0$,  \cite[Prop.~20.2.13]{MR4436039}), so that $\mathcal{T}(Au) = 0$ by \cref{lem:replace}. Inverting $\mathcal{T}$, we find that $Au = 0$ in $\bl{O}_{0}^{(\Phi^\sharp)}/\bl{O}_{0}^{(-\omega)}$, which is to say that $Au \in \bl{O}_{0}^{(-\omega)}$. 
\end{proof}

\subsection{Proof of \cref{thm:localmain}}

There is only one more obstacle we must overcome: $A$ is not a germ FBI transform (\cite[\S~22]{MR4436039}) at $(0,0)$ (only a germ Fourier-like operator), because the phase $\hat\phi$ satisfies $\del_y\hat\phi(0,0) = -\hat\xi$ which has an imaginary part. However, at $z^\circ_\nu = i\hat c^{-1}\hat\xi +i\nu e_n, y=0$ we have seen in \cref{lem:phiflat} that it is indeed a germ FBI transform. 
\begin{lemma}\label{lem:complexshift}
	If $u\in L^\infty_{c,+}(\RR^n)$ with $Su\in\bl{O}^{(-\omega)}_{(0,1)}$ has support sufficiently contracted about $0$, then for $\nu>0$ sufficiently small,
	\[
		Au = 0 \quad\text{in}\ \bl{O}^{(\Phi^\flat)}_{z^\circ_\nu}/\bl{O}^{(-\omega)}_{z^\circ_\nu}\,.
	\]
	\end{lemma}
\begin{proof}
For $\nu >0$ sufficiently small (as will be chosen throughout this proof) writing $z^\circ = z^\circ_\nu$, let us define
\[
	\Phi^{\flat,\circ}(z)\coloneqq \Phi^\flat(z+z^\circ) = \frac{\hat c}{2}((z')^{\RR})^2  + \begin{cases}\frac{\hat c}{2}(z_n^\RR)^2 \quad&\text{if}\ z_n^\RR\geq 0\\ 0\quad&\text{else}\end{cases}\,,
\]
where we used that $\hat \xi= -(0, 2ir^{1/2}(0,\xi'_0))$. Our assumptions and \cref{lem:fbifirst} imply
\[
	A^\circ u(z) \coloneqq Au(\cdot + z^\circ) \in \mathcal{O}^{(\Phi^{\flat,\circ})}(\CC^n)\,,\quad\text{and}\quad A^\circ u(z)\in \bl{O}^{(-\omega)}_{-z^\circ}\,,
\]
and the proof is complete if we show that $A^\circ u \in \bl{O}^{(-\omega)}_{0}$. We give $z_n$ a distinguished role and thus define, for some $\eps^\circ >0$ sufficiently small and some sufficiently large constant $C>0$ independent of $h$,
\[
	F(z', z_n;h) \coloneqq e^{-\eps^\circ/h}C^{-1}A^\circ u((z',z_n);h)\,,
\]
which is an entire function of $z_n$, holomorphic in $z'$ near $0$. Due to the explicit nature of $A^\circ$, that $u$ is compactly supported and the phase in $A^\circ$ is quadratic, for all $z'$ near $0 = (z^\circ)'$ we have the bound
\[
	e^{-{\Phi^{\flat,\circ}}(z',s)\over h}\abs{F(z',s;h)}\leq \begin{cases}1 \quad\text{for all}\ s\in \CC\\ e^{-\eps h^{-1}} \quad\text{for}\ \abs{s+z_n^\circ}\leq \delta' \end{cases}
\]
for some $\eps,\delta>0$ independent of $\eps^\circ$ if $\eps^\circ>0$ is sufficiently small. Note that if $\nu>0$ is chosen sufficiently small, then the circle $\abs{s+z_n^\circ} = \delta$ intersects the negative half of the real line $s^{i\RR}=0, s^{\RR}<0$ in two places, so that there are some $s^\circ,\delta > 0$ so that, for $z'$ sufficiently near $0$,
\[
	e^{-{\Phi^{\flat,\circ}}(z',s)\over h}\abs{F(z',s;h)}\leq \begin{cases}1 \quad\text{for all}\ s\in \CC\\ e^{-\eps h^{-1}} \quad\text{for}\ \abs{s+s^\circ}\leq \delta \end{cases}
\]

We will apply (the proof of) \cite[Lem.~4.1]{zbMATH05707695} to complete the proof. Define 
\[
	g(z',s) \coloneqq h\log \abs{F(z',s;h)}%
\]
which satisfies
\[
	g(z',s)-\frac{\hat c}{2}((z')^\RR)^2 \leq \begin{cases}{\hat c\over 2}(s^\RR)^2 \quad&\text{for}\ s^\RR\geq 0 \\ 0 \quad&\text{for}\ s^{\RR}\leq 0 \\ -\eps\quad&\text{for}\ \abs{s+ s^\circ}\leq \delta\end{cases}\,.
\]
Notice that the function $f(z',s)\coloneqq {2\over \hat c}g(z',-s)-((z')^\RR)^2$ has the same properties as the function given in the proof of \cite[Lem.~4.1]{zbMATH05707695} that is also denoted by $f$, uniformly for $z'$ near $0$. Following the proof therein, one finds that $f(z',s) \leq -\eps'$, for some $\eps'>0$ as long as $\abs{s^\RR}\leq \theta_\RR$ and $\abs{s^{i\RR}}\leq \theta_{i\RR}$ for some $\theta_\RR,\theta_{i\RR}>0$. 
In particular, $g(z',s) \leq -{\hat c\eps'\over 2} + \frac{\hat c}{2}((z')^\RR)^2$ and as in the proof of \cite[Lem.~4.1]{zbMATH05707695}, for $z'$ sufficiently near $0$ we deduce that
\[
	\abs{F(z',s;h)} = e^{\frac{g(z',s)}{h}}\leq e^{-\frac{\eps''}{h}}%
\]
for some $\eps''>0$ and for $s$ with $\abs{s^\RR}\leq \theta_\RR, \abs{s^{i\RR}}\leq \theta_{i\RR}$.

In particular, if we choose $\eps^\circ \in (0,\eps'')$ sufficiently small, then we have shown that $A^\circ u =0$ in $\bl{O}^{(\Phi^{\flat,\circ})}_{0}/\bl{O}^{(-\omega)}_{0}$, which completes the proof.
\end{proof}

We can now state
\begin{corollary}\label{cor:vanish}
    If $u\in L^\infty_{c,+}(\RR^n)$ satisfies $S u \in \bl{O}^{(-\omega)}_{(0,1)}$ and has support sufficiently contracted about $0$, then $u$ vanishes near $0$.
\end{corollary}
\begin{proof}
    Let $S u \in \bl{O}^{(-\omega)}_{(0,1)}$ for some $u\in L^\infty_{c,+}(\RR^n)$. Due to \cref{lem:complexshift,lem:phiflat} and \cite[Thm.~22.3.15]{MR4436039} we have $(0,-\nu \hat c e_n)\not\in \wf(u)$. 

	Now since $\supp u \subset \{y_n \geq 0\}$, %
	microlocal analytic continuation \cite[Thm.~8.5.6']{hoermander1} then shows that $u$ vanishes near the origin. 
\end{proof}

\begin{proof}[Proof of \cref{thm:localmain}]
    According to \cref{cor:Ks}, the proof is a direct consequence of \cref{cor:vanish}.
\end{proof}

\subsection{Proofs of \cref{thm:harmonic,thm:lin}}

\begin{proof}[Proof of \cref{thm:lin}]
	Because $\dot\Lambda_q^\Sigma$ is linear in $q$, we must only show that  if $\dot\Lambda_q^\Sigma$ vanishes near $\hat y$ for some $q\in L^\infty(M)$, then $q=0$ a.e. near $\hat y$. Because $-\dot\Lambda_q = K^\ast q K$, this is a direct consequence of \cref{thm:localmain}. Finally, the last statement in the theorem follows from the analytic identity theorem.
\end{proof}

Having shown \cref{thm:localmain}, we give the
\begin{proof}[Proof of \cref{thm:harmonic}]
    We remark first that since we assumed that $P$ does not have Dirichlet eigenvalue $0$ on $M$, it does not have Dirichlet eigenvalue $0$ on each connected component of $M$. Thus, in order to show that $q = 0$ a.e. in $M$, it suffices to assume that $M$ is connected, because one may use the same argument for each connected component of $M$. We thus assume that $M$ is connected.

    We shall follow \cite{zbMATH05707695} and \cite{feizmohammadi2024partialdatainverseproblems}. Let $y_1 \in M^{\mathrm{int}}$ be fixed arbitrarily and choose $y_0 \in \partial M$ so that $\partial M$ is analytic near $y_0$ and there is $\gamma \colon [0,1] \to M$ smooth with $\gamma(0)=y_0, \gamma(1) = y_1$ and $\gamma((0,1]) \subset M^{\mathrm{int}}$, and $\gamma$ is non-self-intersecting. %
    Introduce analytic boundary normal coordinates in some open neighborhood $W$ of $y_0$ in $M$. We may assume that $W\cap \del M\subset 
    \Sigma$ for $\Sigma\ni y_0$ from the statement of \cref{thm:harmonic}.
    
    Define $\Gamma_\eps(s) \coloneqq \{y \in M \colon d_g(\gamma([0,s]),y) \leq \eps\}$ for $s\in [0,1]$ and $\eps>0$. According to \cref{thm:localmain}, if $\delta>0$ is sufficiently small we know that $\mathrm{ess}\ \supp q \cap B^M_{\delta}(y_0) = \emptyset$. In particular, if $\eps>0$ is sufficiently small and we define the closed set
    \[
	    J = \{t \in [0,1]\colon q = 0 \ \text{a.e. on}\  \Gamma_\eps(t)\}\,,
    \]
    then there is some $s_0 \in J$ with $s_0>0$ and $\Gamma_\eps(1) \cap \partial M \subset B_\delta^M(y_0)$ (where we use that $\gamma((0,1])\subset M^{\mathrm{int}}$). We assume also that $\eps>0$ is smaller than half the injectivity radius of $M$. We keep this choice of $\eps>0$ fixed from now onward. We conclude that if $\eps_0>0$ is sufficiently small, then 
    \[
        B_{\eps_0}^M(\{y \in M\colon d_g(\gamma([s_0,1]),y) \leq \eps\}) \cap \partial M = \emptyset\,,
    \]
    and fix this $\eps_0>0$. 

    Note that the set $J$ defined above is non-empty and closed: the only non-trivial argument to prove this occurring when $(t_j)_{j\in\mathbb{N}}\subset J$ is an increasing sequence converging to some $t \in (0,1]$ and we wish to show $t\in J$. In this case, $\Gamma_\eps(t)\setminus \bigcup_{j\geq 1}\Gamma_{\eps}(t_j) \subset \del B_{\eps}(\gamma(t))$, which has volume zero and thus $q=0$ a.e. there, proving that $t\in J$, and thus $J$ is closed. 
    
    Showing that $J$ is open implies that $\mathrm{ess}\ \supp q$ is empty in a neighborhood of $y_1$. Since $y_1 \in M^{\mathrm{int}}$ was chosen arbitrarily, the proof would then be complete.

    Assume that $s \in J$. If $s< s_0$, then by our definition of $J$ we know that $[0,s_0] \subset J$ so that an open neighborhood of $s$ lies in $J$. We may thus assume that $s\geq s_0$. Also fix some $\eps' \in (\eps/2,\eps)$.

    It is possible to define a compact manifold $M' \subset M$ so that 
    \begin{enumerate}
        \item $M'$ and $M \setminus M'$ are connected, $q$ vanishes a.e. on $M\setminus M'$, 
        \item $\del M'$ is smooth,
        \item $P$ does not have Dirichlet eigenvalue $0$ on $M'$,
        \item $\del M' \supset (\del B_\eps^M(\gamma(s))\setminus \Gamma_{\eps'}(s))$,
        \item and the boundary of $M'$ is analytic in $\del B_\eps^M(\gamma(s))\setminus \Gamma_{\eps'}(s)$.
    \end{enumerate}
    \begin{figure}[h]
    \begin{tikzpicture}
  \coordinate (A) at (0.00, 0.00);
  \coordinate (B) at (0.00, 0.00);
  \coordinate (C) at (-4.00, 0.00);
  \coordinate (D) at (1.50, 1.00);
  \coordinate (E) at (-0.97, 0.55);
  \coordinate (F) at (-0.97, 0.55);
  \coordinate (G) at (-1.00, 1.00);
  \coordinate (H) at (-3.95, 0.57);
  \coordinate (I) at (-1.05, 1.10);
  \coordinate (J) at (-3.95, -0.63);
  \coordinate (K) at (-0.90, 0.00);
  \coordinate (L) at (-3.99, 0.3);
  \coordinate (M) at (-1.02, 0.86);
  \coordinate (N) at (-3.96, -0.37);
  \coordinate (O) at (-0.89, 0.25);
  \coordinate (P) at (-2.25, 0.00);
  \coordinate (Q) at (-1.40, 0.17);
  \coordinate (R) at (-1.50, 0.75);
  \coordinate (S) at (-2.75, 0.50);
  \coordinate (T) at (-1.6, -0.6);
  \coordinate (U) at (-0.25, -0.25);
  \coordinate (V) at (-0.74, 1.17);

  \draw[thick] (A) circle (4.00);
  \draw (C) -- (D);
  \draw[pattern color=black, pattern=north east lines] (E) circle (0.55);
  \draw[dashed] (H) -- (I);
  \draw[dashed] (J) -- (K);
  \draw[thick, dotted] (-0.89, 0.25) arc (-75.0:99.0:0.31);
  \draw[thick] (-3.8,-1.2) arc (-80:82:1.2);
  \draw[thick, dotted] (L) -- (M);
  \node at (-3.34,1.34) {$W$};
  \draw[thick, dotted] (N) -- (O);
  \fill[pattern color=black, pattern=north east lines] (L) -- (C) -- (N) -- (P) -- (Q) -- (R) -- (S) -- cycle;
  \node at (T) {$\Gamma_\eps(\gamma([0,s]))$};
  \fill[black] (C) circle (4pt);
  \fill[black] (D) circle (4pt);
  \fill[black] (E) circle (4pt);
  \fill[black] (H) circle (0pt);
  \fill[black] (I) circle (0pt);
  \fill[black] (J) circle (0pt);
  \fill[black] (K) circle (0pt);
  \fill[black] (L) circle (0pt);
  \fill[black] (M) circle (0pt);
  \node[black, very thick] at (-1.37,0.16) {$\bl{\times}$};
  \node[black, very thick] at (-1.5,0.75) {$\bl{\times}$};
  \fill[black] (N) circle (0pt);
  \fill[black] (O) circle (0pt);
  \fill[black] (C) circle (4pt);
  \node[above left, black] at (C) {$y_0$};
  \fill[black] (D) circle (4pt);
  \node[above right, black] at (D) {$y_1$};
  \node[above, black] at (V) {$B_\eps(\gamma(s))$};
\end{tikzpicture}%
\caption{The manifold $M$ is represented by the interior of the big circle, and $W$ by the interior of the semicircle about $y_0$. Everything outside of the dashed area is $M_0$, the region inside the dotted line is $\Gamma_{\eps/2}(\gamma([0,s]))$. The two thick $\bl{\times}$ marks denote the two locations at which we must smooth out the boundary of $M_0$, which can be done entirely within the outer `tube' $\Gamma_{\eps}(\gamma([0,s]))$, on which we already know that $q$ vanishes a.e.}
\label{fig:fig}
\end{figure}
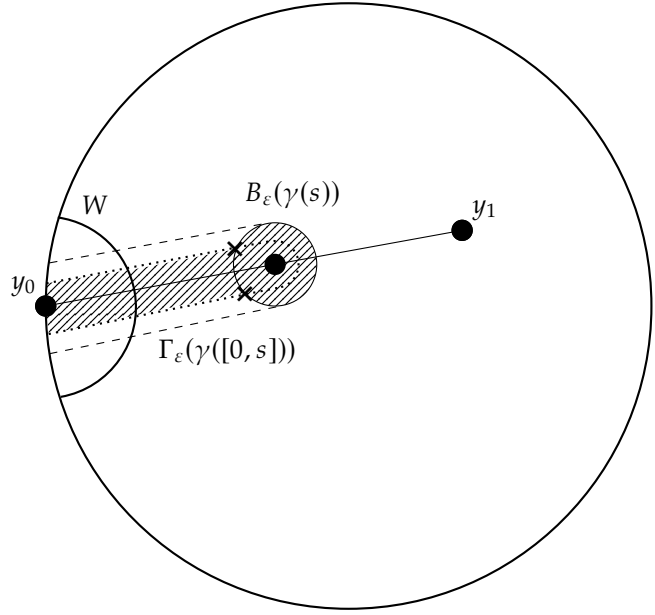
    To construct $M'$, merely start with 
    \[
        M_0 = M\setminus \left(B_{\eps}^M(\gamma(s))\cup B^M_{\eps/2}(\gamma([0,s]))\right)\,,
    \]
    which satisfies properties 1., 4. and 5, see \cref{fig:fig}. By smoothing out the boundary of $M_0$ near $\del B_\eps^M(\gamma(s))\cap \Gamma_{\eps/2}(s)$ within $\Gamma_{\eps'}(s)\setminus \Gamma_{\eps/2}(s)$ and otherwise slightly modifying $M_0$ away from $B_\eps^M(\gamma(s))$ (see \cref{cor:perturbDEV}), making sure these modifications occur only where we know $q=0$ a.e., one obtains $M'$ with the desired properties. 

    Recall that the proof is complete if we can show that every $x \in \del B_\eps^M(\gamma(s))$ has an open neighborhood on which $q$ vanishes. When $x \in \del B_\eps^M(\gamma(s))\cap \Gamma_{\eps'}(s)$, since $\eps'<\eps$ we have $B_{\eps-\eps'}(x)\subset \Gamma_\eps(s)$ and $q=0$ a.e. there already. 

    We thus assume that $x\in \del B_\eps^M(\gamma(s))\setminus \Gamma_{\eps/2}(s) \subset \del M'$ where we used property 4. of $M'$. As a consequence of Runge approximation (\cref{thm:runge} applied to $L=-P$, wherein we choose $\Xi = \emptyset$ which allows taking $\Gamma \subset \Sigma$ with $\del M\setminus \bar M'\subset \Gamma\subset B_\delta^M(y_0)\cap \del M$ for $\delta>0$ small), if $K_{M'}$ is the Poisson operator for $P$ on $M'$, for any $f\in H^{r+3/2}(\del M')$ with $\supp f\vert_{\del M\cap M'} = \emptyset$ there is a sequence $(f_j)_{j\in\mathbb{N}} \subset H^{r+3/2}(\del M)$ with $\supp f_j \subset \Sigma$ so that $K(f_j)\vert_{M'} \to K_{M'}(f)$ in $H^{r+2}$ for all $r\geq 0$. 

    Introduce analytic boundary normal coordinates near $x$ in $M'$. Performing the constructions of \cref{prop:K1,prop:K2} on $M'$ in these coordinates, and using that $q = 0$ on $M\setminus M'$, we have 
    \[
        \int_{M'} q \overline{K_{M'}(\chi e^{\frac{i}{h}\psi_0})} K_{M'}(\td{\chi}_{j_\vartheta(h)} e^{\frac{i}{h}\psi_1}) \dd V_g(y) = \lim_{j,k \to \infty}\int_M q \overline{K(f_{1,j})} K(f_{2,k}) \dd V_g(y) = 0\,,
    \]
    where $(f_{1,j})_{j\in\mathbb{N}},(f_{2,k})_{k\in\mathbb{N}} \subset C^\infty_c(\Sigma)$. Invoking the local \cref{thm:localmain} with $M$ replaced by $M'$, we find that $q$ vanishes near $x$. By compactness, we have shown that there is some $\eps_0>0$ so that $q$ vanishes on $B_{\eps+\eps_0}(\gamma(s))$, which implies that the set $J$ is open. The proof is complete.
\end{proof}

\section{Local recovery of a symbol}\label{sec:findq} 

We will prove the `only if' direction of \cref{thm:localsymb} by finding a function $\td{q}$ that is analytic near $0$ for which $\sigma_b^{\td{q}}$ is (almost) a cl.a.s. By asymptotically expanding $\sigma_b^{\td{q}}$ in powers of $h$, by \cref{prop:symbcalc} we then have access to each term in the formal expansion of the cl.a.s. $\sigma_a^{\td{q}}$, and if $\td{q}$ is chosen correctly these terms will match up with some given cl.a.s. $\sigma_a^{q}$. We will show that $\sigma_b^{\td{q}}$ determines (in a triangular fashion) the full Taylor jet of $\td{q}$ at the boundary, which will facilitate choosing $\td{q}$ as desired. 

We let $a, \psi$ be from \cref{prop:symbcalc} and consider the operator $T$ with kernel
\begin{equation}\label{eq:defofT}
	T(x',y,\xi') \coloneqq (2\pi)^{n-1}\int e^{i(x'-y')\cdot(\eta'-\xi')}e^{i(\psi(y,\xi')-\psi^\ast(y,\eta'))}a(y,\xi') a^\ast(y,\eta')\dd\eta'\,,
\end{equation}
where we consider $\xi' \in \RR^{n-1}\sem$ a parameter. 

Define by $\eps_0 \in (0,1/2)$ the constant so that 
\begin{equation}\label{eq:giveeps0}
	j\geq 0 \implies a_j,\psi \in \mathcal{O}(\{(y,\eta') \in \CC^n\times \CC^{n-1}\sem\colon \abs{y}\leq 2\eps_0, \abs{y_n} \leq \eps_0, \eta' \in \Theta_{\eps_0}\})\,.
\end{equation}

\begin{proposition}\label{prop:allthingst}
	Let $\xi'_0\in S^{n-2}$ be arbitrary and fixed. There is an open complex neighborhood $\hat U$ of $(0,1)\times \CC^n$ in $\CC^{n-1}\times \CC\times \CC^n$ and a formal cl.a.s $\ud{t}(y',\tau,\eta;h) = h\sum_{k\geq 0}h^k t_{k}(y',\tau,\eta)$ with $(t_k)_{k\in\mathbb{N}} \subset \mathcal{O}(\hat U)$ (and fiber variable $\tau$), and a continuous function $t(y',\tau,\eta;h)$ so that for all complex $\tau$ sufficiently near $1$ and for every sufficiently small compact complex neighborhood $K$ of $0$ there is $C>0$ so that for all $C_0>0$ sufficiently large, all $N\in\mathbb{N}$, and all $h > 0$ so that $C_0hN \leq 1/2-\eps_0$, we have
	\begin{equation}\label{eq:tpolybound}
		\sup_{y'\in K}\abs{h^{-1}t- \sum^{N-1}_{k=0} h^k t_k} \leq C^{N+1} h^N \sum_{j=0}^Nj!\abs{\eta}^{N-j}\,.
	\end{equation}
	The function $t_0$ is non-vanishing and $t_{k}(y',\tau,\eta)$ is a polynomial of degree $k$ in $\eta$ with holomorphic coefficients (in $y',\tau$). Writing $t_{k,j}$ for the $j$-th order homogeneous part of $t_k$ ($j\leq k$), we have
	\begin{equation}\label{eq:boundsont_k}
		\abs{t_{k,j}}\leq C^{k+1} (k-j)! \abs{\eta}^j\,,\qquad \abs{t_k} \leq C^{k+1} k! \abs{\eta}^k_+\,.
	\end{equation}

	Finally, for any $q \in C_c^\infty(\RR^n)$ having sufficiently small support near the origin we have
	\begin{equation}\label{eq:sigmaist}
		\tau \in \RR\ \text{near}\ 1 \implies T(\mathbb{1}_{\{y_n\geq 0\}}q)(x',h^{-1}\tau\xi'_0) = \iint e^{i(x'-y)\cdot\eta} t(y',\tau,\eta;h) q(y)\dd\eta \dd y
	\end{equation}
	for each $h>0$ small enough, where
	\[
		 \iint e^{i(x'-y)\cdot\eta} t(y',\tau,\eta;h) q(y)\dd\eta \dd y = \iint e^{i(x'-y)\cdot\eta} (1+\abs{\eta}^2)^{-k}(1+\Delta_y)^{k}\left(t(y',\tau,\eta;h) q(y)\right)\dd\eta \dd y
	\]
	is defined as an oscillatory integral for $k\in \mathbb{N}$ sufficiently large. Above, we wrote $(x'-y)$ for $((x',0)-y)$.
\end{proposition}

Before turning to the proof of \cref{prop:allthingst}, we shall state some intermediate results.
\begin{lemma}\label{lem:bproperties}
	Let $\xi'_0\in S^{n-2}$ be arbitrary and fixed. For $a$ from \cref{eq:symbola}, define for $\tau \in \RR$ near $1$
	\[
		b(y',z_n,\tau,\eta';h) \coloneqq (2\pi)^{n-1}a(y',h z_n,h^{-1}\tau\xi'_0)a^\ast(y',hz_n,h^{-1}(h\eta'+\tau\xi'_0))\,.
	\]
	and for integers $k\geq 0$ and $\tau$ complex near $1$,
	\[
		b_k(y',z_n,\tau,\eta') \coloneqq (2\pi)^{n-1}\sum_{s+l+j+m=k} \frac{\del^s_h}{s!}\left(a^\ast_{j}(y',hz_n,h\eta'+\tau\xi'_0)\right)\frac{\del^m_h}{m!}\left(a_{l}(y',hz_n,\tau\xi'_0)\right)\Big\vert_{h=0}\,.
	\]
	For $k \geq 1$, $b_k$ is a polynomial of degree $k-1$ in $\eta'$ and a polynomial of degree $k$ in $(\eta',z_n)$ with analytic coefficients and $b_0=(2\pi)^{1-n}$, and each $b_k$ depends holomorphically on $\tau$ near $1$. Furthermore, for every compact $K$ there is $C>0$ so that for every $k\geq 0$ we have
	\begin{equation}\label{eq:boundonb_k}
		\sup_{y'\in K}\abs{b_k} \leq C^{k+1} \sum_{j+m\leq k} j! \abs{\eta'}^{k-j-m} \abs{z_n}_+^m \leq C^{k+1}k! \abs{z_n}_+^k\abs{\eta'}^k_+\,.
	\end{equation}

	Furthermore, for every sufficiently large $C_0>0$ there is $C>0$ so that for every $N\in\mathbb{N}$ if $\abs{hz_n}\leq \eps_0$ and $C_0hN \leq 1/2-\eps_0$, we have for $\tau\in \RR$ near $1$,
	\begin{equation}\label{eq:diffboundbb_k}
 		\sup_{y'\in K} \abs{b(y',z_n,\tau,\eta';h)- \sum_{k=0}^{N-1} h^k b_k(y',z_n,\tau,\eta')} \leq C^{N+1} h^N\sum_{s+l=0}^N\abs{z_n}_+^l s! \abs{\eta'}_+^{N-s-l}\,.
	\end{equation}
\end{lemma}
In the statement above, as well as throughout this document, $\eps_0$ always refers to \cref{eq:giveeps0}.
\begin{proof}
	Recall the notation in \cref{eq:symbola}. For any $j\in\mathbb{N}$, by direct computation,
	\begin{equation}\label{eq:ajtaylork}
		\del_h^k a_j^\ast(y',hz_n,h\eta'+\tau\xi'_0)\vert_{h=0} = \sum_{l=0}^k {k \choose l} z_n^l (\del_{y_n}^l (\eta'\cdot \nabla_{\xi'})^{k-l}a_j^\ast)(y',0,\tau\xi'_0)\,.
	\end{equation}
	Recalling the second part of \cref{eq:symbola}, we have
	\[
		((\eta'\cdot \nabla_{\xi'})^{k}a_j^\ast)(y',0,\tau\xi'_0) = \begin{cases} (2\pi)^{1-n}\quad\text{if}\ \ j = k = 0 \\ 0\quad\text{else}\end{cases}\,,\qquad a_j^\ast(y',0,\tau\xi'_0) = (2\pi)^{1-n}\delta_{0j}\,.
	\]
	In particular, \cref{eq:ajtaylork} is a polynomial of degree $\leq k-1$ in $\eta'$ if $k \geq 1$ and identically $(2\pi)^{1-n}$ if $j=k=0$.

	On the other hand,
	\begin{equation}\label{eq:otherajtaylor}
		\del_h^k(a_j(y',hz_n,\tau\xi'_0))\vert_{h=0} = z_n^k (\del_{y_n}^ka_j)(y',0,\tau\xi'_0)
	\end{equation}
	is a polynomial of degree $0$ in $\eta'$ for all $j,k\geq 0$ and identically $1$ if $j=k=0$. Thus, by inspecting the explicit formula of $b_k$ we have indeed found that $b_0 = (2\pi)^{-n+1}$ and $b_k$ is a polynomial of degree $\leq k-1$ in $\eta'$ for $k \geq 1$. The holomorphic dependence of $b_k$ on $\tau$ comes from its explicit form and the fact that each $a_j, a_j^\ast$ is holomorphic in the $\eta'$ variable away from the origin. 

	As a direct consequence of Cauchy's inequalities, the bound $\abs{a_j^\ast(y',0,\tau\xi'_0)} \leq C^j j!$, and the explicit formula \cref{eq:ajtaylork} we have (for every compact $K$)
	\begin{equation}\label{eq:easyaastbound}
		\sup_{y'\in K}\abs{\del_h^k a_j^\ast(y',hz_n,h\eta'+\tau\xi'_0)\vert_{h=0}} \leq C^{k+j+1} j! k! \sum_{s=0}^k \abs{z_n}^s\abs{\eta'}^{k-s}\,.
	\end{equation}
	for some $C>0$ independent of $j,k$. A similar bound holds for \cref{eq:otherajtaylor} which leads to
	\[
		\abs{b_k} \leq \sum_{s+l+j+m=k} C^{s+j+1} j! C^{l+m+1}\abs{z_n}^m l! \sum_{w=0}^s \abs{z_n}^w \abs{\eta'}^{s-w}\,.
	\]
	Using \cref{eq:countsum} leads to the first part of \cref{eq:boundonb_k}. (Observe that the sum of the exponents of $j$, $l$, $\abs{\eta'}$, and $\abs{z_n}$ is always $k$, having bounded $j! \leq j^j, l! \leq l^l$.)
	To obtain the second inequality in \cref{eq:boundonb_k} we merely have to take into account \cref{eq:countsum} a second time.

	We turn now to proving \cref{eq:diffboundbb_k} so that from now on we have $h\abs{z_n} \leq \eps_0$ and $C_0hN \leq 1/2-\eps_0$ for some sufficiently large $C_0>0$ and $\tau \in \RR$ near $1$.

	\textbf{Case 1:} $h\abs{\eta'} \geq \eps_0$. %
	Because $a$ is the holomorphic finite realization of a formal cl.a.s. of order $0$, by \cref{lem:recoversymbol} it is bounded. Thus,
	\begin{equation}\label{eq:boundjustb}
		\abs{b(y',z_n,\tau,\eta';h)} \leq C = Ch^N\abs{\eta'}^N (h\abs{\eta'})^{-N} \leq C\eps_0^{-N} h^N \abs{\eta'}^N\,.
	\end{equation}
	Furthermore, using \cref{eq:boundonb_k} together with $1 \leq \eps_0^{-j} h^j \abs{\eta'}^j$ for all $j\geq 0$ we find that
	\[
		h^k\abs{b_k} \leq h^N C^{k+1} \sum_{j+m=k} j! \abs{\eta'}^{N-j-m} \abs{z_n}_+^m\,.
	\]
	Plugging these two estimates into the calculation
	\[
		\abs{b(y',z_n,\tau,\eta';h) - \sum_{k=0}^{N-1} h^k b_k(y',z_n,\tau,\eta')} \leq \abs{b} + \sum_{k=0}^{N-1} h^k \abs{b_k}\,,
	\]
	gives \cref{eq:diffboundbb_k} in this case.
	
	\textbf{Case 2:} assume that $h\abs{\eta'} \leq \eps_0$. This implies that $\abs{\eta'+h^{-1}\tau\xi'_0}/C_0\geq N$ for $\abs{\tau-1}< 1/2$ because we had assumed $hC_0N\leq 1/2-\eps_0$. By \cref{eq:symbola} and \cref{lem:recoversymbol} there is $C>0$ so that for all $N\in\mathbb{N}$,
	\[
		a^\ast(y',hz_n,h^{-1}(h\eta'+\tau\xi'_0)) = \sum_{j=0}^{N-1} h^ja^\ast_j(y',hz_n,h\eta'+\tau\xi'_0) + R_{N,a^\ast}\,,\qquad \abs{R_{N,a^\ast}} \leq C^{N+1}N!\abs{\eta'+h^{-1}\tau\xi'_0}^{-N}\,.
	\]
	Using the reverse triangle inequality and $\abs{\tau\xi'_0} \geq 1/2$ (for $\tau$ near $1$) to write $\abs{\eta'+h^{-1}\tau\xi'_0}^{-1} = h\abs{h\eta'+\tau\xi'_0}^{-1}\leq h(1/2-\eps_0)^{-1}$, we find
	\begin{equation}\label{eq:rnaastbound}
		\abs{a^\ast(y',hz_n,h^{-1}(h\eta'+\tau\xi'_0)) - \sum_{j=0}^{N-1} h^ja^\ast_j(y',hz_n,h\eta'+\tau\xi'_0)} \leq C^{N+1} h^N N!\,,
	\end{equation}
	and we have the analogous estimates for $a$.
	At this point, let us write, for any $N\in \mathbb{N}$, 
	\begin{align*}
		b(y',z_n,\tau,\eta';h) &= \left(a^\ast(y',hz_n,h^{-1}(h\eta'+\tau\xi'_0)) - \sum_{j=0}^{N-1} h^j a_j^\ast(y',hz_n,h\eta'+\tau\xi'_0)\right)a(y',hz_n,h^{-1}\tau\xi'_0) \\
		&+ \sum_{j=0}^{N-1}h^ja_j^\ast(y',hz_n,h\eta'+\tau\xi'_0) \left(a(y',hz_n,h^{-1}\tau\xi'_0)-\sum_{k=0}^{N-1-j}h^k a_k(y',hz_n,\tau\xi'_0)\right) \\
		&+ \sum_{j=0}^{N-1} h^ja_j^\ast(y',hz_n,h\eta'+\tau\xi'_0) \sum_{k=0}^{N-1-j} h^ka_k(y',hz_n,\tau\xi'_0)\,.%
	\end{align*}
	Noting that for some $C>0$ we have $\abs{a_j^\ast(y',hz_n,h\eta'+\tau\xi'_0)} \leq C^{j+1} j!$ uniformly in all variables because $h\abs{\eta'} \leq \eps_0 < 1$, using \cref{eq:rnaastbound} (and its analogy for $a$) 
	upon rewriting the last term above, we find 
	\begin{equation}\label{eq:firstboundonb}
		\abs{b- \sum_{s=0}^{N-1}h^s\sum_{j+k=s}a_j^\ast a_k} \leq C^{N+1} h^N N!\,.
	\end{equation}
	We now write
	\begin{equation}\label{eq:ajk}
	\begin{aligned}
		a_j^\ast a_k &= a_k\left(a_j^\ast-\sum_{s=0}^{N-j-k-1} h^s \frac{\del_h^s (a_j^\ast)\vert_{h=0}}{s!}\right) + \sum_{s=0}^{N-j-k-1} h^s \frac{\del_h^s (a_j^\ast)\vert_{h=0}}{s!} \left(a_k - \sum_{w=0}^{N-j-k-s-1}h^w \frac{\del_h^w (a_k)\vert_{h=0}}{w!}\right) \\
		&+ \sum_{s=0}^{N-j-k-1} h^s \frac{\del_h^s (a_j^\ast)\vert_{h=0}}{s!}\sum_{w=0}^{N-j-k-s-1}h^w \frac{\del_h^w (a_k)\vert_{h=0}}{w!}\,.
	\end{aligned}
	\end{equation}
	Using Taylor's theorem and the Cauchy inequalities with radius $r = \min\{\eps_0\abs{\eta'}^{-1},\eps_0\abs{z_n}^{-1},\eps_0\}$ (recall $h\abs{\eta'} \leq \eps_0$ and $h\abs{z_n}\leq \eps_0$), we have for any $k$,
	\[
		\abs{a^\ast_j(y',hz_n,h\eta'+\tau\xi'_0) - \sum_{l=0}^{k-1} \frac{h^l}{l!}\del_h^l (a^\ast_j(y',hz_n,h\eta'+\tau\xi'_0))\vert_{h=0}} \leq C^{j+1}j! h^k\eps_0^{-k}\sum_{l=0}^{k} \abs{z_n}_+^l \abs{\eta'}^{k-l}\,.
	\]
	Therefore,
	\[
		\abs{a^\ast_j(y',hz_n,h\eta'+\tau\xi'_0) - \sum_{l=0}^{N-1-j-k} h^l (l!)^{-1}\del_h^l (a^\ast_j(y',hz_n,h\eta'+\tau\xi'_0))\vert_{h=0}}  \leq C^{N+1}j! h^{N-j-k} \sum_{s=0}^{N-j-k}\abs{z_n}_+^s\abs{\eta'}^{N-j-k-s}\,.
	\]
	Repeating analogous calculations for $a(y',hz_n,\tau\xi'_0)$, we find
	\[
		\abs{a_k(y',hz_n,\tau\xi'_0) - \sum_{w=0}^{N-1-j-k-s} h^w (w!)^{-1}\del_h^w (a_k(y',hz_n,\tau\xi'_0))\vert_{h=0}} \leq  C^{N+1}k! h^{N-j-k-s}\sum_{w=0}^{N-j-k-s}\abs{z_n}_+^w\abs{\eta'}^{N-j-k-s-w}\,.
	\]
	Using that $\abs{a_k}=\abs{a_k(y',hz_n,\tau\xi'_0)}\leq C^{k+1}k!$ uniformly in all variables as well as the bounds given above, we find that from \cref{eq:ajk} that
	\begin{equation}\label{eq:ajakbound}
		\abs{a_j^\ast a_k - \sum_{f=0}^{N-1-j-k} h^f \sum_{s+w=f}\frac{\del_h^s(a_j^\ast)\vert_{h=0}}{s!}\frac{\del_h^w (a_k)\vert_{h=0}}{w!}}\leq C^{N+1} h^{N-j-k}j!k!\sum_{s+l=0}^{N-j-k} s!\abs{z_n}_+^l\abs{\eta'}^{N-j-k-s-l}\,.
	\end{equation}
	Combining \cref{eq:ajakbound,eq:firstboundonb} and recalling the formula of $b_k$ we find \cref{eq:diffboundbb_k}.
\end{proof}

\begin{lemma}\label{lem:rhoprops}
	Let $\xi'_0\in S^{n-2}$ be arbitrary and fixed. Using notation from \cref{cor:kerndia}, define for $\tau \in \RR$ near $1$, 
	\[
		\rho(y',z_n,\tau,\eta';h) \coloneqq i(\psi(y',hz_n,h^{-1}\tau\xi'_0)-\psi^\ast(y',hz_n,h^{-1}(h\eta'+\tau\xi'_0)))\,.
	\]
	For any integer $k\geq 0$, $\del^k_h e^\rho\vert_{h=0}$ is a polynomial of degree $k$ in $\eta'$ with analytic coefficients (in $\tau$ as well) and for every complex compact neighborhood $K$ of $0$ there are $C,c>0$ so that for any integer $k \geq 0$, %
	\begin{equation}\label{eq:etabounds}
		\sup_{y'\in K}\abs{\del^k_h e^\rho\vert_{h=0}} \leq C^{k+1} k!  e^{-z_n c}\sum_{j=0}^k \abs{z_n}^j\abs{\eta'}^{k-j}\,.
	\end{equation}
	For any integers $j,k\geq 0$,
	\begin{equation}\label{eq:rhoboundwithint}
		\int_0^\infty \abs{z_n}_+^j \abs{\frac{\del_h^k e^{\rho}}{k!}\vert_{h=0}}\dd z_n \leq C^{k+j+1} j! \sum_{s=0}^k s!\abs{\eta'}^{k-s}\,.
	\end{equation}
	In addition, let $\chi \in C^\infty_c(\RR)$ be identically $1$ in $\abs{z_n} \leq \eps_0/2$ and vanish in $\abs{z_n} \geq \eps_0$. For every $j,N\in\mathbb{N}$ and every compact $K$ we have, for $\tau \in \RR$ near $1$,
	\begin{equation}\label{eq:rhodiffint}
		\sup_{y'\in K}\abs{\int_0^\infty \abs{z_n}_+^j \left(e^\rho - \sum_{k=0}^{N-1} h^k \frac{\del_h^k e^\rho\vert_{h=0}}{k!}\right)\chi\dd z_n} \leq C^{N+j+1}h^Nj!\sum_{s=0}^N s!\abs{\eta'}^{N-s}\,.
	\end{equation}
\end{lemma}
\begin{proof}
	Recalling \cref{cor:kerndia}, 
	using Taylor's theorem with remainder of order $1$ (with respect to $z_n$) as well as the homogeneity of $\phi, r^{1/2}$ of degree $1$,
	\begin{equation}\label{eq:another_rho}
		\rho = -z_n \int_0^1 r^{1/2}(y',shz_n,\del_{y'}\phi(y',shz_n,\tau\xi'_0)) + r^{1/2}(y',shz_n,\del_{y'}\phi(y',shz_n,h\eta'+\tau\xi'_0))^\ast \dd s
	\end{equation}
	and using the positivity and homogeneity of $r$ we find, for $z_n \geq 0$,
	\begin{equation}\label{eq:rhounif}
		\abs{e^\rho} \leq e^{-z_n c(\abs{\tau}+\abs{h\eta'+\tau\xi'_0})}
	\end{equation}
	for some $c>0$ uniform in all variables. 

	We proceed to proving the second part of \cref{eq:etabounds}. By \cref{eq:another_rho}, $\rho$ is holomorphic with respect to $h$ at $h=0$ as long as $\abs{h z_n}, \abs{h \eta'}$ are sufficiently small. Working in the separate cases $\abs{z_n} \leq \abs{\eta'}$ and $\abs{\eta'} \leq \abs{z_n}$, we deduce from the Cauchy inequalities applied to \cref{eq:another_rho} that
	\[
		\sup_{y'\in K}\abs{\del^k_h \rho\vert_{h=0}}\leq C^{k}k!\abs{z_n} (\abs{z_n}^k + \abs{\eta'}^k) \leq C^{k}k!\sum_{j=0}^k\abs{z_n}^{j+1}\abs{\eta'}^{k-j}
	\]
	for some $C>0$ independent of $k$. Using Fa\`a di Bruno's formula we know that, for $\Pi_k$ the set of partitions of $\{1,\dots,k\}$, for all $k\geq 0$,
	\begin{equation}\label{eq:faaonrho}
		\del^k_h e^\rho\vert_{h=0} = e^{\rho}\vert_{h=0} \sum_{\pi \in \Pi_k} \prod_{B\in\pi} \del^{\abs{B}}_h \rho\vert_{h=0}\,,
	\end{equation}
	which together with \cref{eq:another_rho} immediately implies that $\del_h^ke^\rho\vert_{h=0}$ is a polynomial of degree $k$ in $\eta'$ with analytic coefficients. 

	Furthermore, from \cref{eq:rhounif,eq:faaonrho} we obtain (for $\eta'\neq 0$)
	\[
		\abs{\del_h^k e^{\rho}\vert_{h=0}} \leq e^{-z_n c} C^k \abs{\eta'}^k \sum_{\pi \in \Pi_k} \abs{z_n}^{\abs{\pi}} \prod_{B\in\pi} \abs{B}!\sum_{j=0}^{\abs{B}} (\abs{z_n}/\abs{\eta'})^j\,,
	\]
	where we used that $\sum_{B\in \pi} \abs{B} = k$ for any $\pi\in \Pi_k$. Note that for any $i_1,\dots,i_k \geq 0$ with $i_1+\dots+i_k = k$ and any $\zeta\geq 0$ we have
	\[
		\left(\sum_{j=0}^{i_1} \zeta^j\right)\cdots \left(\sum_{j=0}^{i_k} \zeta^j\right) \leq \sum_{j=0}^{k} \sum_{\substack{i_1,\dots,i_k \geq 0\\ i_1+\dots+i_k=j}} \zeta^j \leq 4^{k}\sum_{j=0}^k \zeta^j\,,
	\]
	where we used \cref{eq:countsum}. Plugging in $\abs{z_n}\abs{\eta'}^{-1}$ for $\zeta$ we find 
	\[
		\abs{\del_h^k e^{\rho}\vert_{h=0}} \leq e^{-z_n c}C^{k}4^k \sum_{j=0}^k \abs{z_n}^j \abs{\eta'}^{k-j} \sum_{\pi \in \Pi_k} \abs{z_n}^{\abs{\pi}} \prod_{B\in \pi} \abs{B}!\,.
	\]
	Applying \cref{lem:faabound} completes the proof of \cref{eq:etabounds}.

	To obtain \cref{eq:rhoboundwithint} we apply \cref{eq:etabounds} to find that for any $j,k\geq 0$,
	\[
		\int_0^\infty \abs{z_n}_+^j \abs{\frac{\del_h^k e^{\rho}}{k!}\vert_{h=0}}\dd z_n \leq C^{k+1} \sum_{s=0}^k \abs{\eta'}^{k-s} \int_0^\infty \abs{z_n}_+^{j+s} e^{-z_n c}\dd z_n\,,
	\]
	and because we identify the Gamma function here, we have
	\[
		 \int_0^\infty \abs{z_n}_+^{j+s} e^{-z_n c}\dd z_n\leq \int_0^\infty (1+z_n^{j+s}) e^{-z_n c}\dd z_n = c^{-1} + c^{-1-j-s} (j+s)!\,,
	\]
	which leads to \cref{eq:rhoboundwithint} by applying \cref{eq:sumfac}.

	To prove \cref{eq:rhodiffint}, we again make a case distinction. 

	\textbf{Case 1:} We have $h\abs{\eta'} \geq \eps_0$. First, \cref{eq:rhoboundwithint} implies that
	\[
		\sum_{k=0}^{N-1} h^k \int_0^\infty \abs{z_n}_+^j\abs{\frac{\del_h^k e^\rho\vert_{h=0}}{k!}} \dd z_n \leq \sum_{k=0}^{N-1} h^k C^{k+1} j!\sum_{s=0}^k s!\abs{\eta'}^{k-s} \leq h^N C^{N+1} j!\sum_{s=0}^{N} s! \abs{\eta'}^{N-s}\,.
	\]
	On the other hand, \cref{eq:rhounif} implies 
	\[
		\int_0^\infty \abs{z_n}_+^j\abs{e^\rho} \abs{\chi}\dd z_n \leq Cj! \leq C^{N+1}j!h^N\abs{\eta'}^N\,,
	\]
	for some uniform constant $C>0$. Combining these two estimates gives \cref{eq:rhodiffint} in this case. 

	\textbf{Case 2:} We have $h\abs{\eta'} \leq \eps_0$. Using Taylor's theorem and the Cauchy inequalities we find that
	\[
		\abs{e^\rho - \sum_{k=0}^{N-1} h^k \frac{\del_h^k e^\rho\vert_{h=0}}{k!}} \leq C^{N+1} h^N e^{-z_n c} \sum_{s=0}^N\abs{z_n}_+^s\abs{\eta'}^{N-s}\,,
	\]
	and arguing analogously as for the proof of \cref{eq:rhoboundwithint} we find \cref{eq:rhodiffint}.
\end{proof}

\begin{lemma}\label{lem:justt}
	Let $\xi'_0\in S^{n-2}$ be arbitrary and fixed. Let $\chi \in C^\infty_c(\RR)$ be identically $1$ in $\abs{z_n} \leq \eps_0/2$ and vanish in $\abs{z_n} \geq \eps_0$. Define for $\tau \in \RR$ near $1$,%
	\begin{equation}\label{eq:definet}
		\begin{gathered}
		t(y',\tau,\eta;h)\coloneqq \\ %
		(2\pi)^{n-2}\int_0^\infty e^{i(\psi(y',z_n,h^{-1}\tau\xi'_0)-\psi^\ast(y',z_n,\eta'+h^{-1}\tau\xi'_0)+z_n\eta_n)}a(y',z_n,h^{-1}\tau\xi'_0)a^\ast(y',z_n,\eta'+h^{-1}\tau\xi'_0)\chi(z_n)\dd z_n\,,
		\end{gathered}
	\end{equation}
	which is holomorphic with respect to $y'$ near $0$ for any fixed $\eta\in \RR^{n}$ and jointly continuous in $(y',\tau,\eta')$ for $y'$ near $0$ and all $\eta \in \RR^n$. Furthermore, for any $k\geq 0$ define
	\begin{equation}\label{eq:tkformula}
			t_k(y',\tau,\eta) \coloneqq (2\pi)^{-1}\sum_{s+j+l=k} \int_0^\infty b_l(y',z_n,\tau,\eta') \frac{\del_h^j (e^{\rho(y',z_n,\tau,\eta';h)})}{j!}\vert_{h=0} \frac{(iz_n\eta_n)^s}{s!} \dd z_n\,.
	\end{equation}
	We have $t_0(y',\tau,\eta) = (2\pi)^{-n}(2r^{1/2}(y',0,\tau\xi'_0))^{-1}$ and the function $t_k$ is a polynomial of degree $k$ with respect to $\eta$ and we denote by $t_{k,j}$ the $j$-th order homogeneous term of $t_k$ for any $0\leq j\leq k$. For some $C>0$ independent of all variables we have
	\begin{equation}\label{eq:boundsont_kOLD}
		\abs{t_{k,j}}\leq C^{k+1} (k-j)! \abs{\eta}^j\,,\qquad \abs{t_k} \leq C^{k+1} k! \abs{\eta}^k_+\,,
	\end{equation}
	and each $t_k$ depends holomorphically on $\tau$ for $\tau$ near $1$ and each $t_k$ is positively homogeneous of degree $-k-1$ in $\tau$. 

	Finally, for any complex compact neighborhood $K$ of $0$, for every sufficiently large $C_0>0$, any $N$, and any $h>0$ so that $C_0hN \leq 1/2-\eps_0$, there is some $C>0$ so that
	\begin{equation}\label{eq:diffboundt}
		\sup_{y'\in K}\abs{h^{-1}t - \sum_{k=0}^{N-1} h^k t_k} \leq C^{N+1}h^N \sum_{s=0}^N s!\abs{\eta}^{N-s}\,.
	\end{equation}
\end{lemma}
\begin{proof}
	That $t$ is holomorphic with respect to $y'$ for $y'$ near $0$ and all $\eta'\in \RR^{n-1}$ follows from \cref{cor:kerndia} and the fact that $\chi$ is compactly supported so that we can interchange integration and $\bar\del_{y'}$. 

	That $t_k$ is a polynomial of degree $k$ in $\eta$ with analytic coefficients follows from \cref{lem:bproperties,lem:rhoprops} and that these depend holomorphically on $\tau$ near $1$ is a consequence of the fact that this is true for each $b_k$ and by observing the explicit definition of $\rho$ and each $t_k$. Because $b_0 \equiv (2\pi)^{-n+1}$ and the explicit formula for $t_0$ and $\rho\vert_{h=0}$, we find the given expression for $t_0$. 
	Furthermore, \cref{eq:boundsont_kOLD} is a consequence of \cref{eq:boundonb_k,eq:rhoboundwithint}. 

	Performing the substitution $z_n \mapsto hz_n$ in \cref{eq:definet}, we find
	\[
		t(y',\tau,\eta;h) =  h(2\pi)^{-1}\int_0^\infty e^{\rho(y',z_n,\tau,\eta';h)+ihz_n\eta_n} b(y',z_n,\tau,\eta';h)\chi(hz_n)\dd z_n\,.
	\]
	Note that to calculate $t-\sum^{N-1}_{k=0}h^k t_k$ for any $N\in\mathbb{N}$, we write
	\begin{gather*}
		2\pi h^{-1}t \\
		\begin{aligned} &= \int_0^\infty e^{ihz_n\eta_n}\left(e^\rho-\sum^{N-1}_{k=0} h^k \frac{\del_h^k e^\rho\vert_{h=0}}{k!}\right)b\chi\dd z_n + \int_0^\infty e^{ihz_n\eta_n}\sum^{N-1}_{k=0} h^k \frac{\del_h^k e^\rho\vert_{h=0}}{k!}\left(b-\sum_{j=0}^{N-k-1}h^jb_j\right)\chi\dd z_n \\
		&+ \int_0^\infty e^{ihz_n\eta_n}\sum^{N-1}_{k=0} h^k\frac{\del_h^k e^\rho\vert_{h=0}}{k!}\sum_{j=0}^{N-k-1}h^jb_j(\chi-1)\dd z_n +   \int_0^\infty \sum^{N-1}_{k=0} h^k\frac{\del_h^k e^\rho\vert_{h=0}}{k!}\sum_{j=0}^{N-k-1}h^jb_j\sum_{l=0}^{N-k-j-1}h^l \frac{(iz_n\eta_n)^l}{l!} \dd z_n\\ 
		&+\int_0^\infty \sum^{N-1}_{k=0} h^k\frac{\del_h^k e^\rho\vert_{h=0}}{k!}\sum_{j=0}^{N-k-1}h^jb_j \left(e^{ihz_n\eta_n} - \sum_{l=0}^{N-k-j-1}h^l \frac{(iz_n\eta_n)^l}{l!}\right)\dd z_n\,.
		\end{aligned}
	\end{gather*}
	The fourth term is equal to $(2\pi)\sum_{k=0}^{N-1} h^k t_k$ after reordering and one may use the bounds \cref{eq:rhodiffint,eq:boundonb_k} (and writing $b= b-b_0+b_0$) to bound the first term, \cref{eq:diffboundbb_k,eq:rhoboundwithint} to bound the second term and the first part of \cref{eq:boundonb_k}, and \cref{eq:etabounds} and \cref{eq:rhoboundwithint} to bound the third term, as well as Taylor's theorem and \cref{eq:boundonb_k,eq:rhoboundwithint} to bound the last term to obtain \cref{eq:diffboundt}.%

	To prove the homogeneity of each $t_k$ we proceed as follows: using \cref{eq:definet}, for any $s$ real near $1$, we have $h^{-1} t(y',s\tau,\eta;h) = h^{-1}t(y',\tau,\eta;h/s) = s^{-1} (h/s)^{-1} t(y',\tau,\eta;h/s)$, so that expanding the LHS in terms of $t_k$ and the RHS in terms of $t_k$, we obtain the formal sums $\sum_{k\geq 0} h^k t_k(y',s\tau,\eta)$ and $\sum_{k\geq 0} h^k s^{-k-1} t_k(y,\tau,\eta)$ respectively, and the uniqueness of asymptotic (Poincar\'e) expansions implies that $s^{-k-1} t_k(y,\tau,\eta) = t_k(y',s\tau,\eta)$ for real $s$ near $1$ and all $k$. Since however, both terms are holomorphic with respect to $s$, this equality extends to complex $s$.
\end{proof}

We can now give the
\begin{proof}[Proof of \cref{prop:allthingst}]
	Assuming that $q \in C_c^\infty(\RR^n)$ is supported in $\abs{y_n}\leq \eps_0/2$, by introducing a smooth cut-off $\chi \in C_c^\infty(\RR)$ identically $1$ in $\abs{z_n} \leq \eps_0/2$ and vanishing in $\abs{z_n} \geq \eps_0$, we may write for any $\xi' \in \RR^{n-1}\sem$, by substitution,
	\begin{gather*}
		T(\mathbb{1}_{\{y_n\geq 0\}}q)(x',\xi') = \\
		\iint e^{i(x'-y')\cdot\eta'}\int_0^\infty e^{i(\psi(y',z_n,\xi')-\psi^\ast(y',z_n,\eta'+\xi'))}a(y',z_n,\xi') a^\ast(y',z_n,\eta'+\xi')q(y',z_n)\chi(z_n)\dd z_n\dd\eta'\dd y'\,.
	\end{gather*}
	Using Fourier inversion $q(y',z_n) = (2\pi)^{-1}\iint e^{i(z_n-y_n)\cdot \eta_n}q(y',y_n)\dd \eta_n\dd y_n$, %
	we find the formula \cref{eq:sigmaist} for $t$ defined by \cref{eq:definet}. \cref{lem:justt} completes the proof: because $t$ is continuous in $\eta' \in \RR^n$, the integral above is interpreted as the oscillatory integral ($\xi' = h^{-1}\tau\xi'_0 \in \RR^{n-1}\sem$)
	\[
		T(\mathbb{1}_{\{y_n\geq 0\}}q)(x',h^{-1}\tau\xi'_0) = \iint e^{i(x'-y)\cdot\eta}(1+\abs{\eta}^2)^{-k} (1+\Delta_y)^k \left(t(y',\tau,\eta;h) q(y)\right)\dd\eta\dd y\,,
	\]
	for any $k\in \mathbb{N}$ and taking, say, $N=1$ in \cref{eq:tpolybound}, and applying the Cauchy estimates guarantees that the above is defined via an absolutely integrable integrand when $k$ is sufficiently large.
\end{proof}

\subsection{Finding an analytic candidate}

For $t_k$ from \cref{prop:allthingst}, we write
\[
	\mathrm{Op}^R(t_k) = \int e^{i(x-y)\cdot \eta} t_k(y',\tau,\eta) \dd\eta
\]
for the differential operator defined by $t_k$, where the superscript $R$ signifies the right-quantization. 

We can now prove
\begin{proposition}\label{prop:sigmatotq}
	Let $\xi'_0\in S^{n-2}$ be fixed and $\tau\in \CC$ near $1$. 

	If $\td{q} \in C_c^\infty(\RR^n)\cap C^\omega(U)$ where $U$ is an open neighborhood of $0$ has support sufficiently near $0$, then $h\sum_{k\geq 0}h^k\mathrm{Op}^R(t_k)\td{q}(x',0)$ is a formal cl.a.s. of order $-1$ defined in an open complex neighborhood of $(0,1)$ and $\sigma_a^{\mathbb{1}_{\{y_n\geq 0\}}\td{q}}(x';h^{-1}\tau\xi'_0) \sim h\sum_{k\geq 0}h^k\mathrm{Op}^R(t_k)\td{q}(x',0)$ there.%
\end{proposition}
\begin{proof}
	Let $\td{q}\in C^\omega(U)$. We know from \cref{eq:boundsont_k} that for every $k\geq 0$, $\mathrm{Op}^R(t_k)\td{q} = \sum_{j=0}^k \mathrm{Op}^R(t_{k,j})\td{q}$, where $t_{k,j}$ is a polynomial of degree $j$ in $\eta$ bounded uniformly in $y'$ above by $C^{k+1}(k-j)!\abs{\eta}^j$. Furthermore, because $\mathrm{Op}^R(t_k)\td{q}$ is a differential operator of order $k$ with analytic coefficients acting on $\td{q}$ by \cref{prop:allthingst}, for some $C>0$ and all $k\in\mathbb{N}$
	\[
		\abs{\mathrm{Op}^R(t_k)\td{q}} \leq \sum_{j=0}^k C^{k+1} (k-j)! C^{j+1} j! \leq C^{2k+2} (k+1)! \leq (3C)^{2k} k!\,,
	\]
	in a complex neighborhood of $(0,1)$ (with respect to $x,\tau$), so that we have indeed shown that $\sum_{k\geq 0} h^k \mathrm{Op}^R(t_k)\td{q}\vert_{x_n=0}$ is a formal cl.a.s. near $(0,1)$ (with respect to $x',\tau$).

	According to \cref{eq:sigmaist}, when $\tau\in\RR$ is near $1$, we have
	\[
		\sigma_b^{\mathbb{1}_{\{y_n\geq 0\}}\td{q}}(x',h^{-1}\tau\xi'_0)= T(\mathbb{1}_{\{y_n\geq 0\}}\td{q})(x',h^{-1}\tau\xi'_0) = \iint e^{i(x'-y)\cdot\eta} t(y',\tau,\eta;h) \td{q}(y)\dd\eta \dd y\,. %
	\]
	Now relying on \cref{eq:tpolybound} together with the Cauchy inequalities and using that $\td{q}\in C_c^\infty$, for $j$ sufficiently large, 
	\[
		\abs{\iint e^{i(x'-y)\cdot\eta}\left(t - h\sum_{k=0}^{N-1}h^k t_k\right)\td{q} \dd\eta\dd y} = \abs{\iint e^{i(x'-y)\cdot\eta}(1+\abs{\eta}^2)^{-j}(1+\Delta_y)^j\left(\left(t - h\sum_{k=0}^{N-1}h^k t_k\right)\td{q}\right) \dd\eta\dd y} = \mathcal{O}(h^{N+1})\,.
	\]
	Recall from \cref{prop:symbcalc}, for $(x',\tau)$ real near $(0,1)$,
	\[
		(\sigma_a^{\mathbb{1}_{\{y_n\geq 0\}}\td{q}}-\sigma_b^{\mathbb{1}_{\{y_n\geq 0\}}\td{q}})(x',h^{-1}\tau \xi'_0) = \mathcal{O}_M(h^M)
	\]
	for all $M$. By \cref{prop:direc1}, we may expand $\sigma_a^{\mathbb{1}_{\{y_n\geq 0\}}\td{q}}\sim h\sum_{k\geq 0}h^k\sigma_k$ for some $\sigma_k$, and since (Poincar\'e) asymptotic expansions of functions are unique (\cite[\S~2.1(iii)]{NIST:DLMF}, \cite{zbMATH00994432}), we have shown that for real $(x',\tau)$ near $(0,1)$ we have
	\[
		\sigma_k = \mathrm{Op}^R(t_k)\td{q}\,.
	\]
	Furthermore, because both the LHS and RHS above are holomorphic functions of $(x',\tau)$ near $(0,1)$, and their restrictions to a real open neighborhood of $(0,1)$ align, they must coincide for complex $(x',\tau)$ near $(0,1)$ (see for example \cite[Thm.~3.1.15]{hoermander1}). 
\end{proof}

\begin{proposition}\label{thm:foundtdq}
	Let $\xi'_0\in S^{n-2}$ be fixed arbitrarily and $\tau\in \CC$ sufficiently near $1$. For every $q\in L^\infty(\RR^n)$ supported in $\{y_n \geq 0\}$ for which $\sigma_a^q(x',h^{-1}\tau\xi'_0)$ is a cl.a.s. of order $-1$ (in $h>0$, depending holomorphically on $(x',\tau)$ near $(0,1)$), there is $\td{q}\in C_c^\infty(\RR^n)$ that is analytic near $0$ so that $\sigma_a^{q-\mathbb{1}_{\{y_n\geq 0\}}\td{q}}(x',h^{-1}\tau\xi'_0)$ is a holomorphic function of $(x',\tau)$ on an open complex neighborhood of $(0,1)$ so that %
	\[
		\sigma_a^{q-\mathbb{1}_{\{y_n\geq 0\}}\td{q}}(x',h^{-1}\tau\xi'_0) \in \bl{O}^{(-\omega)}_{(0,1)}\,.
	\]
\end{proposition}
\begin{proof}
	According to our assumption we can write, for some $\sigma_k$,
	\begin{equation}\label{eq:clasexp}
		h^{-1}\sigma_a^q(x',h^{-1}\tau\xi'_0) \sim \sum_{k\geq 0}h^k \sigma_k(x',\tau\xi'_0)\,.
	\end{equation}
	Because $\sigma_a^q(x',h^{-1}\tau\xi'_0)$ is a function of $\tau/h$, for real $s$ near $1$, $h^{-1}\sigma_a^q(x',h^{-1}s\tau\xi'_0) = s^{-1}(h/s)^{-1}\sigma_a^q(x',(h/s)^{-1}\tau\xi'_0)$, so that expanding both of these expressions in powers of $h$ (using \cref{eq:clasexp}), one obtains homogeneity of degree $k+1$ of each $\sigma_k(x',h^{-1}s\tau\xi'_0)$ for real $s$ near $1$, which due to holomorphicity and the identity theorem is upgraded to the desired homogeneity: $\sigma_k(x',\tau) = \tau^{-k-1}\sigma_k(x',1)$ for all complex $(x',\tau)$ near $(0,1)$ and all $k$.

	We seek $\td{q}$ of the form
	\begin{equation}\label{eq:tdq}
		\td{q}(x) = \chi(x)\sum_{k\geq 0} q_k(x';\xi'_0)\frac{x_n^k}{k!}\,,
	\end{equation}
	where $\chi \in C_c^\infty$ is identically $1$ near $0$, and we want 
	\begin{equation}\label{eq:whatwewant}
		\mathrm{Op}^{R}(t_k)\td{q}(x',0) = \tau^{-k-1}\sigma_k(x',\xi'_0) = \sigma_k(x',\tau\xi'_0)
	\end{equation}
	for all $k\geq 0$. Note that if $\tau$ is real and near $1$, then $t_k$ is homogeneous of degree $-k-1$ in $\tau$ by \cref{prop:allthingst}. In particular, if we obtain	    
	\[
		\mathrm{Op}^{R}(t_k)\td{q}(x',0)\vert_{\tau =1} = \sigma_k(x',\xi'_0)
	\]
	then by virtue of the homogeneity we have \cref{eq:whatwewant} for $\tau$ real near $1$ and because both the LHS and RHS are holomorphic in $\tau$ near $1$ we will then have it for all complex $\tau$ near $1$. Thus we restrict to $\tau = 1$. 

	Because $t_0 =(2\pi)^{-n}(2r^{1/2}(y',0,\tau\xi'_0))^{-1}$, define (the factor $(2\pi)^{-n}$ disappears due to Fourier inversion when applying $\mathrm{Op}^R(t_0)$)
	\[
		q_0(x';\xi'_0) = 2 r^{1/2}(x',0,\xi'_0)\sigma_0(x';\xi'_0)\,.
	\]
	We use the explicit formula for $t_k$ from \cref{eq:tkformula}: (which uses that $b_0=(2\pi)^{-n+1}$ and $e^{\rho}\vert_{h=0}$ is explicit too)
	\[
		\mathrm{Op}^R(t_k) = (2r^{1/2}(x',0,\xi'_0))^{-k-1}\del_{x_n}^k + \sum_{j=0}^{k-1} \del_{x_n}^j P_{k,j}(x',D';\xi'_0)\,,
	\]
	where $P_{k,j}$ are differential operators of order $k-j$ in $y'$ with analytic coefficients. %
	We want the RHS above when applied to $\td{q}$ and evaluated at $(x',0)$ to equal $\sigma_k(x',\xi'_0)$, so we inductively define
	\begin{equation}\label{eq:qkdef}
		q_k(x';\xi'_0) \coloneqq (2r^{1/2}(x',0,\xi'_0))^{k+1}\left(\sigma_k(x';\xi'_0)- \sum_{j=0}^{k-1} \del_{x_n}^j P_{k,j}(x',D';\xi'_0)\sum_{s=0}^{k-1} q_s(x')\frac{x_n^s}{s!}\vert_{x_n=0}\right)\,.
	\end{equation}
	Noting that according to \cref{eq:boundsont_k}, we can write $\sum_{j=0}^{k-1} \del_{x_n}^j P_{k,j} = \sum_{j=0}^k \td{P}_{k,j}$, where each $\td{P}_{k,j}$ is a differential operator of order $j$ with analytic coefficients bounded by $C^{k+1}(k-j)!$, we deduce (by an inductive argument) that each $q_k$ defined according to \cref{eq:qkdef} is analytic in $x'$ and satisfies $\abs{\del^\alpha q_k} \leq C^{\abs{\alpha}+k+1}\alpha!k!$ (with $C>0$ independent of $k$). In particular, defining $\td{q}$ according to \cref{eq:tdq} and choosing the support of $\chi$ therein sufficiently small to guarantee convergence of the series, $\td{q}$ is analytic near $x=0$ and (modulo the cut-off $\chi$) uniquely determined by the choices of $q_k$ we have made.

	Because $\chi$ in $\td{q}$ is identically $1$ near the origin, we have achieved that 
	\[
		h^{-1}\sigma_a^q(x',h^{-1}\tau\xi'_0)\sim \sum_{k\geq 0} h^k\sigma_k(x',\tau\xi'_0) = \sum_{k\geq 0}h^k\mathrm{Op}^{R}(t_k)\td{q}(x',0)\,.
	\]
	Now, \cref{prop:sigmatotq} states that %
	\[
		h^{-1}\sigma_a^{\mathbb{1}_{\{y_n\geq 0\}}\td{q}}(x',h^{-1}\tau\xi'_0) \sim \sum_{k\geq 0}h^k\mathrm{Op}^{R}(t_k)\td{q}(x',0)\,,
	\]
	which is to say that for complex $(x',\tau)$ sufficiently near $(0,1)$,
	\[
		\sigma_a^{q- \mathbb{1}_{\{y_n\geq 0\}}\td{q}}(x',h^{-1}\tau\xi'_0)=\sigma_a^q(x',h^{-1}\tau\xi'_0) - \sigma_a^{\mathbb{1}_{\{y_n\geq 0\}}\td{q}}(x',h^{-1}\tau\xi'_0) \sim 0\,.
	\]
	At first, the term on the LHS is \emph{a priori} merely a distribution (well-defined since the support of the argument is in $y_n\geq 0$) in space parametrized by $\tau\xi'_0$, but nevertheless the term in the middle is the finite realization of the formal cl.a.s. $0$, which by definition of cl.a.s. gives the desired exponential decay.
\end{proof}

\subsection{Proof of \cref{thm:localsymb}}

As mentioned in the revised \cite{1312.4065v2}, we do not actually require the fact that $\dot\Lambda_q$ is analytic off the diagonal near $0$ in the following proof except for defining the analytic symbol $\sigma_a^q$. 
\begin{proof}[Proof of \cref{thm:localsymb}]
	The implication 2.$\implies$ 1. of this statement was shown in \cref{prop:direc1}. According to \cref{prop:sigmatotq} we know that $\sigma_a^{1_{\{y_n\geq 0\}}\td{q}}$ has lowest order term
	\[
		\mathrm{Op}(t_0)\td{q}(x',0) = (2r^{1/2}(x',0,\tau\xi'_0))^{-1}\td{q}(x',0)\,,
	\]
	which is non-vanishing near $(0,1)$ if and only if $\td{q}(0) \neq 0$. We turn to proving that 1. $\implies$ 2.

    We first reduce to the case that $q \in L^\infty(M)$ is supported sufficiently close to $\hat y = 0$. Choose $\chi_1 \in C_c^\infty(M)$ identically $1$ near $\hat y$ and supported sufficiently close to $\hat y$. The kernel 
    \[
        \dot\Lambda_{(1-\chi_1)q}(x',z') = -\int_M (1-\chi_1(y))q(y)K(y,x')K(y,z')\dd V_g(y) 
    \] 
    is jointly analytic near $(0,0)$ because the Poisson kernels are integrated only where they are analytic. In particular, by the contour formula given in \cref{prop:anasymb} for $\sigma_a^{(1-\chi_1)q}$, a contour deformation using Stokes' theorem to the contour $\td{\gamma}_{\xi'}(t') = t'+ir\xi'_0$ shows that $\sigma_a^{(1-\chi_1)q}(x',h^{-1}\tau\xi'_0) \in \bl{O}_{(0,1)}^{(-\omega)}$. As a consequence, $\sigma_a^{\chi_1 q}$ is a cl.a.s. of order $-1$ for $(x',\tau)$ near $(0,1)$.

    Now choose $\td{q}$ according to \cref{thm:foundtdq} applied to $\sigma_a^{\chi_1 q}$ and denote 
    \[
        u  = \mathbb{1}_{\{y_n\geq 0\}}\td{q} - q = \mathbb{1}_{\{y_n\geq 0\}}\td{q} - \chi_1 q - (1-\chi_1)q\,,
    \]
    so that $\sigma^u_a(x',h^{-1}\tau\xi'_0)\in \bl{O}^{(-\omega)}_{(0,1)}$. 

	For $(\td{\chi}_j)_{j\in \mathbb{N}}$ an Ehrenpreis cut-off sequence identically $1$ and supported near $0$, for $j_\vartheta$ from \cref{prop:K2}, and $\chi_0 \in C_c^\infty(\del M)$ identically $1$ near $0$ and $\supp \chi_0 \subset \{\td{\chi}_j \equiv 1\forall j\}$, define
	\[
		g_{\tau,h}(z') = \tilde{\chi}_{j_\vartheta(h)}(z')e^{\frac{i}{h}\tau z'\cdot\xi'_0}\,,\quad f_{v',\tau,h}(x') = \chi_0(x')e^{\frac{i}{h}\left( \frac{i}{2}(v'-x')^2+\frac{i}{2}(\tau-1)^2+x'\cdot\xi'_0\right)}\,,
	\]
    Because $\td{q}$ is analytic near the origin, using \cref{prop:direc1} and the assumption on $\dot\Lambda_q$, we know that $\dot\Lambda_u$ is analytic off the diagonal near the origin which allows us to apply \cref{prop:anasymb}: for $\vartheta>0$ sufficiently small we have
	\[
		-\dot\Lambda_u g_{\tau,h} = e^{\frac{i}{h}\tau x'\cdot\xi'_0}\sigma_a^{u}(x',h^{-1}\tau\xi'_0) + e^{\frac{i}{h}\tau x'\cdot\xi'_0}\bl{O}^{(-\omega)}_{(0,1)} =  \bl{O}^{(-\omega)}_{(0,1)}\,.
	\]
	Let $\chi\in C_c^\infty$ be a cut-off supported sufficiently close to $0$ and identically $1$ near $0$. The kernel $\dot\Lambda_{(1-\chi)u}$ is analytic near $(0,0)$ because the two Poisson kernels in its definition are separated from the diagonal, and we may assume that $\supp \chi_j$ is a subset of this neighborhood of analyticity for all $j$. In particular, we may use repeated partial integrations to find that, for $\vartheta>0$ sufficiently small, 
	\[
		\dot\Lambda_{(1-\chi)u} g_{\tau,h} \in \bl{O}^{(-\omega)}_{(0,1)}\,.
	\]
	Because $\dot\Lambda_u g_{\tau,h} = \dot\Lambda_{\chi u} g_{\tau,h} + \dot\Lambda_{(1-\chi)u} g_{\tau,h}$, we find that, using \cref{cor:Ks} and \cref{eq:defofS}, 
	\[
		\bl{O}^{(-\omega)}_{(0,1)}\ni \langle\dot\Lambda_{\chi u} g_{\tau,h},f_{v',\tau,h}\rangle = -\int \chi u Kg_{\tau,h} \overline{Kf_{v',\tau,h}}\dd V_g = S_{-\xi'_0}(\chi u)(v',\tau;h)+\bl{O}^{(-\omega)}_{(0,1)}\,.
	\]
	Choosing the support of $\chi$ sufficiently small, \cref{cor:vanish} implies that $\chi u$, and thus $u$, vanishes near the origin which completes the proof.
\end{proof}

\startappendix

\subsection{Some facts about classical analytic symbols}\label{sec:cla}

As a direct consequence of a minor modification of the proof of \cite[Lem.~1.21]{zbMATH08037414}, we have
\begin{lemma}\label{lem:holsymb}
	Let $\ud{a}\coloneqq (a_j)_{j\in\mathbb{N}} \subset \mathcal{O}(U)$ be a formal cl.a.s. where $U = U_1\times U_2 \subset \CC^{n_1+n_2}$ with $U_1$ some open set and $U_2$ a complex conic open neighborhood of $\RR^{n_2}\sem$ and each $a_j$ is positively homogeneous of degree $-j$. There is $\eps>0$ and a finite realization $a \in \mathcal{O}(U_1\times \Theta_{\eps}^\circ)$ of $\ud{a}$, and $a$ is a symbol of order $0$ in the usual smooth Kohn-Nirenberg sense.
\end{lemma}
Note that the assumptions on \cite[Lem.~1.21]{zbMATH08037414} are stronger than those in \cref{lem:holsymb}: their notion of formal analytic symbol assumes that each $a_j$ is holomorphic in the fiber at the origin. Nevertheless, the proof of \cite[Lem.~1.21]{zbMATH08037414} applies in our setting provided one replaces their $\chi_0$ by some smooth function vanishing near the origin, and identically $1$ away from the origin. One could also deduce \cref{lem:holsymb} as a succession of the following results: \cite[Lem.~17.2.27]{MR4436039}, \cite[Prop.~17.2.1]{MR4436039} and \cite[Lem.~1.31]{zbMATH08037414}.

\begin{lemma}\label{lem:recoversymbol}
	Let $m\in\RR$ and $p(x,\xi) \sim \sum_{k\geq 0} p_k(x,\xi)$ be a formal classical analytic symbol in $\Omega\times \Gamma$ according to \cref{def:sym}, where each $p_k$ is positively homogeneous of degree $m-k$ in $\xi$, and let $p\sim \ud{p}$ be a cl.a.s. 

	For every compact subset $K\Subset \Omega\times \Gamma$ and every sufficiently large $C_0>0$ there is $C>0$ so that for every $N\in\mathbb{N}, N\geq 1$ and $r\geq C_0N$ we have
	\begin{equation}\label{eq:RNerr}
		\sup_K\abs{p(x,r\xi)-\sum_{k=0}^{N-1} r^{m-k}p_k(x,\xi)}\leq C^{N+1}N!r^{m-N}\,.
	\end{equation}
\end{lemma}
\begin{proof}
	After dividing the desired estimate by $r^m$, we may assume that $m=0$. After replacing $r$ by $r\abs{\xi}$ and $\xi$ by $\xi\abs{\xi}^{-1}$, by modifying $C,C_0>0$ we may assume that $K \subset \{\abs{\xi}=1\}$.

	Since $N \leq r/C_0$, if $C_0$ is sufficiently large,
	\[
		p(x,r\xi)-\sum_{k=0}^{N-1} r^{-k}p_k(x,\xi) = p(x,r\xi) - \sum_{k\leq r/C_0} r^{-k} p_k(x,\xi) + \sum_{N \leq k\leq r/C_0} r^{-k}p_k(x,\xi)
	\]
	where the difference of the first two terms is bounded above in absolute value by $Ce^{-\eps r}$ for some $C,\eps>0$. By standard estimates (\cite[Eq.~(1.2.16)]{doi:10.1142/1550}) we have
	\[
		e^{-\eps r} \leq \eps^{-N} N^{N} r^{-N} \leq (e\eps^{-1})^{N} r^{-N} N!\,.
	\]
	For the second term we estimate
	\begin{align*}
		\abs{\sum_{N \leq k\leq r/C_0} r^{-k}p_k} &\leq \sum_{N \leq k\leq r/C_0} r^{-k}C^kk! = \sum_{0\leq k\leq r/C_0-N} r^{-k-N}C^{k+N}(k+N)!\\
		&= r^{-N} (2C)^N N! \sum_{0\leq k\leq r/C_0-N} (2Ck/r)^k \leq r^{-N}(2C)^N N! \sum_{0\leq k\leq r/C_0-N} e^{-k}\,,
	\end{align*}
	where we used the estimates $(k+N)! \leq 2^{k+N} k! N! \leq 2^{k+N}N! k^k$. The last inequality is provided $C_0>2Ce$ is large enough. Then we may bound
	\[
		\sum_{0\leq k\leq r/C_0-N} e^{-k} \leq \sum_{k=0}^\infty e^{-k} \leq \frac{1}{1-e^{-1}}\,,
	\]
	so that we indeed find \cref{eq:RNerr}.
\end{proof}

\subsection{Combinatorial estimates}\label{sec:comb}

The estimates of this section are taken from \cite[\S~1.2]{doi:10.1142/1550}. 

We have the following: for any multi-index $\alpha \in \mathbb{N}^n_0$%
\begin{equation}\label{eq:factbounds}
	\alpha! \leq \abs{\alpha}! \leq \abs{\alpha}^{\abs{\alpha}} \leq e^{\abs{\alpha}}\abs{\alpha}! \leq (en)^{\abs{\alpha}}\alpha!\,.%
\end{equation}
Stirling's approximations will give much more precise statements, but this suffices for our needs.

Furthermore, for any $j,k\in\mathbb{N}$, 
\begin{equation}\label{eq:sumfac}
	(j+k)! \leq 2^{j+k} j!k!\,.%
\end{equation}

By \cite[Thm.~2.2]{zbMATH08031921},
\begin{equation}\label{eq:countsum}
	\sum_{\substack{i_1,\dots,i_k \geq 0,\\ i_1+\dots+i_k=j}} 1 = {k-1+j\choose k-1} \leq 2^{k-1+j}\,,
\end{equation}
where the inequality came from applying \cref{eq:sumfac}.

\begin{lemma}\label{lem:faabound}
	Let $x \geq 0, k\in \mathbb{N}, k\geq 1$ and $\Pi_k$ be the set partitions of $\{1,\dots,k\}$. We have
	\[
		\sum_{\pi \in \Pi_k} x^{\abs{\pi}}\prod_{B\in\pi}\abs{B}! \leq 4^{k}\sum_{j=0}^k x^{k-j}j!\,.
	\]
\end{lemma}
\begin{proof}
	Define $f(y) = e^{(y-1)x}$ and $g(y) = \frac{1}{1-y} = \sum_{k\geq 0}y^k$ for $\abs{y}<1$. We have $\del^n_y g(0)=n!$ and $\del^n_y f(1)=x^n$ for all $n\geq 0$. By Fa\`a di Bruno's formula (first the combinatorial version, then the (equivalent) one using Bell polynomials $B_{k,j}$),
	\[
		\del^k (f\circ g)(0)= \sum_{\pi \in \Pi_k} x^{\abs{\pi}}\prod_{B\in\pi}\abs{B}! = \sum_{j=1}^k x^j B_{k,j}(1!,2!,\dots,(k-j+1)!)\,.
	\]
	According to \cite[\S~3.3~Thm.~B]{zbMATH03443655}, $B_{k,j}(1!,2!,\dots,(k-j+1)!) = {k-1 \choose j-1}\frac{k!}{j!}$, which are called the Lah numbers. 

	Changing variables $m = k-j$, and using 
	\[
		{k-1 \choose j-1}\frac{k!}{j!} = {k \choose m}{k-1\choose m}m! \leq 2^k 2^{k-1} m! \leq 4^k m!\,, 
	\]
	we have
	\[
		\sum_{\pi \in \Pi_k} x^{\abs{\pi}}\prod_{B\in\pi}\abs{B}! = \sum_{j=0}^{k-1} x^{k-j} {k \choose j}{k-1\choose j}j! \leq 4^k \sum_{j=0}^{k-1} x^{k-j}j!\,,
	\]
	which completes the proof.
\end{proof}

\subsection{Dirichlet eigenvalues and Runge approximation}\label{sec:EVrunge}

In this section we consider manifolds $(M,g)$ of dimension $n\geq 2$, and say that a formally self-adjoint differential operator $L$ of order $m$ is strongly elliptic if its principal symbol $l$ satisfies $l(x,\xi) \geq c \abs{\xi}^m$ for all $\xi\neq 0$ and some $c>0$. We first extend the strict domain monotonicity of Dirichlet eigenvalues (\cite{zbMATH03230686}, \cite{zbMATH00048198}) to a class of elliptic differential operators of arbitrary order. 
\begin{theorem}\label{thm:DEV}
	Let $(M,g)$ be a compact Riemannian manifold with smooth boundary. Let $L$ be a strongly elliptic formally self-adjoint differential operator with smooth coefficients on $M$ of even order $m$. For each non-empty open $U \subset M$ with $C^{\infty}$ boundary, the Dirichlet realization of $L$ in $L^2(U)$, i.e. the restriction of $L$ to the domain $H^m(U) \cap H^{m/2}_0(U)$ has real discrete spectrum bounded below and unbounded above: one may enumerate the eigenvalues $\lambda_1(U) \leq \lambda_2(U) \leq \dots \subset \RR$ counted according to multiplicity.

    Furthermore, for any $U_1, U_2 \subset M$ open sets with $C^{\infty}$ boundary,
    \begin{enumerate}
        \item if $U_1 \subset U_2$, then $\lambda_k(U_1) \geq \lambda_k(U_2)$ for all $k\in \mathbb{N}$, and $\lambda_k(U_1) = \lambda_k(U_2)$ for all $k$ implies $U_1 = U_2$.
        \item For any sequence of open sets $W_j \subset M, j\in\mathbb{N}$ each having $C^{\infty}$ boundary so that $W_j \subset W_{j+1} \forall j$ and $W \coloneqq \bigcup_{j\geq 1}W_j$ has $C^{\infty}$ boundary, we have $\lambda_k(W_j) \to \lambda_k(W)$ as $j\to\infty$ for all $k\in \mathbb{N}$.
    \end{enumerate}
    Assume in addition that $L$ satisfies the following unique continuation property: for any connected open $G\subset M$, $\lambda\in\RR$, and any $u \in H^m(G)\cap H^{m/2}_0(G)$ satisfying $L u = \lambda u$ in $G$ and vanishing on some open subset of $G$, we have $u\equiv 0$ on $G$. This additional assumption implies that
    \begin{enumerate}\setcounter{enumi}{2}
        \item If $U_1 \subset U_2$ and $U_2\setminus U_1$ contains an open set and $U_2$ is connected, then $\lambda_k(U_1) > \lambda_k(U_2)$ for all $k\in \mathbb{N}$.
    \end{enumerate}
\end{theorem}
\begin{proof}
    The basic properties of the spectrum of the Dirichlet realization of $L$ follow from \cite[Thm.~1.7.2]{zbMATH00857699}, see also \cite[Chp.~2~\S~9]{MR350177}. We shall invoke \cite{zbMATH01346906} to prove the statements 1. and 2., for which we shall provide some set-up. The G\aa{}rding inequality, see \cite[Thm.~1.7.2]{zbMATH00857699}, implies there is some constant $b>0$ so that the form (defined initially on $C_c^\infty(M^{\mathrm{int}})$)
    \[
        A(u,v) \coloneqq \langle Lu, v\rangle_{L^2(M)} + b \langle u,v\rangle_{L^2(M)}
    \]
    extends continuously to a positive closed form on $H_0^{m/2}(M)$, and defines an equivalent norm to $\norm{\cdot}_{H^{m/2}}$. 

    For every smooth open $U\subset M$, denote by $D_U$ the closure in $H_0^{m/2}(M)$ of the zero extension to $M$ of functions in $C_c^\infty(U)$, so that $D_U$ can be identified with $H^{m/2}_0(U)$. The operator corresponding to the restriction of the form $A$ to $D_U$ is $L_U+b$, where $L_U$ is the Dirichlet realization of $L$ on $H^m(U) \cap H^{m/2}_0(U)$. 
    
    By the Rellich-Kondrachov embedding theorem \cite[Thm.~2.34]{MR1636569}, the embedding of $D_U$ into $L^2(U)$ is compact, and if $U_1\subset U_2$, then $D_{U_1} \subset D_{U_2}$ for any smooth open $U_1, U_2 \subset M$, and if $D_{U_1}=D_{U_2}$, then we must have $U_1 = U_2$. If $(W_j)_{j\in \mathbb{N}}, W$ are according to 2. in the statement, the union $\bigcup_{j=1}^\infty D_{W_j}$ is dense in $D_W$: every function $\in C_c^\infty(W)$ is supported in some $W_j$. Thus, the statements 1. and 2. are consequences of \cite{zbMATH01346906} after subtracting the common shift $b$ from the eigenvalues. 
    
    The proof of statement 3. under the additional unique continuation assumption follows the proof of \cite[Thm.~2.3]{zbMATH03393073} almost verbatim, see also the proof of \cite[Prop.~C.3]{busch2025generalizedboundaryrigidityminimal}, and we shall not reiterate it.
\end{proof}
In the situation of \cref{thm:DEV} (without the additional unique continuation assumption on $L$), for any $U\subset M$ with $C^{\infty}$ boundary, we call the set $(\lambda_k(U))_{k\in \mathbb{N}}$ the \emph{Dirichlet eigenvalues of $L$ restricted to $U$}. We say that $L$ \emph{has Dirichlet eigenvalue $0$ restricted to $U$} if there is some $k\in \mathbb{N}$ so that $\lambda_k(U) = 0$.

Repeating the proof of \cite[Prop.~4.3]{busch2025generalizedboundaryrigidityminimal}, which merely relies on the strict monotonicity of eigenvalues (3.) and the `continuity' of the eigenvalues with respect to the domain (2.), we have
\begin{corollary}\label{cor:perturbDEV}
    In the setting of \cref{thm:DEV}, assuming that $L$ satisfies the unique continuation property therein, if $M' \subset M$ is some connected open set with $C^{\infty}$ boundary so that $L$ has Dirichlet eigenvalue $0$ restricted to $M'$, for any $x \in \overline{M'}$ and open neighborhood $W$ of $x$ in $\overline{M'}$, there is some open neighborhood $U\subset W$ of $x$ so that $M'\setminus \bar U$ is connected, has $C^{\infty}$ boundary and $L$ does not have Dirichlet eigenvalue $0$ restricted to $M'\setminus \bar U$. 
\end{corollary}
\begin{proof}
    Assume that $\lambda_k(M') = 0$ and $\lambda_i(M') < 0$ for all $i < k$, i.e., $\lambda_k(M')$ is the first zero Dirichlet eigenvalue. Let $U_j \subset W, j\in\mathbb{N}$ be a sequence of bounded open neighborhoods of $x$ with smooth boundary so that $\bigcap_{j \geq 1} U_j = \{x\}$, $U_{j+1} \subset U_{j}$, $U_{j}\setminus \bar U_{j+1}$ contains and open set and $M'\setminus \bar U_j$ is connected for every $j$. Define $\Omega_j = M'\setminus \bar U_j$, and note that strict monotonicity of eigenvalues gives $\lambda_k(\Omega_j) > 0$ for all $j$, which completes the proof if $k=1$: choose $U = U_1$, say.

    If, however, $k>1$, because $\lambda_k(\Omega_j)>0$ for any $j$, if $i\in \mathbb{N}$ so that $\lambda_i(\Omega_j)=0$, we must have $i < k$. Thus consider the following list of $k-1$ different sequences 
    \[
        (\lambda_1(\Omega_j))_{j\in\mathbb{N}}, \dots, (\lambda_{k-1}(\Omega_j))_{j\in\mathbb{N}}\,.
    \]
    Strict monotonicity of eigenvalues implies that each of these sequences is strictly decreasing, so that there is some $j^\circ$ sufficiently large so that 
    \[
        \lambda_1(\Omega_{j^\circ})\neq 0, \dots, \lambda_{k-1}(\Omega_{j^\circ}) \neq 0\,,
    \]
    for if, say $\lambda_i(\Omega_j)=0$ for some $i,j$, we must have $\lambda_i(\Omega_{j+1})>0$. Since $\lambda_{l}(\Omega_{j^\circ}) \geq \lambda_k(\Omega_{j^\circ})>0$ for all $l\geq k$, the proof is complete by choosing $U= U_{j^\circ}$.
\end{proof}
In words: if $M'$ happens to be a domain on which $L$ has Dirichlet eigenvalue $0$, then any sufficiently small modification of $M'$ will remedy this property.

\begin{remark}
   The unique continuation property imposed on $L$ in \cref{thm:DEV} is automatically true if $m=2$ by the usual unique continuation statements (\cite{zbMATH03174689}, \cite[\S~17.2]{zbMATH05129478}). On the other hand, this is also true if $L$ has analytic coefficients in $M$ without any restriction on its order (\cite[Thm.~8.6.1]{hoermander1}). Note also that it suffices to require that the sets $U, U_1, U_2, W_j, W$ in \cref{thm:DEV} have $C^{1,1}$ boundary if $m=2$. 
\end{remark}

We also have
\begin{proposition}\label{prop:GK}
	Let $(M,g)$ be a compact Riemannian manifold with smooth boundary. Let $L$ be a strongly elliptic formally self-adjoint differential operator with smooth coefficients on $M$ of even order $m$. If $L$ does not have Dirichlet eigenvalue $0$ (restricted to $M$), the Green's and Poisson operators $G$ and $K$ are well-defined and for every $s\geq m$
	\[
		G \colon H^{s-m}(M) \to H^{s}(M)\cap H^{m/2}_0(M)\,,\qquad K \colon \prod_{j=0}^{m/2-1} H^{s-j-\frac{1}{2}}(\partial M) \to H^s(M)
	\]
	are continuous. Here, $G, K$ are uniquely determined by the property that 
    \[
        LG u = u\,, \quad (\del_\nu^j Gu\vert_{\del M})_{j=0}^{m/2-1} = 0\,, \qquad \forall u \in H^{s-m}(M)\,,
    \]
    and
    \[
            LK g = 0\,, \quad (\del_\nu^j Kg\vert_{\del M})_{j=0}^{m/2-1} = g\,, \qquad \forall g \in \prod_{j=0}^{m/2-1} H^{s-j-\frac{1}{2}}(\partial M)\,.
    \]
\end{proposition}
\begin{proof}
	According to \cite[Thm.~20.1.2]{zbMATH05129478} or \cite[Prop.~5.11.16,~Prop.~5.11.10]{zbMATH05771951} the operator $H^s(M) \ni u \mapsto \{Lu\}\times\prod_{j=0}^{m/2-1} \{\del_\nu^j u\vert_{\partial M}\} \in H^{s-m}(M)\times \prod_{j=0}^{m/2-1} H^{s-j-\frac{1}{2}}(\partial M)$ is Fredholm. Furthermore, \cite[Exc.~5.11.3]{zbMATH05771951} implies that the map $H^s(M)\cap H^{m/2}_0(M) \ni u\mapsto Lu \in H^{s-m}(M)$ is Fredholm of index $0$, and since we assumed its injectivity, it is bijective. In particular, $L$ satisfies the assumptions of \cite[Thm.~11.14]{zbMATH05366499}, which completes the proof. We could have also just invoked \cite[Chp.~2~Eq.~(8.14)]{MR350177}.
\end{proof}

We shall extend the Runge approximation result of \cite[Prop.~A.2]{zbMATH07203082} to the present situation.
\begin{theorem}\label{thm:runge}
	Let $(M,g)$ be a connected compact Riemannian manifold with smooth boundary, let non-empty $U \subset M$ be open, so that $U$ has smooth boundary, $M\setminus \bar U$ is connected, and $\del M \setminus \bar U$ has non-empty interior. Let $\Xi \subset \bar U \cap \del M$ be closed in $\del M$ (we allow $\Xi$ to be empty) and let $\Gamma\subset \del M$ be open with $\Gamma \supset \Xi$ and $\Gamma \cap (M\setminus \bar U) \neq \emptyset$ so that $\Gamma \cup (\del M \cap \bar U)$ is open in $\del M$. Let $\Gamma_0 \subset \Gamma\setminus \bar U$ be a non-empty open subset of $\del M$. Let $L$ be a strongly elliptic formally self-adjoint differential operator with smooth coefficients on $M$ of even order $m$ so that
	\begin{enumerate}
	\item $L$ does not have Dirichlet eigenvalue $0$ (restricted to $M$).
	\item We assume the following unique continuation property: any $u \in C^\infty((M^{\mathrm{int}}\setminus \bar U)\cup \Gamma_0)$ so that $L^\ast u = 0$ in $M^{\mathrm{int}}\setminus \bar U$ and $\del_\nu^j u\vert_{\Gamma_0} = 0 \forall j\geq 0$ vanishes in $M^{\mathrm{int}}\setminus \bar U$.
	\end{enumerate}
	For any $s \geq m$, the set
	\[
		\mathcal{R} = \left\{u\vert_U \colon u\in H^s(M), L u = 0\ \text{in}\ M, (\del_\nu^j u\vert_{\del M})_{j=0}^{m/2-1} \in C_c^\infty(\Gamma)^{m/2}\right\}
	\]
	is dense in 
    \[
        \mathcal{S} = \left\{u \in H^s(U) \colon Lu = 0 \ \text{in}\ U, (\del_\nu^j u\vert_{(\del M \cap \bar U)\setminus \Xi})_{j=0}^{m/2-1} = 0\right\}
    \]
    in the $H^s(U)$ norm.
\end{theorem}
The two extremal cases above arise when $\Xi = \emptyset$ or $\Xi = \bar U \cap \del M$. 
\begin{remark}
    As before, if $L$ is of order $m=2$ or has analytic coefficients in $M$, it automatically satisfies assumption 2. above for any open $\Gamma_0 \subset \del M\setminus \bar U$. 
\end{remark}
\begin{proof}[Proof of \cref{thm:runge}]
	As in the proof of \cite[Prop.~A.2]{zbMATH07203082} we take arbitrary $F \in (H^s(U))^\ast$ satisfying $\langle F,Kf\vert_U \rangle = 0$ for all $f\in C_c^\infty(\Gamma)^{m/2}$. By the Hahn-Banach theorem, proving that $\langle F, u \rangle = 0$ for all $u\in\mathcal{S}$ completes the proof. 

	We first extend $F$ to $\hat F \in (H^s(M))^\ast$ as $\langle \hat F,u\rangle \coloneqq \langle F, u\vert_U\rangle$ for all $u\in H^s(M)$. Furthermore, invoking \cref{prop:GK}, we define
	\[
		\varphi \in (H^{s-m}(M))^\ast\,,\quad \langle \varphi, h\rangle \coloneqq \langle \hat F, G h\rangle \quad \forall h\in H^{s-m}(M)\,,
	\]
	and
	\[
		\rho \in \prod_{j=0}^{m/2-1} (H^{s-j-1/2}(\partial M))^\ast = \prod_{j=0}^{m/2-1} H^{-(s-j-1/2)}(\partial M)\,, \quad \langle \rho, g \rangle \coloneqq \langle \hat F, Kg\rangle\quad\forall g\in \prod_{j=0}^{m/2-1} H^{s-j-1/2}(\partial M)\,.
	\]
	Defining the trace map $\gamma\colon H^s(M) \ni u \mapsto (\del_\nu^0 u\vert_{\partial M},\dots, \del_\nu^{m/2-1} u\vert_{\partial M})\in \prod_{j=0}^{m/2-1} H^{s-j-1/2}(\partial M)$, any $u \in H^s(M)$ can be written $u = u - K(\gamma u) + K(\gamma u) = u_0 + K(\gamma u)$ with $GL u_0 = u_0$. In particular,
	\begin{equation}\label{eq:imp}
		\langle \hat F, u\rangle = \langle \hat F, u_0 \rangle + \langle \hat F, K(\gamma u)\rangle = \langle \varphi, Lu_0\rangle + \langle \rho, \gamma u\rangle =  \langle \varphi, Lu\rangle + \langle \rho, \gamma u\rangle\,,\quad \forall u \in H^s(M)\,,
	\end{equation}
	where we used the definitions of $\varphi, \rho$ and that $L K = 0$. Now by assumption, for all $f\in C_c^\infty(\Gamma)^{m/2}$, 
	\begin{equation}\label{eq:imp2}
		0 = \langle F, Kf\vert_U\rangle = \langle \hat F, Kf\rangle = \langle \rho, f\rangle\,,\quad \text{so that}\quad \rho\vert_{\Gamma} = 0\,,
	\end{equation}
	(or more precisely, $\rho\vert_{\Gamma^{m/2}} = 0$).

	Introduce some compact manifold $N$ without boundary so that $N\supset M$ and extend $\varphi$ as $\td{\varphi}$ to $(H^{s-m}(N))^\ast = H^{m-s}(N)$ (\cite[\S~4~Prop.~3.2]{zbMATH05771951}) defined by $\langle \td{\varphi}, \psi \rangle_{N} = \langle \varphi, \psi\vert_M\rangle$ for all $\psi \in H^{s-m}(N)$. We also extend $L$ elliptically with smooth coefficients to $N$ as $\td{L}$. We will show that $\supp \varphi \subset \bar U \cup (\del M\setminus \Gamma)$: for all $\psi \in C_c^\infty(N)$ so that $\supp \psi \cap \bar U = \emptyset$ and $\supp \psi\vert_{\del M} \subset \Gamma$, from \cref{eq:imp,eq:imp2} we conclude
	\[
		\langle \td{L}^\ast \td{\varphi}, \psi \rangle_{N} = \langle \td{\varphi}, \td{L}\psi\rangle_{N} = \langle \varphi, L\psi\rangle = \langle \hat F, \psi\rangle - \langle \rho, \gamma \psi\rangle = 0\,,
	\]
	where we also used that $\supp \hat F \subset \bar U$. Therefore, elliptic regularity for $\td{L}$ (\cite[Cor.~8.3.2]{hoermander1}) implies that $\td{\varphi} \in C^\infty(N \setminus (\bar U \cup (\del M\setminus \Gamma)))$, which implies that $\td{\varphi}$ is smooth in an open neighborhood of $\Gamma_0 \subset \del M$ in $N$ and we know $\td{\varphi}$ vanishes in $N\setminus M$. We conclude that $\varphi \in C^\infty((M^{\mathrm{int}}\setminus \bar U)\cup \Gamma_0)$ satisfies $L^\ast \varphi = 0$ in $M^{\mathrm{int}}\setminus \bar U$ and $\del_\nu^j \varphi\vert_{\Gamma_0} = 0\forall j\geq 0$ so that the unique continuation assumption implies that $\varphi$ vanishes in $M^{\mathrm{int}}\setminus \bar U$. 

	Denote by $\Gamma' \supset \bar U \cap \del M$ some non-empty open subset of $\del M$ so that $\Gamma' \subset \Gamma \cup ((\del M \cap \bar U)\setminus \Xi)$. Given arbitrary $u\in\mathcal{S}$, there is an open $W\supset \bar U$ (with respect to the topology on $M$) with $W\cap \del M \subset \Gamma'$ and an extension $\hat u \in H^s(M)$ of $u$ so that $\supp \hat u \subset W$. %
	In particular, $\supp \gamma \hat u\vert_{\del M} \subset \Gamma'$. By assumption on $u$ and the fact that $\hat u$ extends $u$, we know that $\gamma \hat u \vert_{(\bar U \cap \del M)\setminus \Xi} = 0$, which combined with \cref{eq:imp2} (and the fact that $\Gamma' \subset \Gamma \cup ((\del M \cap \bar U)\setminus \Xi)$) implies $\langle \rho, \gamma \hat u\rangle = 0$.

	Thus, \cref{eq:imp} implies
	\[
		\langle F, u \rangle_U = \langle \hat F, \hat u\rangle = \langle \varphi, L\hat u\rangle = \langle \varphi\vert_W, L\hat u\vert_W\rangle_W\,.
	\]
	Recall that $\supp \varphi \subset \bar U \cup (\del M \setminus \Gamma)$ and because $W$ was chosen so that
	\[
		W\cap \del M \subset \Gamma' \subset \Gamma \cup ((\del M \cap \bar U)\setminus \Xi) = \Gamma \cup (\del M \cap \bar U)\,,
	\]
	we have $\supp \varphi\vert_W \subset \bar U$, so that we may interpret $\varphi\vert_W$ as an element of $H^{m-s}_{\bar U}(N) = \{v \in H^{m-s}(N)\colon \supp v \subset \bar U\}$, so that according to \cite[Thm.~3.29]{zbMATH01446717} there are some $\varphi_j \in C_c^\infty(U)$ so that $\varphi_j \to \varphi\vert_W$ as $j\to \infty$ in $H^{m-s}(N)$. We thus have
	\[
		\langle F, u \rangle_U = \langle \varphi\vert_W, L\hat u\vert_W\rangle_W = \lim_{j\to\infty}\langle \varphi_j, L \hat u\rangle_N = \lim_{j\to\infty}\langle \varphi_j, L u\rangle_U = 0\,,
	\]
	completing the proof.
\end{proof}

\printbibliography[title=Bibliography] 

\end{document}